\documentclass[a4paper,12pt]{amsart}

\usepackage{amsthm, amsmath, amssymb, bbm}
\usepackage[margin=4cm]{geometry}
\usepackage[T1]{fontenc}
\usepackage{mathrsfs}
\usepackage{tikz}
\usetikzlibrary{calc}

\usepackage{enumerate}
\usepackage{graphics}
\usepackage[all]{xy}

\newdir{ >}{{}*!/-10pt/@{>}}

\usepackage{color}

\usepackage[colorlinks=true, pdfstartview=FitV, linkcolor=blue, %
citecolor=blue, urlcolor=blue]{hyperref}

\usepackage{marginnote}
\usepackage{xspace}
\usepackage{ulem}
\newcommand{\nc}{\newcommand}
\newenvironment{rouge}
{\relax\color{red}}
{\hspace*{.3ex}\relax}
\newcommand{\ber}{\begin{rouge}{}\marginnote{\mbox{$\bullet$}}{}}
\newcommand{\er}{\end{rouge}}
\newcommand{\bera}{\begin{rouge}{}\marginnote{\fbox{\scshape\lowercase{A}}}{}}
\newcommand{\berm}{\begin{rouge}{}\marginnote{\fbox{\scshape\lowercase{M}}}{}}

\newcommand{\erm}{\end{rouge}}

\newenvironment{bleu}
{\relax\color{blue}}
{\hspace*{.3ex}\relax}
\newcommand{\beb}{\begin{bleu}}
\newcommand{\bebm}{\begin{bleu}{}\marginnote{\fbox{\scshape\lowercase{M}}}{}}
\newcommand{\beba}{\begin{bleu}{}\marginnote{\fbox{\scshape\lowercase{A}}}{}}
\newcommand{\eb}{\end{bleu}}

\usepackage{units}

\usepackage{accents}

\usepackage{tensor}

\renewcommand{\leq}{\leqslant}
\renewcommand{\geq}{\geqslant}

\numberwithin{equation}{subsection}

\theoremstyle{plain}
\newtheorem{theorem}{Theorem}[subsection]
\newtheorem*{theorem*}{Theorem}
\newtheorem{corollary}[theorem]{Corollary}
\newtheorem{proposition}[theorem]{Proposition}
\newtheorem{lemma}[theorem]{Lemma}

\theoremstyle{definition}
\newtheorem{definition}[theorem]{Definition}
\newtheorem{example}[theorem]{Example}
\newtheorem*{example*}{Example}
\newtheorem{notation}[theorem]{Notation}
\newtheorem{remark}[theorem]{Remark}

\nc{\Lemma}{\begin{lemma}}
\nc{\enlemma}{\end{lemma}}
\nc{\Prop}{\begin{proposition}}
\nc{\enprop}{\end{proposition}}
\nc{\Def}{\begin{definition}}
\nc{\edf}{\end{definition}}

\renewcommand{\emptyset}{\varnothing}

\newcommand{\point}{{\{\mathrm{pt}\}}}

\nc{\scup}{\mathop{\scalebox{.8}{$\displaystyle\bigcup$}}\mspace{1mu}\limits}
\nc{\scap}{\mathop{\scalebox{.8}{$\displaystyle\bigcap$}}\limits}
\nc{\ssqcup}{\mathop{\scalebox{.8}{$\displaystyle\bigsqcup$}}\limits}
\newcommand{\DUnion}{\bigsqcup\limits}

\newcommand{\union}{\cup}
\newcommand{\Union}{\bigcup\limits}

\newcommand{\Inter}{\bigcap\limits}

\newcommand{\C}{\mathbb{C}}

\newcommand{\R}{\mathbb{R}}

\newcommand{\Z}{\mathbb{Z}}

\newcommand{\op}{\mathrm{op}}

\DeclareMathOperator{\id}{id}

\newcommand{\derived}[1]{\mathrm{#1}}

\newcommand{\derd}{\derived{D}}
\newcommand{\dere}{\derived{E}}
\newcommand{\derr}{\derived{R}}
\newcommand{\derl}{\derived{L}}
\nc{\derb}{\derd^{\mathrm{b}}}
\newcommand{\BDC}{\derd^{\mathrm{b}}}

\nc{\soplus}{\scalebox{.65}{\raisebox{.2ex}{$\displaystyle\bigoplus$}}}
\newcommand{\indsum}{\mathop{\raisebox{.2ex}{\rm``$\soplus$''}}}
\newcommand{\DSum}{\mathop{\bigoplus}}
\newcommand{\dsum}[1][]{\mathbin{\oplus_{#1}}}

\newcommand{\ilim}[1][]{\mathop{\varinjlim}\limits_{#1}}

\DeclareMathOperator{\coker}{coker}
\DeclareMathOperator{\im}{im}
\DeclareMathOperator{\coim}{coim}

\newcommand{\comp}{\mathbin{\circ}}

\renewcommand{\to}[1][]{\xrightarrow{#1}}
\newcommand{\from}[1][]{\xleftarrow{#1}}
\newcommand{\isofrom}[1][]{\xleftarrow[#1]%
{\raisebox{-.4ex}[0ex][-.4ex]{$\mspace{2mu}\sim\mspace{2mu}$}}}
\newcommand{\isoto}[1][]{\xrightarrow[#1]{%
{\raisebox{-.6ex}[0ex][0ex]{$\mspace{1mu}\sim\mspace{2mu}$}}}}

\newcommand{\Endo}[1][]{\mathrm{End}_{\raise1.5ex\hbox to.1em{}#1}}

\newcommand{\Hom}[1][]{\mathrm{Hom}_{\raise1.5ex\hbox to.1em{}#1}}

\newcommand{\RHom}[1][]{\derr\mathrm{Hom}_{\raise1.5ex\hbox to.1em{}#1}}

\newcommand{\Ext}[2][]{\mathrm{Ext}_{\raise1.5ex\hbox to.1em{}#1}^{#2}}

\newcommand{\Mod}{\mathrm{Mod}}

\newcommand{\Tens}[1][]{\mathbin{\otimes_{\raise1.5ex\hbox to-.1em{}#1}}}

\newcommand{\LTens}[1][]{\mathbin{\otimes_{\raise1.5ex\hbox to-.1em{}#1}^{\derl}}}

\newcommand{\Tor}[2][]{\mathrm{Tor}^{\raise1.5ex\hbox to.1em{}#1}_{#2}}

\newcommand{\sheaffont}[1]{\mathcal{#1}}

\def\shc{\sheaffont{C}}
\def\shd{\sheaffont{D}}

\def\shi{\sheaffont{I}}

\def\shl{\sheaffont{L}}
\def\shm{\sheaffont{M}}
\def\shn{\sheaffont{N}}
\def\sho{\sheaffont{O}}
\def\shp{\sheaffont{P}}

\def\shr{\sheaffont{R}}
\def\shs{\sheaffont{S}}

\newcommand{\sect}{\varGamma}
\newcommand{\rsect}{\derr\varGamma}

\newcommand{\shendo}[1][]{{\sheaffont{E}nd}_{\raise1.5ex\hbox to.1em{}#1}}

\renewcommand{\hom}[1][]{{\sheaffont{H}om}_{\raise1.5ex\hbox to.1em{}#1}}
\newcommand{\aut}[1][]{{\sheaffont{A}ut}_{\raise1.5ex\hbox to.1em{}#1}}
\newcommand{\inn}[1][]{{\sheaffont{I}nn}_{\raise1.5ex\hbox to.1em{}#1}}

\newcommand{\rhom}[1][]{{\derr\sheaffont{H}om}_{\raise1.5ex\hbox to.1em{}#1}}

\newcommand{\ext}[2][]{{\sheaffont{E}xt}_{\raise1.5ex\hbox to.1em{}#1}^{#2}}

\newcommand{\thom}[1][]{{\sheaffont{T}hom}_{\raise1.5ex\hbox to.1em{}#1}}

\newcommand{\tens}[1][]{\mathbin{\otimes_{\raise1.5ex\hbox to-.1em{}#1}}}

\newcommand{\ltens}[1][]{\mathbin{\otimes_{\raise1.5ex\hbox to-.1em{}#1}^{\derl}}}

\newcommand{\tor}[2][]{{\sheaffont{T}or}^{\raise1.5ex\hbox to.1em{}#1}_{#2}}

\newcommand{\etens}[1][]{\mathbin{\boxtimes_{\raise1.5ex\hbox to-.1em{}#1}}}

\DeclareMathOperator{\supp}{supp}

\newcommand{\oim}[1]{#1_*}
\newcommand{\eim}[1]{#1_!}
\newcommand{\eeim}[1]{#1_{!!}}

\newcommand{\roim}[1]{\derr#1_*}
\newcommand{\roimv}[1]{\derr#1}
\newcommand{\reim}[1]{\derr#1_{\mspace{.5mu}!}\mspace{2mu}}
\newcommand{\reimv}[1]{\derr#1\mspace{2mu}}
\newcommand{\reeim}[1]{\derr#1_{\mspace{1mu}!!}\mspace{1mu}}

\newcommand{\opb}[1]{#1^{-1}}

\newcommand{\epb}[1]{#1^{\mspace{1.5mu}!}\mspace{2mu}}

\DeclareMathOperator{\ori}{or}

\newcommand{\tenstop}[1][]{\mathbin{\hat{\otimes}_{\raise1.5ex\hbox to-.1em{}#1}}}

\newcommand{\homtop}[1][]{\sheaffont{L}_{\raise1.5ex\hbox to.1em{}#1}}

\newcommand{\Homtop}[1][]{\mathrm{L}_{\raise1.5ex\hbox to.1em{}#1}}

\newcommand{\D}{\sheaffont{D}}

\renewcommand{\O}{\sheaffont{O}}

\newcommand{\Db}{\sheaffont{D}b}

\DeclareMathOperator{\chv}{char}

\newcommand{\detens}[1][]%
{\mathbin{\boxtimes_{\raise1.5ex\hbox to-.1em{}#1}^{\mspace{2mu}\mathsf{D}}}}

\newcommand{\doim}[1]{{\mathsf{D}#1}_*\mspace{1mu}}

\newcommand{\dopb}[1]{{\mathsf{D}#1}^{\mspace{1mu}*}}

\newcommand{\dtens}[1][]{\mathbin{\otimes_{\raise1.5ex\hbox to-.1em{}#1}^{\mathsf{D}}}}

\newcommand{\ddual}[1][X]{\mathbb{D}_{#1}}

\newcommand{\hol}{\mathrm{hol}}

\newcommand{\reghol}{{\mathrm{rh}}}

\newcommand{\Cfield}{\C}
\newcommand{\iCfield}{\ind\C}
\newcommand{\field}{\mathbf{k}}
\newcommand{\ind}{\mathrm{I}\mspace{2mu}}
\newcommand{\ifield}{\ind\field}

\newcommand{\Rc}{{\R\text-\mathrm{c}}}
\newcommand{\Cc}{{\C\text-\mathrm{c}}}

\newcommand{\cl}{\colon}

\newcommand{\ctens}{\mathbin{\mathop\otimes\limits^+}}

\newcommand{\cihom}{{\shi hom}^+}

\newcommand{\PR}{\mathsf{P}}

\newcommand{\sa}{\mathrm{sa}}

\newcommand{\tmp}{\mathsf{t}}
\newcommand{\dr}{\mathcal{DR}}

\newcommand{\sol}[1][X]{\mathcal Sol_{#1}}

\newcommand{\Dbt}{\Db^{\mspace{1mu}\tmp}}

\newcommand{\Ot}{\O^{\mspace{2mu}\tmp}}

\renewcommand{\Im}{\operatorname{Im}}

\newcommand{\ihom}[1][]{{\shi hom}_{\raise1.5ex\hbox to.1em{}#1}}
\newcommand{\rihom}[1][]{{\derr\mspace{2mu}\shi hom}_{\raise1.5ex\hbox to.1em{}#1}}
\newcommand{\ii}[1][]{{\sheaffont{I}h}_{\raise1.5ex\hbox to.1em{}#1}}

\newcommand{\indlim}[1][]{\mathop{\text{\rm``$\varinjlim$''}}\limits_{#1}}
\newcommand{\prolim}[1][]{\mathop{\text{\rm``$\varprojlim$''}}\limits_{#1}}

\renewcommand{\comp}[1][]{\mathbin{\mathop{\circ}\limits_{#1}}}

\newcommand{\dcomp}[1][]{\mathbin{\circ_{\raise1.5ex\hbox to-.1em{}#1}^{\mathsf{D}}}}

\newcommand{\enh}{\derived{E}}
\newcommand{\Tmp}{\derived{T}}
\newcommand{\OEn}{\O^\enh}
\newcommand{\DbT}{\Db^\Tmp}
\newcommand{\DbE}{\Db^\enh}
\newcommand{\OvE}{\Omega^\enh}
\newcommand{\drE}[1][X]{\mathcal{DR}^\enh_{#1}}
\newcommand{\solE}[1][X]{\mathcal{S}ol^{\mspace{1mu}\enh}_{#1}}

\newcommand{\fhom}{\mathcal{H}om^\enh}

\newcommand{\BEC}[2][\ifield]{\dere^{\mathrm{b}}(#1_{#2})}

\newcommand{\BECp}[2][\ifield]{\dere^{\mathrm{b}}_+(#1_{#2})}

\newcommand{\BECpm}[2][\ifield]{\dere^{\mathrm{b}}_\pm(#1_{#2})}

\newcommand{\Edual}{\dual^\enh}
\newcommand{\Eoim}[1]{{\enh#1}_*}
\newcommand{\Eoimv}[1]{{\enh#1}}
\newcommand{\Eeeim}[1]{{\enh#1}_{!!}}

\newcommand{\Eopb}[1]{{\enh#1}^{-1}}
\newcommand{\Eopbv}[1]{{\enh#1}}
\newcommand{\Eepb}[1]{{\enh\mspace{1mu}#1}^{\mspace{1.5mu}!}}
\newcommand{\Eepbv}[1]{{\enh\mspace{1mu}#1}}

\newcommand{\LE}{\operatorname{L^\enh}}
\newcommand{\RE}{\operatorname{R^\enh}}

\newcommand{\suban}{{\operatorname{suban}}}

\newcommand{\semicolon}{\nobreak \mskip2mu\mathpunct{}\nonscript\mkern-\thinmuskip{;}\mskip6mu plus1mu\relax}

\newcommand{\dual}{\mathrm{D}}

\newcommand{\sep}{\mspace{2mu}}

\newcommand{\defeq}{\mathbin{:=}}
\newcommand{\bl}{\bigl(}
\newcommand{\br}{\bigr)}
\newcommand{\To}[1][]{\xrightarrow[]{\mspace{10mu}{#1}\mspace{10mu}}}

\newenvironment{myarray}[1]{\relax\setlength{\arraycolsep}{1pt}

\begin{array}{#1}}{\end{array}\relax}

\newcommand{\ba}{\begin{myarray}}
\newcommand{\ea}{\end{myarray}}

\newcommand{\hs}{\hspace*}

\newcommand{\be}{\begin{enumerate}}
\newcommand{\ee}{\end{enumerate}}
\newcommand{\bnum}{\be[{\rm(i)}]}
\newcommand{\bna}{\be[{\rm(a)}]}
\nc{\bwr}{\mbox{\large{$\wr$}}}
\nc{\vphi}{\varphi}
\nc{\seteq}{\mathbin{:=}}
\nc{\noi}{\noindent}
\nc{\ro}{{\rm(}}
\nc{\rf}{{\rm)}\xspace}
\nc{\ms}{\mspace}
\nc{\sbcup}{\mathop{\scalebox{0.75}{$\displaystyle\bigcup$}}}
\nc{\ol}{\overline}
\nc{\scbul}{{\,\raise1pt\hbox{$\scriptscriptstyle\bullet$}\,}}
\nc{\set}[2]{\left\{#1\;\semicolon\; #2 \right\}}
\nc{\extp}{\mathop{\raisebox{.3ex}{\scalebox{0.8}{$\displaystyle\bigwedge$}}}\limits}

\newenvironment{myequation}
{\relax\setlength{\arraycolsep}{1pt}\begin{eqnarray}}
{\end{eqnarray}}
\newenvironment{myequationn}
{\relax\setlength{\arraycolsep}{1pt}\begin{eqnarray*}}
{\end{eqnarray*}}

\newenvironment{myalign}
{\relax\begin{align}}
{\end{align}}
\newenvironment{myalignn}
{\relax\begin{align*}}
{\relax\end{align*}}

\nc{\eq}{\begin{myequation}}
\nc{\eneq}{\end{myequation}}
\nc{\eqn}{\begin{myequationn}}
\nc{\eneqn}{\end{myequationn}}

\nc{\eqa}{\begin{myalign}}
\nc{\eneqa}{\end{myalign}}
\nc{\eqan}{\begin{myalignn}}
\nc{\eneqan}{\end{myalignn}}

\nc{\on}{\operatorname}
\nc{\Ind}{\on{Ind}}
\nc{\Proof}{\begin{proof}}
\nc{\QED}{\end{proof}}
\nc{\cor}{\field}
\nc{\tone}{\To[+1]}
\renewcommand{\ge}{\geq}
\renewcommand{\le}{\leq}

\newcommand{\CS}{\operatorname{CS}}
\newcommand{\LCS}{\operatorname{LCS}}

\newcommand{\fihom}{\shi hom^\enh}

\newcommand{\LEpm}{\operatorname{L}^\enh_\pm}

\newcommand{\REpm}{\operatorname{R}^\enh_\pm}
\renewcommand{\tor}{\mathrm{tor}}

\newcommand{\category}[1]{\mathcal{#1}}

\newcommand{\catc}{\category{C}}

\newcommand{\catn}{\category{N}}
\newcommand{\cats}{\category{S}}
\newcommand{\catt}{\category{T}}

\newcommand{\derdR}[2][]{\derd^{#1}(#2\times\bR)}

\newcommand{\bordered}[1]{{\mathsf{#1}}}
\newcommand{\bopen}[1]{{#1}}
\newcommand{\bclose}[1]{{\accentset{\vee}{#1}}}
\newcommand{\unbordered}[1]{{\accentset{\circ}{#1}}}

\newcommand{\bR}{{\R_\infty}}
\newcommand{\oR}{\R}
\newcommand{\cR}{{\overline\R}}

\newcommand{\bM}{\bordered{M}}
\nc{\unb}{\unbordered}
\nc{\eps}{\varepsilon}
\nc{\inb}{\inbordered}
\nc{\colim}{\varinjlim\limits}
\nc{\ssubset}{\subset\ms{-3mu}\subset}
\nc{\al}{\alpha}
\nc{\qtq}[1][and]{\quad\text{#1}\quad}
\nc{\olG}[1][f]{{\overset{\ms{4mu}\rule[-.05ex]{1.6ex}{.115ex}}{\Gamma}}_{%
\ms{-3mu}#1}}

\newcommand{\unbM}{\unbordered{\bM}}
\newcommand{\sunbM}{\accentset{\scalebox{.35}{$\circ$}}{\bM}}

\newcommand{\oM}{\bopen{M}}
\newcommand{\cM}{\bclose{M}}

\newcommand{\bN}{\bordered{N}}
\newcommand{\unbN}{\unbordered{\bN}}
\newcommand{\oN}{\bopen{N}}
\newcommand{\cN}{\bclose{N}}
\nc{\cf}{\bclose{f}}

\newcommand{\inbordered}[1]{{#1_\infty}}

\newcommand{\bZ}{\inbordered{Z}}
\newcommand{\bS}{\inbordered{S}}
\newcommand{\oZ}{Z}

\newcommand{\bU}{\inbordered{U}}

\newcommand{\quot}{\derived Q}

\newcommand{\Qfield}{\field^\quot}
\newcommand{\Efield}{\field^\enh}
\newcommand{\Qdual}{\dual^\quot}

\newcommand{\ECfield}{\C^\enh}

\newcommand{\Dp}[2][p]{\tensor*[^{#1}]{\derd}{^{#2}}}
\newcommand{\Dprc}[2][p]{\tensor*[^{#1}]{\derd}{_\Rc^{#2}}}
\newcommand{\Dmid}[1]{\Dprc[1/2]{#1}}

\newcommand{\tEp}[2][p]{\tensor*[_{#1}]{\dere}{^{#2}}}

\newcommand{\Erc}{\dere_\Rc}

\newcommand{\tEprc}[2][p]{\tensor*[_{#1}]{\dere}{_\Rc^{#2}}}
\newcommand{\dEprc}[2][p]{\tensor*[^\prime_{#1}]{\dere}{_\Rc^{#2}}}

\newcommand{\Eprc}[2][p]{\tensor*[^{#1}]{\dere}{_\Rc^{#2}}}
\newcommand{\Emid}[1]{\Eprc[1/2]{#1}}

\newcommand{\st}{\mathrm{st}}

\begin{document}
\title{Enhanced perversities}

\author[A.~D'Agnolo]{Andrea D'Agnolo}
\address{Dipartimento di Matematica\\
Universit{\`a} di Padova\\
via Trieste 63, 35121 Padova, Italy}
\email{dagnolo@math.unipd.it}

\author[M.~Kashiwara]{Masaki Kashiwara}
\address{Research Institute for Mathematical Sciences\\
Kyoto University\\
Kyoto 606-8502, Japan}
\email{masaki@kurims.kyoto-u.ac.jp}

\thanks{The second author
was supported by Grant-in-Aid for Scientific Research (B)
15H03608, Japan Society for the Promotion of Science.}


\date{September 12, 2015}

\begin{abstract}
On a complex manifold, the Riemann-Hilbert correspondence embeds the triangulated category of (not necessarily regular) holonomic $\D$-modules into that of $\R$-constructible enhanced ind-sheaves. The source category has a standard
t-structure.
Here, we provide the target category with a middle perversity t-structure, and prove that the embedding is exact.

In the paper, we also discuss general perversities in the framework of $\R$-constructible enhanced ind-sheaves on bordered subanalytic spaces.
\end{abstract}

\maketitle

\tableofcontents

\section*{Introduction}
On a complex manifold $X$, the classical Riemann-Hilbert correspondence establishes an equivalence
\[\dr_X \colon \BDC_\reghol(\D_X)\isoto\BDC_\Cc(\C_X)\]
between the derived category of $\D_X$-modules with regular holonomic cohomologies, and the derived category of sheaves of $\C$-vector spaces on $X$ with $\C$-constructible cohomologies (\cite{Kas84}).
Here, $\dr_X(\shm) = \Omega_X\ltens[\D_X]\shm$ is the de Rham functor, and $\Omega_X$ the sheaf of top-degree holomorphic differential forms.
Moreover, the functor $\dr_X$ interchanges the standard t-structure on $\BDC_\reghol(\D_X)$ with the middle perversity t-structure on $\BDC_\Cc(\C_X)$. In particular, $\dr_X$ induces an equivalence between the abelian category of regular holonomic $\D_X$-modules and that of perverse sheaves on $X$.

The Riemann-Hilbert correspondence of \cite{DK13} provides a fully faithful embedding
\[
\xymatrix{
\drE \colon \BDC_\hol(\D_X) \ar@{ >->}[r] & \dere^{\mathrm{b}}_\Rc(\iCfield_X)
}
\]
from the derived category of $\D_X$-modules with (not necessarily regular) holonomic cohomologies, into the triangulated category of $\R$-constructible enhanced ind-sheaves of $\C$-vector spaces on $X$. Here, $\drE$ is the enhanced version of the de Rham functor.
The source category $\BDC_\hol(\D_X)$ has a standard t-structure.
In this paper, we provide the target category $\dere^{\mathrm{b}}_\Rc(\iCfield_X)$ with a \emph{generalized} middle perversity t-structure, and prove that $\drE$ is an exact functor.

Generalized t-structures have been introduced in \cite{Kas15}, as a reinterpretation of the notion of
slicing from \cite{Bri07}.
For example, let $\BDC_\Rc(\C_X)$ be the derived category of sheaves of $\C$-vector spaces on $X$ with
$\R$-constructible cohomologies.
Then, if $X$ has positive dimension, $\BDC_\Rc(\C_X)$ does not admit a middle perversity t-structure in the
classical sense.
That is, there is no perversity whose induced
t-structure on $\BDC_\Rc(\C_X)$ is self-dual.
However, it is shown in \cite{Kas15} that $\BDC_\Rc(\C_X)$ has a natural middle perversity t-structure in the generalized sense.
This generalized t-structure induces the
middle perversity t-structure
on the subcategory $\BDC_\Cc(\C_X)$.
Moreover, it is compatible with our construction
of the generalized middle perversity t-structure
on $\dere^{\mathrm{b}}_\Rc(\iCfield_X)$,
since the natural embedding
\[
\xymatrix{\BDC_\Rc(\C_X) \ar@{ >->}[r] & \dere^{\mathrm{b}}_\Rc(\iCfield_X)}
\]
turns out to be exact.

From now on, we shall use the term t-structure for the one in the generalized
sense, and refer to the classical notion as a \emph{classical} t-structure.

Let $\field$ be a field and $M$ a real analytic manifold, or more generally a
bordered subanalytic space.
Let $\dere^{\mathrm{b}}_\Rc(\ifield_M)$ be the triangulated category of
$\R$-constructible enhanced ind-sheaves of $\field$-vector spaces on $M$.
In this paper, we also discuss the t-structures on
$\dere^{\mathrm{b}}_\Rc(\ifield_M)$ associated with arbitrary perversities, and
study their functorial properties. Let us give some details.

On the set of maps $p\colon\Z_{\geq 0}\to\R$, consider the involution $*$ given by
\[
p^*(n) \defeq -p(n)-n.
\]
A perversity is a map $p\colon\Z_{\geq 0}\to\R$ such that $p$ and $p^*$ are decreasing.

Let $\BDC_\Rc(\field_M)$ be the derived category of $\R$-constructible sheaves
of $\field$-vector spaces on $M$.
For a locally closed subset $Z$ of $M$, let $\field_Z$ be the extension by zero
to $M$ of the constant sheaf on $Z$.
For $c\in\R$, set
\begin{align*}
{}^p\derd_\Rc^{\leq c}(\field_M) \defeq \{&F\in \BDC_\Rc(\field_M) \semicolon
\text{for any $k\in\Z_{\geq 0}$ there exists a closed} \\
&\text{subanalytic subset $Z\subset M$ of dimension $< k$ such that}\\
& H^j(\field_{M\setminus Z}\tens F)\simeq 0 \text{ for } j> c+p(k)\}, \\
{}^p\derd_\Rc^{\geq c}(\field_M) \defeq \{&F\in \BDC_\Rc(\field_M) \semicolon
\text{for any $k\in\Z_{\geq 0}$ and any closed}\\
&\text{subanalytic subset $Z\subset M$ of dimension $\leq k$ one has}\\
&H^j\rhom(\field_Z, F)\simeq 0 \text{ for } j< c+p(k) \}.
\end{align*}
Then $\bl {}^p\derd_\Rc^{\leq c}(\field_M), {}^p\derd_\Rc^{\geq c}(\field_M) \br_{c\in\R}$ is a t-structure in the sense of Definition~\ref{def:gent}.
Moreover, the duality functor interchanges ${}^p\derd_\Rc^{\leq c}(\field_M)$
and ${}^{p^*}\derd_\Rc^{\geq -c}(\field_M)$.
In particular, the t-structure
$\bl \Dmid{\leq c}(\field_M), \Dmid{\geq c}(\field_M) \br_{c\in\R}$
associated with the middle perversity
$\mathsf m(n)=-n/2$ is self-dual.

The analogous definition for $\R$-constructible enhanced ind-sheaves is
\begin{align*}
{}_p\dere_\Rc^{\leq c}(\ifield_M) \defeq \{&K\in \dere^{\mathrm{b}}_\Rc(\ifield_M) \semicolon
\text{for any $k\in\Z_{\geq 0}$ there exists a closed} \\
&\text{subanalytic subset $Z\subset M$ of dimension $< k$ such that}\\
&H^j(\opb\pi\field_{M\setminus Z}\tens K)\simeq 0 \text{ for } j> c+p(k)\},
\displaybreak[2]\\
{}_p\dere_\Rc^{\geq c}(\ifield_M) \defeq \{&K\in \dere^{\mathrm{b}}_\Rc(\ifield_M) \semicolon
\text{for any $k\in\Z_{\geq 0}$ and any closed}\\
&\text{subanalytic subset $Z\subset M$ of dimension $\leq k$ one has}\\
&H^j\rihom(\opb\pi\field_Z, K)\simeq 0 \text{ for } j< c+p(k) \}.
\end{align*}
It turns out that $\bl {}_p\dere_\Rc^{\leq c}(\ifield_M), {}_p\dere_\Rc^{\geq
c}(\ifield_M) \br_{c\in\R}$ is a t-structure, but it does not behave well
with respect to the duality functor $\Edual_M$. Hence we set
\begin{align*}
\Eprc{\leq c}(\ifield_M) \defeq \{K\in\dere^{\mathrm{b}}_\Rc(\ifield_M) \semicolon{}& K\in\tEprc{\leq c}(\ifield_M),\\
&\Edual_M K \in \tEprc[p^*]{\geq -c-1/2}(\ifield_M) \}, \\
\Eprc{\geq c}(\ifield_M) \defeq \{K\in\dere^{\mathrm{b}}_\Rc(\ifield_M) \semicolon{}& K\in\tEprc{\geq c-1/2}(\ifield_M),\\
&\Edual_M K \in \tEprc[p^*]{\leq -c}(\ifield_M) \} .
\end{align*}
Then $\bl \Eprc{\leq c}(\ifield_M), \Eprc{\geq c}(\ifield_M) \br_{c\in\R}$ is a t-structure, and the duality functor interchanges $\Eprc{\leq c}(\ifield_M)$
and $\Eprc[p^*]{\geq -c}(\ifield_M)$.
In particular, the t-structure
$\bl \Emid{\leq c}(\bM), \Emid{\geq c}(\bM) \br_{c\in\R}$
associated with the middle perversity
$\mathsf m(n)=-n/2$ is self-dual.

Going back to the Riemann-Hilbert correspondence,
the enhanced de Rham functor
\[
\xymatrix{
\drE \colon \BDC_\hol(\D_X) \ar@{ >->}[r] & \dere^{\mathrm{b}}_\Rc(\iCfield_X)
}
\]
is exact with respect to the t-structure
associated with the middle perversity.

The contents of this paper are as follows.

In Section~\ref{sec:perv},
we recall the notion of t-structure on a triangulated category. We also recall the t-structure on the derived category of $\R$-constructible sheaves on a subanalytic space associated with a given perversity.

In Section~\ref{se:enh}, we recall the notions of ind-sheaves and of enhanced ind-sheaves
on a bordered space. In both cases we also discuss the exactness of Grothendieck operations
with respect to the standard classical t-structures.

In Section~\ref{se:enhp}, we introduce the t-structure(s) on the derived category of $\R$-constructible enhanced ind-sheaves on a bordered subanalytic space associated with a given perversity.
We also discuss the exactness of Grothendieck operations
with respect to these t-structures.

Finally, in Section~\ref{se:RH}, we prove the exactness of the embedding, provided by the Riemann-Hilbert correspondence, from the triangulated category of holonomic $\D$-modules on a complex manifold into that of $\R$-constructible enhanced ind-sheaves.

\medskip
\noi
{\bf Acknowledgments} \quad
The first author acknowledges the kind hospitality at RIMS,
Kyoto University, during the preparation of this paper.

\section*{Notations}

In this paper, we take a field $\field$ as base ring.

For a category $\catc$, we denote by $\catc^\op$
the opposite category of $\catc$.

One says that a full subcategory
$\shs$ of a category $\shc$ is \emph{strictly full} if it contains every object of $\shc$ which is isomorphic to an object of $\shs$.

Let $\shc$, $\shc'$ be categories and $F\cl\shc\to \shc'$ a functor.
The {\em essential image} of $\shc$ by $F$, denoted by $F(\shc)$,
is the strictly full subcategory of $\shc'$ consisting of objects which are isomorphic to $F(X)$ for some $X\in\shc$.

For a ring $A$, we denote by $A^\op$ the opposite ring of $A$.

We say that a topological space is \emph{good} if it is Hausdorff, locally compact, countable at infinity, and has finite soft dimension.

\section{T-structures}\label{sec:perv}

The notion of t-structure on a triangulated category was introduced in \cite{BBD82}.
As shown in \cite{Sch99}, the derived category of a quasi-abelian category has two natural t-structures. They were presented in \cite{Kas08} in a unified manner, by generalizing the notion of t-structure. A further generalization is described in \cite{Kas15}, reinterpreting the notion of slicing from \cite{Bri07}.
In the present paper, we use the term t-structure in this more general sense, and we refer to the notion introduced in \cite{BBD82} as a \emph{classical} t-structure.
A basic result of \cite{BBD82} asserts that the \emph{heart} of a classical t-structure is an abelian category. More generally, it is shown in \cite{Bri07} that small \emph{slices} of a t-structure are quasi-abelian categories.

It is shown in \cite{BBD82} that, on a complex manifold, the \emph{middle perversity} induces a self-dual classical t-structure on the triangulated category of $\C$-constructible sheaves. On a real analytic manifold, using results of \cite{KS90}, it is shown in \cite{Kas15} that the middle perversity induces a self-dual t-structure on the triangulated category of $\R$-constructible sheaves.

Here we recall these facts, considering general perversities.

\subsection{Categories}

References are made to \cite[Chapter I]{KS90}, and to \cite{Sch99} for the notion of quasi-abelian category (see also \cite[\S2]{Kas08}).
\medskip

Let $\shc$ be an additive category.
The \emph{left and right orthogonal} of
a subcategory $\shs$ are the strictly full subcategories
\begin{align*}
{}^\bot\shs & \defeq \{X\in\shc\semicolon \Hom[\shc](X,Y) \simeq 0\text{ for any }Y\in\shs \}, \\
\shs^\bot & \defeq \{X\in\shc\semicolon \Hom[\shc](Y,X) \simeq 0\text{ for any }Y\in\shs \}.
\end{align*}

Assume that $\shc$ admits kernels and cokernels.
Given $f\cl X\to Y$ a morphism in $\shc$, one sets
\[
\im f\seteq \ker\bl Y\to\coker f\br, \quad \coim f\seteq\coker\bl \ker f\to X\br.
\]
The morphism $f$ is called \emph{strict} if the canonical morphism $\coim f \to \im f$ is an isomorphism.

The category $\shc$ is called \emph{abelian} if all morphisms are strict.
It is called \emph{quasi-abelian} if every pull-back of a strict epimorphism is a strict epimorphism, and every pushout of a strict monomorphism is a strict monomorphism.

\subsection{T-structures}
Let $\catt$ be a triangulated category.
Recall the notion of t-structure from \cite{BBD82}.

\begin{definition}\label{def:t-clas}
A \emph{classical t-structure} $\bl\catt^{\leq0},\catt^{\geq0}\br$ on $\catt$ is a pair of
strictly full subcategories of $\catt$ such that,
setting
\[
\catt^{\leq n}\defeq\catt^{\leq0}[-n], \quad \catt^{\geq n}\defeq\catt^{\geq0}[-n]
\]
for $n\in\Z$, one has:
\begin{itemize}
\item[(a)] $\catt^{\leq 0}\subset\catt^{\leq1}$ and
$\catt^{\geq 1}\subset\catt^{\geq 0}$;
\item[(b)] $\Hom[\catt](\catt^{\leq 0},\catt^{\geq 1})=0$;
\item[(c)] for any $X\in\catt$, there exists a distinguished triangle
\[
X_{\leq 0}\to X\to X_{\geq 1}\tone
\]
in $\catt$ with $X_{\leq 0}\in \catt^{\leq 0}$ and $X_{\geq 1}\in\catt^{\geq 1}$.
\end{itemize}
\end{definition}

The following definition of \cite{Kas15} is a reinterpretation of the notion of slicing from \cite{Bri07}.

\begin{definition}
\label{def:gent}
A \emph{t-structure} $\bl \catt^{\leq c}, \catt^{\geq c} \br_{c\in\R}$ on $\catt$ is a pair of families of strictly full subcategories of $\catt$ satisfying conditions (a)--(d) below, where we set
\[
\catt^{<c} \defeq \Union_{c'<c}\catt^{\leq c'}\qtq
\catt^{> c} \defeq \Union_{c'>c}\catt^{\geq c'}\qtq[for $c\in\R$.]
\]
\begin{itemize}
\item[(a)] $\catt^{\leq c}=\Inter_{c'>c}\catt^{\leq c'}$ and
$\catt^{\geq c}=\Inter_{c'<c}\catt^{\geq c'}$ for any $c\in\R$,
\item[(b)] $\catt^{\leq c+1}=\catt^{\leq c}[-1]$ and $\catt^{\geq c+1}=\catt^{\geq c}[-1]$ for any $c\in\R$,
\item[(c)]
$\Hom[\catt](\catt^{<c},\catt^{>c})=0$ for any $c\in\R$,
\item[(d)]
for any $X\in\catt$ and $c\in\R$, there are distinguished triangles in $\catt$
\[
X_{\leq c}\to X\to X_{> c}\tone\qtq X_{< c}\to X\to X_{\geq c}\tone
\]
with $X_{L}\in \catt^L$ for $L$ equal to $\leq c$, $>c$, $<c$ or $\geq c$.
\end{itemize}
\end{definition}

Condition (c) is equivalent to either of the following:
\begin{itemize}
\item[(c)$'$]
$\Hom[\catt](\catt^{\leq c},\catt^{>c})=0$ for any $c\in\R$,
\item[(c)$''$]
$\Hom[\catt](\catt^{<c},\catt^{\geq c})=0$ for any $c\in\R$.
\end{itemize}

The next lemma is elementary but useful.
It shows for example that, under condition (a), for any $c\in\R$ one has
\[
\catt^{\leq c}=\Inter_{c'>c} \catt^{< c'}, \quad
\catt^{\geq c}=\Inter_{c'<c} \catt^{> c'}.
\]

\begin{lemma}\label{lem:Xcc'}
Let $X$ be a set.
\bnum
\item Let $(X^{\leq c})_{c\in\R}$ be a family of subsets of $X$
such that $X^{\leq c}=\Inter_{c'>c} X^{\leq c'}$ for any $c\in\R$.
Set $X^{<c}\defeq \Union_{c'<c} X^{\leq c'}$.
Then
\begin{align*}
X^{<c}=\Union_{c'<c} X^{<c'}, \quad
X^{\leq c}=\Inter_{c'>c} X^{<c'}.
\end{align*}
\item Conversely, let $(X^{<c})_{c\in\R}$ be a family of subsets of $X$
such that $X^{<c}=\Union_{c'<c} X^{<c'}$ for any $c\in\R$.
Set $X^{\leq c}\defeq \Inter_{c'>c} X^{<c'}$.
Then
\begin{align*}
X^{\leq c}=\Inter_{c'>c} X^{\leq c'}, \quad
X^{<c}=\Union_{c'<c} X^{\leq c'}.
\end{align*}
\item
Let $(X^{\leq c})_{c\in\R}$ and $(X^{< c})_{c\in\R}$ be as in {\rm(i)}.
Let $a,b\in\R$ with $a<b$.
If $X^{<c} = X^{\leq c}$ for any $c$ such that $a<c\leq b$, then $X^{\leq a} = X^{\leq b}$.
\ee
\end{lemma}

Let $\bl\catt^{\leq0},\catt^{\geq0}\br$ be a classical t-structure.
For $c\in\R$, set
\begin{align*}
\catt^{\leq c} \defeq \catt^{\leq 0}[-n] &\quad\text{for $n\in\Z$ such that }n\leq c < n+1, \\
\catt^{\geq c} \defeq \catt^{\geq 0}[-n] &\quad\text{for $n\in\Z$ such that }n-1 < c \leq n.
\end{align*}
Then, $\bl \catt^{\leq c}, \catt^{\geq c} \br_{c\in\R}$ is a t-structure.
A classical t-structure is regarded as a t-structure by this
correspondence.

\smallskip
Conversely, if $\bl \catt^{\leq c}, \catt^{\geq c} \br_{c\in\R}$ is a
t-structure, then
\begin{equation}
\label{eq:tcclass}
\bl \catt^{\leq c+1},\catt^{>c} \br\qtq \bl \catt^{<c+1},\catt^{\geq c}\br
\end{equation}
are classical t-structures for any $c\in\R$.

For $c\in\R$, set
\[
\catt^c \defeq \catt^{\leq c} \cap \catt^{\geq c}.
\]

\begin{definition}
Let $\Sigma\subset\R$ be a discrete subset such that $\Sigma = \Sigma+\Z$.
A t-structure $\bl \catt^{\leq c}, \catt^{\geq c} \br_{c\in\R}$ is \emph{indexed by $\Sigma$}
if $\catt^c = 0$ for any $c\in\R\setminus\Sigma$.
\end{definition}

If $\Sigma$ is non empty,
this is equivalent to the fact that for any $c\in\R$ one has
\begin{align*}
\catt^{<c} = \catt^{\leq s'}, &\quad \catt^{\leq c} = \catt^{\leq s''}, \\
\catt^{>c} = \catt^{\geq t'}, &\quad \catt^{\geq c} = \catt^{\geq t''},
\end{align*}
where
\begin{align*}
s' &\defeq \max \{s\in\Sigma \semicolon s< c\}, &
s''&\defeq \max \{s\in\Sigma \semicolon s\leq c\}, \\
t' &\defeq \min \{s\in\Sigma \semicolon s > c\},&
t'' &\defeq \min \{s\in\Sigma \semicolon s \geq c\}.
\end{align*}

Classical t-structures correspond to t-structures indexed by $\Z$.
In this paper, we will mainly consider t-structures indexed by $\frac12\Z$.

The following lemma is easily proved
by using Lemma~\ref{lem:Xcc'}~(iii).

\Lemma Let $\bl \catt^{\leq c}, \catt^{\geq c} \br_{c\in\R}$
be a t-structure on $\catt$.
The following two conditions are equivalent.
\bna
\item $\bl \catt^{\leq c}, \catt^{\geq c} \br_{c\in\R}$ is indexed by some
 discrete subset $\Sigma\subset\R$ such that $\Sigma = \Sigma+\Z$.
\item For any $c\in\R$, there exist $a,b\in\R$ such that
$a<c<b$ and $\catt^{< c}=\catt^{\leq a}$ and $\catt^{> c}=\catt^{\geq b}$.
\ee
\enlemma

\subsection{Slices}
Let $\bl \catt^{\leq c}, \catt^{\geq c} \br_{c\in\R}$ be a t-structure on $\catt$.
Note the following facts.

For any $c\in\R$, one has
\begin{align*}
\catt^{ > c} = (\catt^{\leq c})^\bot, &\quad \catt^{\leq c} = {}^\bot(\catt^{> c}), \\
\catt^{\geq c} = (\catt^{< c})^\bot, &\quad \catt^{< c} = {}^\bot(\catt^{\geq c}).
\end{align*}

The embeddings $\catt^{\leq c}\subset\catt$ and $\catt^{< c}\subset\catt$
admit left adjoints
\[
\tau^{\leq c}\colon \catt \to \catt^{\leq c}\qtq \quad
\tau^{< c}\colon \catt \to \catt^{< c},
\]
called the \emph{left truncation functors}.
Similarly, the embeddings $\catt^{\geq c}\subset\catt$ and $\catt^{> c}\subset\catt$
admit right adjoints
\[
\tau^{\geq c}\colon \catt \to \catt^{\geq c}\qtq
\tau^{> c}\colon \catt \to \catt^{> c},
\]
called the \emph{right truncation functors}.

The distinguished triangles in Definition~\ref{def:gent} (d) are unique up to
unique isomorphism. They are, respectively, given by
\[
\tau^{\leq c}X \to X \to \tau^{> c}X \tone\qtq
\tau^{< c}X \to X \to \tau^{\geq c}X \tone.
\]

Summarizing the above notations, to a half-line $L$
(i.e.\ an unbounded connected subset $L\subsetneq \R$) is associated a truncation functor
\[
\tau^L\colon \catt \to \catt^L.
\]
If $L'\subset\R$ is another half-line, there is an isomorphism of functors
\begin{equation}
\label{eq:HI}
\tau^{L} \comp \tau^{L'} \simeq \tau^{L'} \comp \tau^{L}\colon \catt \to \catt^L \cap \catt^{L'}.
\end{equation}

Let $I\subset\R$ be a proper interval
(i.e.\ a bounded connected non empty subset $I\subset\R$). Then there are two half-lines $L,L'$
(unique up to ordering) such that
\[
I= L\cap L'.
\]
The \emph{slice} of $\catt$ associated with $I$ is the additive category
\[
\catt^I \defeq \catt^{L} \cap \catt^{L'},
\]
and one denotes the functor \eqref{eq:HI} by
\[
H^I \colon \catt \to \catt^I.
\]
For example,

$\catt^{[c,c')} = \catt^{\geq c} \cap \catt^{<c'}$ for $c<c'$, and $\catt^{\{c\}} = \catt^c$.
One writes for short
\[
H^c\seteq H^{\{c\}}.
\]

Note that the map $I\to \R/\Z$ is bijective if and only if
$I=[c,c+1)$ or $I=(c,c+1]$ for some $c\in\R$.
The map $I\to \R/\Z$ is injective if and only if there exists $c\in\R$
such that $I\subset [c,c+1)$ or $I\subset(c,c+1]$.

The following result generalizes the fact that
the heart $\catt^0$ of a classical t-structure
$\bl \catt^{\leq 0}, \catt^{\geq 0} \br$ is abelian.

\begin{proposition}[{cf.\ \cite[Lemma 4.3]{Bri07}}]
\label{pro:qabelian}
Let $\bl \catt^{\leq c}, \catt^{\geq c} \br_{c\in\R}$
be a {\em t-structure} on $\catt$, and
let $I\subset\R$ be an interval.
\bnum
\item
If $I\to \R/\Z$ is injective, then
the slice $\catt^I$ is a quasi-abelian category, and strict short exact sequences in $\catt^I$ are in one-to-one correspondence with distinguished triangles in $\catt$ with all vertices in $\catt^I$.
\item
If $I\to \R/\Z$ is bijective, then
the slice $\catt^I$ is an abelian category and the functor $H^I \colon \catt \to \catt^I$ is cohomological.
\ee
\end{proposition}

\begin{remark}
The notion of slicing from \cite{Bri07} is equivalent to the datum of a t-structure $\bl \catt^{\leq c}, \catt^{\geq c} \br_{c\in\R}$ such that $\catt$ is generated by the family of subcategories
$\{\catt^c\}_{c\in\R}$.
\end{remark}

\subsection{Exact functors}
Let $\cats$ and $\catt$ be triangulated categories.
Let
\[
\bl \cats^{\leq c}, \cats^{\geq c} \br_{c\in\R} \qtq
\bl \catt^{\leq c}, \catt^{\geq c} \br_{c\in\R}
\]
be t-structures on $\cats$ and $\catt$, respectively.

\begin{definition}
A triangulated functor $\Phi \colon \cats \to \catt$ is called
\begin{itemize}
\item[(i)]
\emph{left exact}, if one has $\Phi (\cats^{\geq c}) \subset \catt^{\geq c}$
for any $c\in\R$;
\item[(ii)]
\emph{right exact}, if one has $\Phi (\cats^{\leq c}) \subset \catt^{\leq c}$
for any $c\in\R$;
\item[(iii)]
\emph{exact}, if it is both left and right exact.
\end{itemize}
\end{definition}

\begin{lemma}
Consider two triangulated functors
\[
\Phi \colon \cats \to \catt\qtq\Psi \colon \catt \to \cats.
\]
Assume that $(\Phi,\Psi)$ is an adjoint pair. This means that
$\Phi$ is left adjoint to $\Psi$, or equivalently that $\Psi$ is right adjoint to $\Phi$.
Then, $\Psi$ is left exact if and only if $\Phi$ is right exact.
\end{lemma}

\begin{proof}
Let $c\in\R$. If $\Psi$ is exact,
then, for $S\in\cats^{\leq c}$ and $T\in\catt^{> c}$, one has
\begin{align*}
\Hom[\catt](\Phi(S),T)
&\simeq \Hom[\cats](S,\Psi(T)) \\
&\in \Hom[\cats](\cats^{\leq c},\cats^{> c}) = 0.
\end{align*}
Hence, $\Phi(S) \in {}^\bot(\catt^{> c}) = \catt^{\leq c}$.
Thus $\Phi$ is right exact.
The converse can be proved similarly.
\end{proof}

\subsection{Sheaves}\label{subsec:sheaf}
Let $M$ be a good topological space.
Denote by $\Mod(\field_M)$ the abelian category of sheaves of $\field$-vector spaces on $M$, and by $\BDC(\field_M)$ its bounded derived category. It has a
standard classical t-structure
\[
\bl \derd^{\leq 0}(\field_M), \derd^{\geq 0}(\field_M) \br.
\]

For a locally closed subset $S\subset M$, denote by $\field_S$ the sheaf on $M$ obtained
extending by zero the constant sheaf on $S$ with stalk $\field$.

For $f\colon M\to N$ a morphism of good topological spaces,
denote by $\tens$, $\rhom$, $\opb f$, $\roim f$, $\reim f$, $\epb f$ the six
Grothendieck operations for sheaves.

We define the duality functor of $\BDC(\field_M)$ by
\[
\dual_M F = \rhom(F,\omega_M) \quad\text{for $F\in\BDC(\field_M)$,}
\]
where $\omega_M$ denotes the dualizing complex.
If $M$ is a C$^0$-manifold, one has $\omega_M \simeq \ori_M[d_M]$, where $d_M$ denotes the dimension of $M$ and $\ori_M$ the orientation sheaf.
For a map $f\cl M\to N$ of C$^0$-manifolds,
the {\em relative orientation sheaf}
is defined as
$\ori_{M/N}\seteq\epb{f}\cor_N[d_N-d_M]\simeq\ori_M\tens \opb{f}\ori_N$.

\subsection{$\R$-constructible sheaves}
Recall the notion of subanalytic subsets of a real analytic manifold
(see \cite{Hi73,BM88}).

\begin{definition}
\bnum
\item
A \emph{subanalytic space} $M=(M,\shs_M)$ is an $\R$-ringed space which is locally isomorphic to $(Z,\shs_Z)$, where $Z$ is a closed subanalytic subset
of a real analytic manifold, and $\shs_Z$ is the sheaf of $\R$-algebras of real valued subanalytic continuous functions.
In this paper, we assume that subanalytic spaces are good topological spaces.
\item
A morphism of subanalytic spaces is a morphism of $\R$-ringed spaces.
\item
A subset $S$ of $M$ is {\em subanalytic} if $i(S\cap U)$ is a subanalytic subset of $N$ for any open subset $U$ of $M$, any real analytic manifold $N$
 and any subanalytic morphism $i\cl U\to N$ of subanalytic spaces
such that $i$ induces an isomorphism from $U$
to a closed subanalytic subset of $N$.
\ee
\end{definition}

Let $M$ be a subanalytic space.
One says that a sheaf $F\in\Mod(\field_M)$ is {\em $\R$-constructible}
if there exists a locally finite family of locally closed subanalytic subsets
$\{S_i\}_{i\in I}$ of $M$ such that $M=\sbcup_{i\in I}S_i$ and
$F$ is locally constant of finite rank on each $S_i$.
Denote by $\BDC_\Rc(\field_M)$ the full subcategory of $\BDC(\field_M)$ whose objects have $\R$-constructible cohomologies.

\subsection{Perversities}\label{sse:perv}
On the set of maps $p\colon\Z_{\geq 0}\to\R$, consider the involution $*$ given by
\[
p^*(n) \defeq -p(n)-n.
\]

\begin{definition}
\begin{itemize}
\item[(i)]
A function $p\colon\Z_{\geq 0}\to\R$ is a \emph{perversity} if both $p$ and $p^*$ are decreasing, i.e.\ if
\[
0\leq p(n)-p(m)\leq m-n \quad\text{for any $m,n\in\Z_{\ge0}$ such that $n\le m$.}
\]
\item[(ii)]
A \emph{classical perversity} is a $\Z$-valued perversity.
\end{itemize}
\end{definition}

Let $M$ be a subanalytic space. To a classical perversity $p$ is associated a classical t-structure $\bl {}^p\derd^{\leq 0}_\Rc(\field_M), {}^p\derd^{\geq 0}_\Rc(\field_M) \br$
on $\BDC_\Rc(\field_M)$ (refer to \cite{BBD82} and \cite[\S10.2]{KS90}).
Here, slightly generalizing a construction in \cite{Kas15}, we will associate a t-structure to a perversity.

\begin{notation}
Set
\[
\CS_M \defeq \{\text{closed subanalytic subsets of $M$}\}.
\]
For $Z\in\CS_M$, denote by
\[
i_Z\colon Z \to M
\]
the embedding. Set
$$\text{$d_Z\defeq\dim Z$ (with $d_\emptyset= -\infty$).}$$
For $k\in\Z$, set
\begin{align*}
\CS^{< k}_M &\defeq \{Z\in \CS_M \semicolon d_Z< k\}, \\
\CS^{\leq k}_M &\defeq \{Z\in \CS_M \semicolon d_Z\leq k\}.
\end{align*}
\end{notation}

\begin{definition}\label{def:midd}
Let $p$ be a perversity, $c\in\R$ and $k\in\Z_{\geq 0}$. Consider the following conditions
on $F\in\BDC(\field_M)$
\begin{align*}
(p_k^{\leq c}) &\colon
\opb i_{M\setminus Z} F\in \derd^{\leq c+p(k)}(\field_{M\setminus Z}) \text{ for some } Z\in\CS_M^{<k}, \\
(p_k^{\geq c}) &\colon
\epb{i_Z} F \in \derd^{\geq c+p(k)}(\field_Z) \text{ for any } Z\in\CS_M^{\leq k}.
\end{align*}
We define the following strictly full subcategories of $\BDC(\field_M)$
\begin{align*}
{}^p\derd^{\leq c}(\field_M)
&\defeq \{ F\in\BDC(\field_M) \semicolon (p_k^{\leq c})\text{ holds for any $k\in\Z_{\geq 0}$} \}, \\
{}^p\derd^{\geq c}(\field_M)
&\defeq \{ F\in\BDC(\field_M) \semicolon (p_k^{\geq c})\text{ holds for any $k\in\Z_{\geq 0}$} \}.
\end{align*}
Let us also set
\begin{align*}
{}^p\derd_\Rc^{\leq c}(\field_M) &\defeq {}^p\derd^{\leq c}(\field_M) \cap \BDC_\Rc(\field_M), \\
{}^p\derd_\Rc^{\geq c}(\field_M) &\defeq {}^p\derd^{\geq c}(\field_M) \cap \BDC_\Rc(\field_M).
\end{align*}
\end{definition}

Note that $\bl \Dp{\leq c}(\field_M), \Dp{\geq c}(\field_M) \br_{c\in\R}$
is not a t-structure if $\dim M>0$.

\begin{lemma}\label{lem:classtleq}
For $c\in\R$, $k\in\Z_{\geq 0}$ and $F\in\BDC_\Rc(\field_M)$,
the following conditions are equivalent
\begin{itemize}
\item[(i)] $F$ satisfies $(p_k^{\leq c})$,
\item[(ii)] $\dim(\supp(H^j F))<k$ for any $j$ with $j>c+p(k)$.
\end{itemize}
\end{lemma}

\begin{proof}
It is enough to remark that
$\opb i_{M\setminus Z} F\in \derd^{\leq c+p(k)}(\field_{M\setminus Z})$
if and only if $\supp(H^j F)\subset Z$ for any $j$ such that $j>c+p(k)$.
\end{proof}

\begin{proposition}\label{pro:t-pre}
We have the following properties.
\begin{itemize}
\item[(i)]
$\bl \Dprc{\leq c}(\field_M), \Dprc{\geq c}(\field_M) \br_{c\in\R}$
is a t-structure on $\BDC_\Rc(\field_M)$.
\item[(ii)]
For any $c\in\R$, the duality functor $\dual_M$ interchanges $\Dprc{\leq c}(\field_M)$ and $\Dprc[p^*]{\geq -c}(\field_M)$.
\item[(iii)]
For any interval $I\subset\R$ such that $I\to\R/\Z$ is injective,
the prestack on $M$
\[
U\mapsto \Dprc{I}(U)
\]
is a stack of quasi-abelian categories.
\end{itemize}
\end{proposition}

\begin{proof}
Note that, for (iii), it is enough to consider the case where $I\to\R/\Z$ is bijective,
i.e.\ the case where $I=[c,c+1)$ or $I=(c,c+1]$ for some $c\in\R$.

(a) If $p$ is a classical perversity, the result is due to \cite{{BBD82}}.
More precisely, for the statements (i), (ii) and (iii) refer to Theorem 10.2.8, Proposition 10.2.13 and Proposition 10.2.9 of \cite{KS90}, respectively.

(b) Let now $p$ be an arbitrary perversity.

For $c\in\R$, denote by $\lfloor c\rfloor$ the largest integer not greater than $c$, and by $\lceil c \rceil$ the smallest integer not smaller than $c$.
Note that $ \lceil c \rceil + \lfloor -c\rfloor = 0$.

Statements (i) and (iii) follow from (a) by noticing that for any $c\in\R$
\[
\bl \Dprc{< c+1}(\field_M), \Dprc{\geq c}(\field_M) \br\qtq
\bl \Dprc{\leq c}(\field_M), \Dprc{> c-1}(\field_M) \br
\]
are the classical t-structures associated to the classical perversities
\[
p_{c,\;+}(n) \defeq \lceil c + p(n) \rceil, \quad
p_{c,\;-}(n) \defeq \lfloor c+ p(n) \rfloor,
\]
respectively.

\noi
Statement (ii) follows from (a) by noticing that one has
\[
(p_{c,\;\pm})^* = (p^*)_{-c,\;\mp}.
\]
\end{proof}
Note that $\bl \Dprc{\leq c}(\field_M), \Dprc{\geq c}(\field_M) \br_{c\in\R}$
is indexed by
$\sbcup\limits_{0\le k\le d_M}(-p(k)+\Z)$.

\begin{definition}
The \emph{middle perversity t-structure}
\[\bl \Dmid{\leq c}(\field_M), \Dmid{\geq c}(\field_M) \br_{c\in\R}
\]
is the one associated with the \emph{middle perversity} $\mathsf m(n) \defeq -n/2$.
\end{definition}

Note that $\mathsf m$ is the only perversity stable by $*$. In particular, the middle perversity t-structure is self-dual. It is indexed by $\frac12\Z$.

\section{Enhanced ind-sheaves}\label{se:enh}

Let $M$ be a good topological space.
The derived category of enhanced ind-sheaves on $M$ is defined as a quotient of the derived
category of ind-sheaves on the bordered space $M\times\R_\infty$.
We recall here these notions and some related results from \cite{DK13}.
We also discuss the exactness of Grothendieck operations
with respect to the standard classical t-structures.

References are made to \cite{KS01} for ind-sheaves, and to \cite{DK13} for bordered spaces and enhanced ind-sheaves.
See also \cite{KS14} for enhanced ind-sheaves on bordered spaces and \cite{KS15} for an exposition.

\subsection{Semi-orthogonal decomposition}
Let $\catt$ be a triangulated category, and $\catn\subset\catt$ a strictly full triangulated subcategory.
We denote by $\catt/\catn$ the quotient triangulated category
(see e.g.\ \cite[{\S\,10.2}]{KS06}).

\begin{proposition}
\label{pro:D/N}
Let $\catn\subset\catt$ be a strictly full triangulated subcategory which contains every direct summand in $\catt$ of an object of $\catn$.
Then the following conditions are equivalent:
\bnum
\item the embedding $\catn \to \catt$ has a left adjoint,
\item the quotient functor $\catt\to\catt/\catn$ has a left adjoint,
\item the composition ${}^\bot\catn \to \catt \to \catt/\catn$ is an equivalence of categories,
\label{botDN}
\item for any $X\in\catt$ there is a distinguished triangle $X'\to X\to X''\tone$ with $X'\in {}^\bot\catn$ and $X''\in \catn$,
\label{dt-bot}
\item the embedding ${}^\bot\catn\to\catt$ has a right adjoint, and $\catn \simeq ({}^\bot\catn)^\bot$.
\ee
A similar result holds switching ``left'' with ``right''.
\end{proposition}

\subsection{Ind-sheaves}
Let $\catc$ be a category and denote by $\catc^\wedge$ the category of
contravariant functors from $\catc$ to the category of sets. Consider the Yoneda
embedding $h\colon \catc \to \catc^\wedge$, $X\mapsto\Hom[\catc](\ast,X)$.
The category $\catc^\wedge$ admits small colimits.
As colimits do not commute with $h$, one denotes by $\colim$ 
the colimits taken in $\catc$, and by 
$\indlim$ the colimits taken in
$\catc^\wedge$.

An ind-object in $\catc$ is an object of $\catc^\wedge$ isomorphic to
$\indlim \varphi$ for some functor $\varphi\colon I\to\catc$ with $I$ a small filtrant category.
Denote by $\operatorname{Ind}(\catc)$ the full subcategory of $\catc^\wedge$ consisting of ind-objects in $\catc$.

\medskip

Let $M$ be a good topological space. The category of ind-sheaves on $M$ is the category
\[
\ind(\field_M) \defeq \operatorname{Ind}(\Mod_c(\field_M))
\]
of ind-objects in the category $\Mod_c(\field_M)$ of sheaves with
compact support.

The category $\ind(\field_M)$ is abelian, and the prestack on $M$ given by $U\mapsto \ind(\field_U)$ is a stack of abelian categories. There
is a natural exact fully faithful functor $\iota_M\colon\Mod(\field_M) \to \ind(\field_M)$ given by $F\mapsto\indlim (\field_U \tens F)$, for $U$ running over the relatively compact open subsets of $M$.
The functor $\iota_M$ has an exact left adjoint
$\alpha_M\colon\ind(\field_M) \to \Mod(\field_M)$ given by $\alpha_M(\indlim
\varphi) = \ilim \varphi$.

In this paper, we set for short
\[
\derd(M) \defeq \BDC(\ind(\field_M)),
\]
and denote by
\[
\bl \derd^{\leq 0}(M), \derd^{\geq 0}(M) \br
\]
its standard classical t-structure.

For $f\colon M\to N$ a morphism of good topological spaces, denote by $\tens$,
$\rihom$, $\opb f$, $\roim f$, $\reeim f$, $\epb f$ the six Grothendieck
operations for ind-sheaves.

Since ind-sheaves form a stack, they have a sheaf-valued hom-functor $\hom$. One has $\rhom \simeq \alpha_M\circ\rihom$.

\subsection{Bordered spaces}\label{sse:bord}
A \emph{bordered space} $\bM = (\oM,\cM)$ is a pair
of a good topological space $\cM$
and an open subset $\oM$ of $\cM$.

\begin{notation}
\label{not:Gammaf}
Let $\bM = (\oM,\cM)$ and $\bN = (\oN,\cN)$ be bordered spaces.
For a continuous map $f\colon \oM\to\oN$,
denote by $\Gamma_f \subset \oM\times\oN$ its graph, and by
$\olG$ the closure of $\Gamma_f$ in $\cM\times\cN$.
Consider the projections
\[
\xymatrix{
\cM & \ar[l]_-{q_1} \cM \times \cN \ar[r]^-{q_2} & \cM.
}
\]
\end{notation}

Bordered spaces form a category as follows:
a morphism $f\colon \bM \to \bN$ is a continuous map
$f\colon \oM\to\oN$ such that $q_1|_{\olG}\cl\olG\to\cM$ is proper;
the composition of two morphisms is the composition of the underlying continuous maps.

\begin{remark}\bnum
\item
If $f\cl \oM\to \oN$ can be extended to a continuous map $\cf\cl\cM\to\cN$,
then $f$ is a morphism of bordered space.
\item
The forgetful functor from the category of bordered spaces to that of good topological spaces is given by
\[
\bM = (\oM,\cM) \mapsto \unbM \defeq \oM.
\]
It has a fully faithful left adjoint $M\mapsto(M,M)$. By this functor, we consider good topological spaces as particular bordered spaces, and denote $(M,M)$ by $M$.

Note that $\bM = (\oM,\cM) \mapsto \cM$ is not a functor.
\ee
\end{remark}

Let $\bM = (\oM,\cM)$ be a bordered space.
The continuous maps $\oM \to[\id] \oM \hookrightarrow \cM$ induce
morphisms of bordered spaces
\begin{equation}
\label{eq:joM}
\oM \To \bM \To[j_\bM] \cM.
\end{equation}

Note that $\bM \simeq (\oM,\overline M)$, where $\overline M$ is the closure of $\oM$ in $\cM$.

\begin{notation}\label{not:bZ}
For a locally closed subset $Z$ of $\oM$, set $\bZ=(Z,\overline Z)$, where $\overline Z$ is the closure of $Z$ in $\cM$, and denote $i_\bZ\colon\bZ\to\bM$ the morphism induced by the embedding $Z\subset \oM$.
\end{notation}

\begin{lemma}
Let $f\colon \bM \to \bN$ be a morphism of bordered spaces.
Let $Z\subset \unbM$ and $W\subset \unbN$ be locally closed subsets
such that $f(Z)\subset W$.
Then $f\vert_Z\cl Z\to W$ induces a morphism
$\inbordered Z \to \inbordered W $ of bordered spaces.
\end{lemma}

In particular, the bordered space $\bZ$ only depends on $\bM$
(and not on $\cM$).

\begin{definition}
We say that a morphism $f\colon \bM \to \bN$ is \emph{semi-proper}
if $q_2|_{\olG}\cl\olG\to\cN$ is proper.
We say that $f$ is \emph{proper} if moreover $\unb f\cl\unbM\to\unbN$
is proper.
\end{definition}

For example, $j_\bM$ and $i_\bZ$ are semi-proper.

\Def\label{def:subset}
A subset $S$ of a bordered space $\bM=(\oM,\cM)$ is a subset of $M$.
We say that $S$ is open (resp.\ closed, locally closed) if it is so in $\oM$.
We say that $S$ is relatively compact if it is contained
in a compact subset of $\cM$.
\edf

As seen by the following obvious lemma,
the notion of relatively compact subsets only depends on $\bM$
(and not on $\cM$).

\Lemma\label{lem:proper}
Let $f\cl \bM\to\bN$ be a morphism of bordered spaces.
\bnum
\item
If $S$ is a relatively compact subset of $\bM$, then its image
$\unb f(S)\subset \unbN$ is a relatively compact subset of $\bN$.
\item Assume furthermore that $f$ is semi-proper.
If $S$ is a relatively compact subset of $\bN$, then its inverse image
${\unb f}{}^{-1}(S)\subset \unbM$ is a relatively compact subset of $\bM$.
\ee
\enlemma

\subsection{Ind-sheaves on bordered spaces}\label{sse:Ibord}
Let $\bM$ be a bordered space.
The abelian category of ind-sheaves on $\bM$ is
\[
\ind(\field_\bM) \defeq \operatorname{Ind}(\Mod_c(\field_\bM)),
\]
where $\Mod_c(\field_\bM)\subset \Mod(\field_\unbM)$ is the full subcategory of sheaves on $\unbM$ whose support is relatively compact in $\bM$.

There is a natural exact embedding $\iota_\bM\colon\Mod(\field_\unbM) \to \ind(\field_\bM)$ given by
$F\mapsto \indlim (\field_U \tens F)$, for $U$ running over the family of relatively compact open subsets of $\bM$.

We set for short
\[
\derd(\bM) \defeq \BDC(\ind(\field_\bM)),
\]
and denote by
\[
\bl \derd^{\leq 0}(\bM), \derd^{\geq 0}(\bM) \br
\]
its standard classical t-structure.

Let $\bM = (\oM,\cM)$, and consider the embeddings
\[
\xymatrix{
\cM\setminus\oM \ar[r]^-i & \cM & \oM \ar[l]_-j.
}
\]
The functor $\roim i \simeq \reeim i$ induces
the embedding $\derd(\cM\setminus\oM) \subset \derd(\cM)$,
which admits a left and a right adjoint.

\Prop There is an equivalence of triangulated categories:
$$\derd(\bM)\simeq \derd(\cM)/\derd(\cM\setminus\oM).$$
\enprop

\Proof
The functor $\eim {{j}}$ induces an exact functor
$\Mod_c(\field_\bM)\to\Mod_c(\field_\cM)$, which induces an exact functor
$\ind(\field_\bM)\to \ind(\field_\cM)$
and functors of triangulated categories
$\derd(\bM)\to \derd(\cM)\to \derd(\cM)/\derd(\cM\setminus\oM)$.

On the other hand, the functor $\opb j$ induces an
exact functor $\Mod_c(\field_\cM)\to
\Mod_c(\field_\bM)$, which induces an exact functor
$\ind(\field_\cM)\to \ind(\field_\bM)$
and a functor of triangulated categories
$\derd(\cM)\to \derd(\bM)$. Since the composition
$\derd(\cM\setminus\oM)\to \derd(\cM)\to \derd(\bM)$ vanishes,
we obtain a functor
$\derd(\cM)/\derd(\cM\setminus\oM)\to\derd(\bM)$.

It is obvious that these functors between $\derd(\bM)$
and $\derd(\cM)/\derd(\cM\setminus\oM)$ are quasi-inverse to each other.
\QED

Thus, there are equivalences
\[
\derd(\bM)
\simeq \derd(\cM)/\derd(\cM\setminus\oM)
\simeq {}^\bot\derd(\cM\setminus\oM)
\simeq \derd(\cM\setminus\oM)^\bot,
\]
and one has 
\begin{align*}
{}^\bot\derd(\cM\setminus\oM) &\simeq \{F\in \derd(\cM) \semicolon \field_M\tens F \isoto F \}, \\
\derd(\cM\setminus\oM)^\bot &\simeq \{F\in \derd(\cM) \semicolon 
\rihom(\field_M, F) \isofrom F \}.
\end{align*}

Denote by
\[
\derived q_\bM \colon \derd(\cM) \to \derd(\bM), \quad
\derived l_\bM,\derived r_\bM \colon \derd(\bM) \to \derd(\cM)
\]
the quotient functor and its left and right adjoint, respectively.
For $F\in\derd(\cM)$, they satisfy
\begin{equation}
\label{eq:lqrq}
\derived l_\bM\derived q_\bM F \simeq \field_M\tens F, \quad
\derived r_\bM\derived q_\bM F \simeq \rihom(\field_M, F).
\end{equation}

\begin{remark}
At the level of sheaves, there is a natural equivalence
\[
\BDC(\field_M) \simeq \BDC(\field_\cM)/\BDC(\field_{\cM\setminus\oM}).
\]
There is a commutative diagram
\[
\xymatrix@R=2.5ex{
\BDC(\field_M) \ar@{ >->}[r]^{\iota_\bM} \ar@{-}[d]_\wr & \derd(\bM) \ar@{-}[d]_\wr \\
\BDC(\field_\cM)/\BDC(\field_{\cM \setminus \oM}) \ar@{ >->}[r] & \derd(\cM)/\derd(\cM\setminus\oM).
}
\]\end{remark}
The functor $\iota_\bM\cl \BDC(\field_\unbM)\to \derd(\bM)$ has a left adjoint
$$\al_\bM\cl \derd(\bM)\to \BDC(\field_\unbM).$$
It coincides with the composition
$$\derd(\bM)\To\derd(\unbM)\To[\al_{\sunbM}] \BDC(\field_\unbM).$$

Let $f\colon\bM\to\bN$ be a morphism of bordered spaces.
The six Grothendieck operations for ind-sheaves on bordered spaces
\begin{align*}
\tens &\colon \derd(\bM) \times \derd(\bM) \to \derd(\bM), \\
\rihom &\colon \derd(\bM)^\op \times \derd(\bM) \to \derd(\bM), \\
\reeim f, \roim f &\colon \derd(\bM) \to \derd(\bN), \\
\opb f, \epb f &\colon \derd(\bN) \to \derd(\bM)
\end{align*}
are defined as follows.
Recalling Notation~\ref{not:Gammaf}, observe that $\Gamma_f$ is locally closed in $\cM\times\cN$.
For $F,F'\in\derd(\cM)$ and $G\in\derd(\cN)$, one sets
\begin{align*}
\derived q_\bM F\tens \derived q_\bM F' &\defeq \derived q_\bM(F \tens F'), \\
\rihom(\derived q_\bM F, \derived q_\bM F') &\defeq \derived q_\bM\rihom(F, F'), \\
\reeim f \derived q_\bM F &\defeq \derived q_\bN\roimv{q_{2\sep!!}}(\field_{\Gamma_f} \tens \opb{q_1} F), \\
\roim f \derived q_\bM F &\defeq \derived q_\bN\roimv{q_{2\sep*}}\rihom(\field_{\Gamma_f}, \epb{q_1} F), \\
\opb f \derived q_\bN G &\defeq \derived q_\bM\roimv{q_{1\sep!!}}(\field_{\Gamma_f} \tens \opb{q_2} G), \\
\epb f \derived q_\bN G &\defeq \derived q_\bM \roimv{q_{1\sep*}} \rihom(\field_{\Gamma_f}, \epb{q_2} G).
\end{align*}

\begin{remark}
The natural embedding
\[
\iota_\bM\colon \BDC(\field_\unbM) \to \derd(\bM)
\]
commutes with the operations $\tens$,
$\rihom$, $\opb f$, $\roim f$, $\epb f$.
If $f$ is semi-proper, one has
\begin{equation}
\label{eq:iotaeim}
\reeim f \circ \iota_\bM \isoto \iota_\bN\circ \reim{ \unb f}.
\end{equation}
\end{remark}

\begin{remark}
Let $\bM=(\oM,\cM)$.
For the natural morphism $j_\bM\colon \bM \to \cM$, one has
\[
\derived q_\bM \simeq \opb{j_\bM} \simeq \epb{j_\bM}, \quad
\derived l_\bM \simeq \roimv{j_{\oM\sep!!}}, \quad
\derived r_\bM \simeq \roimv{j_{\oM\sep*}}.
\]
\end{remark}

The following result generalizes \eqref{eq:lqrq}.

\begin{lemma}\label{lem:reimtensanddual}
Let $Z$ be a locally closed subset of\/ $\bM$, and let $F\in\derd(\bM)$.
Using {\rm Notation~\ref{not:bZ}}, one has
\begin{align*}
\field_{Z} \tens F &\simeq \roimv{i_{\bZ\sep!!}}\opb i_\bZ F, \\
\rihom(\field_Z, F) &\simeq \roimv{i_{\bZ\sep*}}\epb{i_\bZ} F.
\end{align*}
\end{lemma}

\begin{proof}
To avoid confusion, let us denote by $\field_{Z|\unbM}$ the extension by zero to
$\unbM$ of the constant sheaf $\field_Z$ on $Z$.
Since $i_\bZ$ is semi-proper, \eqref{eq:iotaeim} implies
$\field_{Z|\unbM}\simeq\roimv{i_{\bZ\sep!!}}\field_Z$. Hence
\begin{align*}
\field_{Z|\unbM} \tens F
&\simeq (\roimv{i_{\bZ\sep!!}}\field_Z) \tens F \\
&\simeq \roimv{i_{\bZ\sep!!}}(\field_Z \tens \opb i_\bZ F) \\
&\simeq \roimv{i_{\bZ\sep!!}}\opb i_\bZ F.
\end{align*}
We can prove the second isomorphism similarly.
\end{proof}

 Let $\bM=(\oM,\cM)$ be a bordered space.
By \cite[\S3.4]{DK13}, one has
\begin{align*}
\derd^{\leq 0}(\bM) &=\{ F \in \derd(\bM) \semicolon \roimv{j_{\bM\sep!!}} F \in \derd^{\leq 0}(\cM) \}, \\
\derd^{\geq 0}(\bM) &= \{ F \in \derd(\bM) \semicolon \roimv{j_{\bM\sep!!}} F \in \derd^{\geq 0}(\cM) \}.
\end{align*}

\begin{proposition}
\label{pro:t-board}
Let $\bM$ be a bordered space.
\begin{itemize}
\item[(i)]
The bifunctor $\tens$ is exact, i.e.\ for any $n,n'\in\Z$ one has
\begin{align*}
\derd^{\leq n}(\bM)\tens \derd^{\leq n'}(\bM) &\subset \derd^{\leq n+n'}(\bM), \\
\derd^{\geq n}(\bM)\tens\derd^{\geq n'}(\bM) &\subset \derd^{\geq n+n'}(\bM).
\end{align*}
\item[(ii)]
The bifunctor $\rihom$ is left exact, i.e.\ for any $n,n'\in\Z$ one has
\[
\rihom(\derd^{\leq n}(\bM), \derd^{\geq n'}(\bM)) \subset \derd^{\geq n'-n}(\bM).
\]
\end{itemize}
Let $f\colon \bM \to \bN$ be a morphism of bordered spaces.
\begin{itemize}
\item[(iii)]
$\reeim f$ and $\roim f$ are left exact, i.e.\ for any $n\in\Z$ one has
\begin{align*}
\reeim f \derd^{\geq n}(\bM) &\subset \derd^{\geq n}(\bN), \\
\roim f \derd^{\geq n}(\bM) &\subset \derd^{\geq n}(\bN).
\end{align*}
\item[(iv)]
$\opb f$ is exact, i.e.\ for any $n\in\Z$ one has
\begin{align*}
\opb f \derd^{\leq n}(\bN) &\subset \derd^{\leq n}(\bM), \\
\opb f \derd^{\geq n}(\bN) &\subset \derd^{\geq n}(\bM).
\end{align*}
\end{itemize}
Let $d\in\Z_{\geq 0}$ and assume that $\opb f(y)\subset \unbM$ has soft-dimension $\leq d$ for any $y\in \unbN$.
\begin{itemize}
\item[(v)]
$\reeim f(\ast)[d]$ is right exact, i.e., for any $n\in\Z$ one has
\[
\reeim f \derd^{\leq n}(\bM) \subset \derd^{\leq n+d}(\bN).
\]
\item[(vi)]
$\epb f(\ast)[-d]$ is left exact, i.e., for any $n\in\Z$ one has
\[
\epb f \derd^{\geq n}(\bN) \subset \derd^{\geq n-d}(\bM).
\]
\end{itemize}
\end{proposition}

\begin{proof}
When $\bM$ and $\bN$ are good topological spaces, statements
(i)--(iv) follow from \cite{KS01}.

Let $\bM=(\oM,\cM)$ and $\bN=(\oN,\cN)$.
Replacing $(\oM,\cM)$ with $(\oM,\olG)$, we may assume from the beginning that
$f\cl \oM\to\oN$ extends to $\cf\cl\cM\to\cN$.

\smallskip
\noi
(i) follows from the topological space case,
using the fact that $\roimv{j_{\oM\sep!!}}$ commutes with $\tens$.

\smallskip\noi
(ii) follows from (i) by adjunction.

\smallskip\noi
(iii) and (iv) follow from the topological space case
using the isomorphisms
\[
\reeim{f}\simeq\opb{j_\bN}\reeim{{\cf}}\reeim{{j_\bM}}, \quad
\roim{f}\simeq\opb{j_\bN}\roim{{\cf}}\roim{{j_\bM}}, \quad
\opb{f}\simeq\opb{{j_\bM}}\opb{\smash{\cf}}\reeim{{j_\bN}}.
\]

\smallskip
As (vi) follows from (v) by adjunction, we are left to prove (v).

\medskip
By d\'evissage, it is enough to show that for $F\in\ind(\field_\bM)$ one has
\[
H^k\reeim f F \simeq 0 \quad\text{for }k>d.
\]
Writing $F=\indlim[i]F_i$ with $F_i\in\Mod_c(\field_\bM)$, one has
\[
H^k \reeim f F
\simeq \indlim[i] H^k \reim {f} F_i.
\]
Then, for any $y\in N$,
\[
\bl H^k \reim { f} F_i \br_y
\simeq H^k_c({f}^{-1}(y); F_i\big\vert_{{f}^{-1}(y)}) \simeq 0,
\]
since $f^{-1}(y)$ has soft-dimension $\leq d$.
\end{proof}

\begin{proposition}\label{pro:opbVan}
Let $f\cl\bM\to \bN$ be a morphism of bordered spaces. Let $n\in\Z$ and $G\in\derd(\bN)$.
Assume
\begin{itemize}
\item[(a)]
$f$ is semi-proper,
\item[(b)]
$\unb f\colon \unbM \to \unbN$ is surjective.
\end{itemize}
Then
\begin{itemize}
\item[(i)]
$\opb f G\in \derd^{\geq n}(\bM)$ implies $G\in \derd^{\geq n}(\bN)$,
\item[(ii)]
$\opb f G\in \derd^{\leq n}(\bM)$ implies $G\in \derd^{\leq n}(\bN)$.
\end{itemize}
\end{proposition}

\begin{proof}
Let $\bM = (\oM,\cM)$ and $\bN = (\oN,\cN)$.
Since $\opb{f}$ is exact, it is enough to show that,
for $G\in \derd^{0}(\bN) \simeq \ind(\field_\bN)$,
$\opb f G \simeq 0$ implies $G \simeq 0$.

Write $G=\indlim G_i$, where $\{G_i\}_{i\in I}$ is a filtrant inductive system of objects $G_i\in\Mod_c(\field_\bN)$. Recall that this means that
$G_i\in\Mod(\field_\oN)$ and $\supp(G_i)$ is relatively compact in $\cN$.
Since $f$ is semi-proper, 
$\opb f G_i\in\Mod_c(\field_\bM)$ by Lemma~\ref{lem:proper} (ii).
The assumption $\opb f G=\indlim \opb f G_i\simeq0$
implies that, for any $i\in I$,
there exists $i \to j$ in $I$ whose induced morphism $\opb f G_i \to \opb f G_j$ is the zero map.
Since $f$ is surjective, $G_i\to G_j$ is the zero map. Thus $G=0$.
\end{proof}

\Prop\label{prop:patch}
Let $f\cl M\to N$ be a continuous map of good topological spaces,
and $\{ V_i\}_{i\in I}$ an open covering of $N$.
Let $K_i\in\derd(\opb{f}V_i)$ satisfy
$\roim{f}\rhom(K_i,K_i)\in\derd^{\ge0}(\cor_{V_i})$ and let
\[
u_{ij}\cl K_j\vert_{\opb{f}V_i\cap \opb{f}V_j}\isoto K_i\vert_{\opb{f}V_i\cap \opb{f}V_j}
\]
be isomorphisms satisfying the usual cochain condition{\rm:} $u_{ij}\circ u_{jk} = u_{ik}$ on $\opb{f}V_i\cap \opb{f}V_j\cap \opb{f}V_k$.
Then there exist $K\in\derd(M)$ and isomorphisms
$u_i\cl K\vert_{\opb{f}V_i}\isoto K_i$ compatible with $u_{ij}$, that is,
$u_{ij}\circ u_j=u_i$ on $\opb{f}V_i\cap \opb{f}V_j$.
Moreover, such a $K$ is unique up to a unique isomorphism.
\enprop

\begin{proof}
The arguments we use are standard (see e.g.\ \cite[Proposition 5.9]{Kas15}).
Let us set $U_i \defeq \opb{f}V_i \subset M$.

(i) Let us first discuss uniqueness. Let $K'\in\derd(M)$ be such that there are isomorphisms
$u'_i\cl K'\vert_{U_i}\isoto K_i$ compatible with $u_{ij}$.
Note that for any open subset $V$ of $N$, one has
\[
\Hom[\derd(f^{-1}V)](K\vert_{f^{-1}V},K'\vert_{f^{-1}V}) \simeq
H^0\rsect(V; \roim{f}\rhom(K,K')).
\]
Since one has $\roim{f}\rhom(K,K')\vert_{V_i} \simeq \roim{f}\rhom(K_i,K_i)\in\derd^{\ge0}(\cor_{V_i})$, we have
$\roim{f}\rhom(K,K')\in\derd^{\ge0}(\cor_{N})$.
Hence
\[
\text{$V\mapsto \Hom[\derd(f^{-1}V)](K\vert_{f^{-1}V},K'\vert_{f^{-1}V}) $ is a sheaf on $N$.}
\]
We thus get an isomorphism $K\isoto K'$ on $M$ by patching together the isomorphisms
$u_i^{\prime -1}\circ u_i$ on $U_i$.

\smallskip\noi
(ii) Let us now prove the existence of $K$ as in the statement.

\smallskip\noi
(ii-1) Assume that $I$ is finite. In order to prove the statement, by induction we reduce to the case $I=\{1,2\}$. Set $V_0\defeq V_1\cap V_2$ and $K_0 \defeq K_1\vert_{U_0}\simeq K_2\vert_{U_0}$.
Let $j_i\colon U_i\to M$ ($i=0,1,2$) be the open inclusion.
By adjunction, for $i=1,2$ there are natural morphisms
\[
\beta_i\colon \roimv{j_{0\sep!!}} K_0 \to \roimv{j_{i\sep!!}}K_i.
\]
Let us complete the morphism $(\beta_1,\beta_2)$ into a distinguished triangle
\[
\roimv{j_{0\sep!!}}K_0 \To[(\beta_1,\beta_2)] \roimv{j_{1\sep!!}}K_1 \dsum \roimv{j_{2\sep!!}}K_2 \To K\tone.
\]
Then $K$ satisfies the desired condition.

\smallskip\noi
(ii-2) Assume that $I=\Z_{\geq0}$ and that $\{V_n\}_{n\in\Z_{\geq0}}$ is an increasing
sequence of open subsets of $N$.
Then $K_{n+1}\vert_{U_n}\simeq K_n$.
Let $j_n\colon U_n\to M$ ($n\in\Z_{\geq 0}$) be the open inclusion.
By adjunction, there are natural morphisms $\beta_n\colon \roimv{j_{n\sep!!}} K_n \to \roimv{j_{n+1\sep!!}}K_{n+1}$ ($n\in\Z_{\geq 0}$).
Let $K$ be the homotopy colimit of the inductive system $\{\roimv{j_{n\sep!!}} K_n\}_{n\in\Z_{\ge0}}$, that is, let $K$ be the third term of the distinguished triangle
\[
\DSum_{n\in\Z_{\ge0}} \roimv{j_{n\sep!!}} K_n \To[\beta]
\DSum_{n\in\Z_{\ge0}} \roimv{j_{n\sep!!}} K_n \To
K \tone.
\]
Here $\beta$ is the only morphism making the following diagram commute for any $m\in\Z_{\geq 0}$
\[
\xymatrix{
\roimv{j_{m\sep!!}} K_m \ar[d] \ar[r]^-{(\id,-\beta_m)}
& \roimv{j_{m\sep!!}} K_m \dsum \roimv{j_{m+1\sep!!}} K_{m+1} \ar[d] \\
\DSum_{n\in\Z_{\ge0}} \roimv{j_{n\sep!!}} K_n \ar[r]^-\beta &
\DSum_{n\in\Z_{\ge0}} \roimv{j_{n\sep!!}} K_n.
}
\]
Then $K$ satisfies the desired condition.

\smallskip\noi
(ii-3) Let $I$ be arbitrary. Let
$\{Z_n\}_{n\in\Z_{\geq0}}$ be an increasing
sequence of compact subsets of $N$ such that
$N=\Union\nolimits_{n\in\Z_{\geq0}}Z_n$.
Let us take an increasing sequence $\{I_n\}_{n\in\Z_{\ge0}}$
of finite subsets of $I$
such that $Z_n$ is covered by $\{V_i\}_{i\in I_n}$,
and set $V'_n \defeq \Union\nolimits_{i\in I_n} V_i$, $U_n' \defeq \opb{f}V_n'$.
Applying (ii-1) with $N=V'_n$ and $I=I_n$,
we can find an object $K_n \in\derd(U'_n)$ such that $K_n\vert_{U_i}\simeq K_i$ for any $i\in I_n$. Then we can apply (ii-2) with $V_n=V_n'$.
\end{proof}

\subsection{Ind-sheaves with an extra variable}
Let $\cR\defeq\R\union\{+\infty,-\infty\}$ be the two-point compactification of the affine line.
The \emph{bordered line} is
\[
\bR \defeq (\oR,\cR).
\]
Let $\bM$ be a bordered space.
Consider the morphisms
\begin{equation}
\label{eq:q1q2mu}
\mu,q_1,q_2 \colon \bM\times\bR\times\bR \to \bM\times\bR,
\end{equation}
where $\mu(x,t_1,t_2) = (x,t_1+t_2)$, and $q_1,q_2$ are the natural projections.
The convolution functors
\begin{align*}
\ctens &\colon \derdR{\bM} \times \derdR{\bM} \to \derdR{\bM}, \\
\cihom &\colon \derdR{\bM}^\op \times \derdR{\bM} \to \derdR{\bM}
\end{align*}
are defined as follows, for $F_1,F_2\in\derdR{\bM}$,
\begin{align*}
F_1\ctens F_2 &\defeq \reeim \mu (\opb{q_1} F_1 \tens \opb{q_2} F_2), \\
\cihom(F_1,F_2)&\defeq \roimv{q_{1\sep*}} \rihom(\opb{q_2} F_1, \epb\mu F_2).
\end{align*}

\begin{example}
Let $\bM=\point$ and let $a,b\in\R$.
\begin{itemize}
\item[(i)]
For $a\leq b$, one has
\begin{align*}
\field_{\{t\geq 0\}} \ctens \field_{\{t \geq a\}} &\simeq \field_{\{t\geq a\}}, &
\field_{\{t\geq 0\}} \ctens \field_{\{a\leq t < b\}} &\simeq \field_{\{a\leq t < b\}}, \\
\cihom(\field_{\{t\geq 0\}}, \field_{\{t\geq a\}}) &\simeq \field_{\{t<a\}}[1], &
\cihom(\field_{\{t\geq 0\}}, \field_{\{a\leq t < b\}}) &\simeq \field_{\{a\leq t < b\}}.
\end{align*}
\item[(ii)]
For $0<a\leq b$, one has
\[
\field_{\{0\leq t <a\}} \ctens \field_{\{0\leq t <b\}}
\simeq \field_{\{0\leq t <a\}} \dsum \field_{\{b\leq t <a+b\}}[-1].
\]
\end{itemize}
\end{example}

Consider the standard classical t-structure
\[
\bl \derdR[\leq 0]{\bM}, \derdR[\geq 0]{\bM} \br
\]
on $\derdR{\bM}$ discussed in \S\ref{sse:Ibord}.

\begin{lemma}\label{lem:tC-ctenscihom}
Let $\bM$ be a bordered space.
\begin{itemize}
\item[(i)]
For $n,n'\in\Z$ one has
\begin{align*}
\derdR[\leq n]{\bM}\ctens \derdR[\leq n']{\bM} &\subset \derdR[\leq n+n'+1]{\bM}, \\
\derdR[\geq n]{\bM}\ctens \derdR[\geq n']{\bM} &\subset \derdR[\geq n+n']{\bM}.
\end{align*}
In particular, the bifunctor $\ctens$ is left exact.
\item[(ii)]
For $n,n'\in\Z$ one has
\[
\cihom(\derdR[\leq n]{\bM}, \derdR[\geq n']{\bM}) \subset \derdR[\geq n'-n-1]{\bM}.
\]
\end{itemize}
\end{lemma}

\begin{proof}
Recall the maps \eqref{eq:q1q2mu}.

(i) By the definition, for $F_1,F_2\in\derdR{\bM}$ one has
\[
F_1 \ctens F_2 \defeq \reeim\mu(\opb{q_1}F_1 \tens \opb{q_2}F_2).
\]
Then the statement follows from Proposition~\ref{pro:t-board}.

(ii) The proof is similar, recalling that
\[
\cihom(F_1, F_2) \defeq \roimv{q_{1\sep*}}\rihom(\opb{q_2}F_1, \epb\mu F_2).
\]
\end{proof}

\begin{remark}\label{rem:cihomderdR}
There are no estimates of the form
\[
\cihom(\field_{\{t\geq 0\}}, \derdR[0]{\bM}) \subset \derdR[\leq m]{\bM}
\]
with $m\in\Z_{\geq 0}$ independent of $\bM$.
In fact, setting,
\[
\bM=\R^n\ (n\geq 1), \quad F = \field_{\{x\neq 0,\ t = 1/|x| \}},
\]
one has
\begin{equation}
\label{eq:remcihom}
\cihom(\field_{\{t\geq 0\}}, F) \notin \derdR[\leq n-3]{\bM},
\end{equation}
which follows from
\eqn
&&\opb\pi\field_{\{x=0\}} \tens \cihom(\field_{\{t\geq 0\}}, F)
\simeq \opb\pi \field_{\{x=0\}}[1] \dsum \opb\pi\field_{\{x=0\}}[2-n].
\eneqn
\end{remark}

\begin{lemma}
\label{lem:tgeq0Hk}
For $K\in\derdR{\bM}$ and $n\in\Z$ one has
\begin{align*}
\field_{\{t\geq 0\}} \ctens \tau^{\leq n} (\field_{\{t\geq 0\}} \ctens K) &\isoto \tau^{\leq n} (\field_{\{t\geq 0\}} \ctens K), \\
\field_{\{t\geq 0\}} \ctens \tau^{\geq n} (\field_{\{t\geq 0\}} \ctens K) &\isoto \tau^{\geq n} (\field_{\{t\geq 0\}} \ctens K), \\
\end{align*}
\end{lemma}

Let us give a proof of this result slightly different from that in
\cite[Proposition 4.6.2]{DK13}.

\begin{proof}
Consider the distinguished triangle
\begin{align*}
\field_{\{t> 0\}} \ctens \tau^{\leq n} (\field_{\{t\geq 0\}} \ctens K) &\to
\field_{\{t> 0\}} \ctens (\field_{\{t\geq 0\}} \ctens K) \\
&\to \field_{\{t> 0\}} \ctens \tau^{> n} (\field_{\{t\geq 0\}} \ctens K) \tone.
\end{align*}
Since the middle term vanishes, one has
\[
\field_{\{t> 0\}} \ctens \tau^{> n} (\field_{\{t\geq 0\}} \ctens K) \simeq
\field_{\{t> 0\}} \ctens \tau^{\leq n} (\field_{\{t\geq 0\}} \ctens K)[1].
\]
By Lemma~\ref{lem:tC-ctenscihom}, the first term belongs to $\derdR[> n]{\bM}$ and the second term belongs to $\derdR[\leq n]{\bM}$.
Hence they both vanish.
\end{proof}

\subsection{Enhanced ind-sheaves}
Let $\bM$ be a bordered space, and consider the natural
morphisms
\[
\xymatrix@C=5ex{
\bM & \bM\times\bR \ar[l]_-\pi \ar[r]^-{j} & \bM\times\cR \ar[r]^-{\overline\pi} & \bM.
}
\]
Consider the full subcategories of $\derdR{\bM}$
\begin{align*}
\catn_\pm
&\defeq \{K\in \derdR{\bM}\semicolon \field_{\{\mp t\geq 0\}} \ctens K \simeq 0\} \\
&= \{K\in\derdR{\bM} \semicolon \cihom(\field_{\{\mp t\geq 0\}}, K) \simeq 0\}, \\
\catn
&\defeq \catn_+ \cap \catn_- \\
&= \opb \pi\derd(\bM),
\end{align*}
where the equalities hold by \cite[Corollary 4.3.11 and Lemma 4.4.3]{DK13}.

The categories of enhanced ind-sheaves are defined by
\[
\BECpm \bM \defeq \derdR{\bM}/\catn_\mp, \quad
\BEC \bM \defeq \derdR{\bM}/\catn.
\]
In this paper, we set for short
\[
\dere_\pm(\bM) \defeq \BECpm \bM, \quad \dere(\bM) \defeq \BEC \bM.
\]

By \cite[Proposition 4.4.4]{DK13}, there are natural equivalences
\begin{align*}
\dere_\pm(\bM)
&\simeq \catn_\pm/\catn
\simeq {}^\bot\catn_\mp
= \catn_\pm \cap {}^\bot\catn, \\
\dere(\bM)
&\simeq {}^\bot\catn
\simeq \dere_+(\bM) \dsum \dere_-(\bM),
\end{align*}
and the same equivalences hold when replacing left with right orthogonals.
Moreover, one has
\begin{align*}
{}^\bot\catn_\mp
&= \{K \in \derdR{\bM} \semicolon \field_{\{\pm t\geq 0\}} \ctens K \isoto K\}, \\
{}^\bot\catn
&= \{K\in \derdR{\bM} \semicolon (\field_{\{t\geq 0\}} \dsum \field_{\{t\leq 0\}}) \ctens K \isoto K\} \\
&= \{K\in \derdR{\bM} \semicolon \reeim\pi K \simeq 0\},
\end{align*}
and the same equalities hold for right orthogonals, replacing $\ctens$ with $\cihom$ and $\reeim\pi$ with $\roim\pi$.

We use the following notations
\[
\xymatrix@C3.5em{
\derdR{\bM} \ar@<.5ex>[r]^-{\quot_\bM} & \dere(\bM) \ar@<.5ex>[l]^-{\LE,\;\RE},
}
\quad
\xymatrix@C3.5em{
\derdR{\bM} \ar@<.5ex>[r]^-{\quot_\bM^\pm} & \dere_\pm(\bM) \ar@<.5ex>[l]%
^-{\LEpm,\;\REpm},
}
\]
for the quotient functors and their left and right adjoints, respectively.
For $F\in\derdR{\bM}$ one has
\begin{align*}
\LE(\quot_\bM F) &\simeq \bl \field_{\{t\geq 0\}} \dsum \field_{\{t\leq 0\}} \br \ctens F, \\
\RE(\quot_\bM F) &\simeq \cihom \bl
\field_{\{t\geq 0\}} \dsum \field_{\{t\leq 0\}},\; F \br.
\end{align*}

For a locally closed subset $Z\subset \oM\times\oR$, we set
\begin{equation}
\label{eq:Qfield}
\Qfield_Z \defeq \quot_\bM (\field_Z) \in \dere(\bM).
\end{equation}

There are functors
\begin{align}
\label{eq:epsilon}
\epsilon&\colon \derd(\bM) \To \dere(\bM), &
F&\mapsto \Qfield_{\{t = 0\}} \tens \opb\pi F, \\
\notag
\epsilon_\pm&\colon \derd(\bM) \To \dere_\pm(\bM), &
F&\mapsto \Qfield_{\{\pm t \geq 0\}} \tens \opb\pi F.
\end{align}
The functors $\epsilon_\pm$ are fully faithful and $\epsilon(F)\simeq\epsilon_+(F)\dsum\epsilon_-(F)$.

The bifunctors
\begin{align*}
\fihom &\colon \dere(\bM) \times \dere(\bM) \to \derd(\bM)%
, \\
\fhom &\colon \dere(\bM)^\op \times \dere(\bM) \to \BDC(\field_{\unb\bM}),
\end{align*}
are defined by
\begin{align*}
\fihom(K,K')
&\defeq \roim\pi\rihom(\LE K,\LE K') \\
&\simeq \roim\pi\rihom(\LE K,\RE K') \\
&\simeq \roim\pi\rihom(\RE K,\RE K') \\
&\simeq \roim{\overline\pi}\rihom(\reeim j \LE K,\roim j \RE K')\qtq\\
\fhom &\defeq \alpha_\bM \circ \fihom.
\end{align*}
One has
\begin{align}
\Hom[\dere(\bM)](K,K')
&\simeq \Hom[\derd(\bM)](\cor_M, \fihom(K,K')).\label{eq:Hfhom}
\end{align}
If $\bM$ is a topological space, that is, if
$\unbM\to \bM$ is an isomorphism,
one has
\[
\Hom[\dere(\bM)](K,K')
\simeq H^0\rsect\bl \unbM;\fhom(K,K')\br.
\]
Note, however, that $\Hom[\dere(\bM)](K,K')
\simeq H^0\rsect\bl \bM;\fhom(K,K')\br$ does not hold in general.

\begin{definition}[{\cite[Definition 4.6.3]{DK13}}]
For $n\in\Z$, set
\begin{align*}
\dere^{\leq n}(\bM) &\defeq \{ K\in\dere(\bM)\semicolon \LE K\in \derdR[\leq n]{\bM} \}, \\
\dere^{\geq n}(\bM) &\defeq \{ K\in\dere(\bM)\semicolon \LE K\in \derdR[\geq n]{\bM} \}.
\end{align*}
\end{definition}

Note that
\eqn
\dere^0(\bM) &\simeq& \{ F\in\ind(\field_{\bM\times\bR}) \semicolon
\text{$\bl\field_{\{t\geq 0\}} \dsum \field_{\{t\leq 0\}}\br \ctens F \isoto F$
in $\derdR{\bM}$} \}\\
&=& \{ F\in\ind(\field_{\bM\times\bR}) \semicolon
\reeim{\pi}F\simeq0 \text{ in $\derd(\bM)$}\}.
\eneqn

\begin{proposition}[{\cite[Proposition 4.6.2]{DK13}}]
$\bl \dere^{\leq 0}(\bM), \dere^{\geq 0}(\bM) \br$ is a classical $t$-structure on $\dere(\bM)$.
\end{proposition}

\begin{example}\label{ex:LERE}
Let $a,b\in\R$ with $a<b$.
In the category $\dere(\point)$, one has
\begin{align*}
\LE \Qfield_{\{a\leq t\}} \simeq \field_{\{a\leq t\}}, &\quad
\LE \Qfield_{\{a\leq t < b\}} \simeq \field_{\{a\leq t < b\}}, \\
\RE \Qfield_{\{a\leq t\}} \simeq \field_{\{t<a\}}[1], &\quad
\RE \Qfield_{\{a\leq t < b\}} \simeq \field_{\{a\leq t < b\}}.
\end{align*}
In particular,
\[
\Qfield_{\{a\leq t\}},\ \Qfield_{\{a\leq t < b\}} \in \dere^{0}(\point).
\]
\end{example}

\begin{proposition}\label{prop:Estack}
Let $M$ be a good topological space. Then the prestack on $M$
given by $U \mapsto \dere^0(U)$ is a stack of abelian categories.
\end{proposition}

\begin{proof}
The statement holds since $U \mapsto \dere^0(U)$ is a sub-prestack
of the direct image by $\pi$ of the stack of ind-sheaves on $M\times\bR$.
More precisely, one has
\[
\dere^0(U) \simeq \{F\in\ind(\field_{U\times\bR}) \semicolon
(\field_{\{t\geq 0\}}\oplus\field_{\{t\leq 0\}}) \ctens F \isoto F \}.
\]
\end{proof}

\begin{lemma}\label{lem:dereQ}
For any $n\in\Z$ one has
\begin{align*}
\quot_\bM \derdR[\leq n]{\bM} &\subset \dere^{\leq n+1}(\bM), \\
\quot_\bM \derdR[\geq n]{\bM} &= \dere^{\geq n}(\bM).
\end{align*}
In particular, $\quot_\bM$ is left exact.
\end{lemma}

\begin{proof}
(i) For $F\in\derdR{\bM}$, one has
$\LE \quot_\bM F \simeq (\field_{\{t\geq 0\}}\oplus\field_{\{t\leq 0\}})\ctens F$.
Hence the inclusions ``$\subset$'' follow from Lemma~\ref{lem:tC-ctenscihom}.

\smallskip\noi
(ii)
It remains to show the opposite inclusion
 $\quot_\bM \derdR[\geq n]{\bM}\supset\dere^{\geq n}(\bM)$.
If $K\in\dere^{\geq n}(\bM)$, then $F \defeq \LE K \in \derdR[\geq n]{\bM}$,
and $K\simeq\quot_\bM(F)$.
\end{proof}

\begin{lemma}\label{lem:t-RE}
For any $n\in\Z$ one has
\[
\RE \dere^{\geq n}(\bM) \subset \derdR[\geq n-1]{\bM}.
\]
\end{lemma}

\begin{proof}
By Lemma~\ref{lem:dereQ}, the functor $\quot_\bM[1]$ is right exact.
Hence its right adjoint $\RE[-1]$ is left exact.
\end{proof}

\begin{remark}
\bnum
\item
It follows from Example~\ref{ex:LERE} that the estimate in Lemma~\ref{lem:t-RE} is optimal.
\item
It follows from Remark~\ref{rem:cihomderdR} that there are no estimates of the form
\[
\RE\dere^0(\bM) \subset \derdR[\leq m]{\bM}
\]
with $m\in\Z$ independent of $\bM$.
\item
The example in Remark~\ref{rem:cihomderdR} shows that
\begin{multline*}
\bl \{K\in\dere(\bM) \semicolon \RE K \in \derdR[\leq0]{\bM}\},\\
\{K\in\dere(\bM) \semicolon \RE K \in \derdR[\geq0]{\bM}\} \br
\end{multline*}
is not a classical t-structure on $\dere(\bM)$, in general.
\ee
\end{remark}

\begin{proposition}\label{pro:t-homE}
The functors $\fihom$ and $\fhom$ are left exact, i.e.\ for $n,n'\in\Z$ one has
\begin{itemize}
\item[(i)]
$\fihom(\dere^{\leq n}(\bM),\dere^{\geq n'}(\bM)) \subset \derd^{\geq n'-n}(\bM)$,
\item[(ii)]
$\fhom(\dere^{\leq n}(\bM),\dere^{\geq n'}(\bM)) \subset \derd^{\geq n'-n}(\field_\unb\bM)$.
\end{itemize}
\end{proposition}

\begin{proof}
(i) By the definition, for $K,K'\in\dere(\bM)$ one has
\[
\fihom(K, K') = \roim\pi\rihom(\LE K,\LE K').
\]
Hence the statement follows from Proposition~\ref{pro:t-board}.

(ii) One has $\fhom = \alpha_\bM\, \fihom$. Since $\alpha_\bM$ is exact, the statement follows from (i).
\end{proof}

\subsection{Operations}
Let $f\colon\bM\to\bN$ be a morphism of bordered spaces.
The six Grothendieck operations for enhanced ind-sheaves
\begin{align*}
\ctens &\colon \dere(\bM) \times \dere(\bM) \to \dere(\bM), \\
\cihom &\colon \dere(\bM)^\op \times \dere(\bM) \to \dere(\bM), \\
\Eeeim f, \Eoim f &\colon \dere(\bM) \to \dere(\bN), \\
\Eopb f, \Eepb f &\colon \dere(\bN) \to \dere(\bM)
\end{align*}
are defined as follows.
Set $f_\bR = f\times\id_\bR\cl \bM\times\bR\to\bN\times\bR$.
For $F,F'\in\derdR{\bM}$ and $G\in\derdR{\bN}$, one sets
\begin{align*}
\quot_\bM F\ctens \quot_\bM F' &\defeq \quot_\bM(F \ctens F'), \\
\cihom(\quot_\bM F, \quot_\bM F') &\defeq \quot_\bM\cihom(F, F'), \\
\Eeeim f \quot_\bM F &\defeq \quot_\bN\roimv{f_{\R_\infty\sep!!}}F, \\
\Eoim f \quot_\bM F &\defeq \quot_\bN\roimv{f_{\R_\infty\sep*}}F, \\
\Eopb f \quot_\bN G &\defeq \quot_\bM\opb{f_\bR}G, \\
\Eepb f \quot_\bN G &\defeq \quot_\bM\epb{f_\bR}G.
\end{align*}

The duality functor is defined by
\[
\Qdual_\bM\colon\dere(\bM)\to \dere(\bM)^\op, \qquad K \mapsto \cihom(K,\omega^\quot_\bM),
\]
where
$\omega_\bM \seteq \epb {j_\bM}\omega_\cM\in\derd(\bM)$
and $\omega^\quot_\bM \seteq\epsilon(\omega_\bM)
\seteq\opb\pi\omega_\bM \tens \Qfield_{\{t=0\}}\in \dere(\bM)$.

\begin{lemma}[{\cite[Lemma 4.3.2]{DK13}}]
\label{lem:Qdual}
Let $\bM=(\oM,\cM)$.
For $F\in\derd(\cor_{M\times\R})$, one has
\[
\Qdual_\bM(\quot_\bM F) \simeq \quot_\bM(\opb a \dual_{\oM\times\oR} F),
\]
where $a$ is the involution of $M\times\R$ defined by $a(x,t) = (x,-t)$.
\end{lemma}

\begin{example}\label{ex:abEdual}
Let $a,b\in\R$ with $a<b$.
In the category $\dere(\point)$, one has
\[
\Qdual \Qfield_{\{a\leq t\}} \simeq \Qfield_{\{t<-a\}}[1] \simeq \Qfield_{\{-a\leq t\}}
\qtq
\Qdual \Qfield_{\{a\leq t < b\}} \simeq \Qfield_{\{-b\leq t < -a\}}[1].
\]
In particular,
\[
\Qdual\Qfield_{\{a\leq t\}} \in \dere^0(\point)\qtq
\Qdual\Qfield_{\{a\leq t < b\}} \in \dere^{-1}(\point).
\]
\end{example}

\begin{proposition}\label{pro:t-Eop}
Let $\bM$ be a bordered space.
\begin{itemize}
\item[(i)]
For $n,n'\in\Z$ one has
\begin{align*}
\dere^{\leq n}(\bM) \ctens \dere^{\leq n'}(\bM) &\subset \dere^{\leq n+n'+1}(\bM), \\
\dere^{\geq n}(\bM) \ctens \dere^{\geq n'}(\bM) &\subset \dere^{\geq n+n'}(\bM).
\end{align*}
In particular, the bifunctor $\ctens$ is left exact.
\item[(ii)]
For $n,n'\in\Z$ one has
\[
\cihom(\dere^{\leq n}(\bM), \dere^{\geq n'}(\bM)) \subset \dere^{\geq n'-n-1}(\bM).
\]
\end{itemize}
Let $f\colon \bM \to \bN$ be a morphism of bordered spaces.
\begin{itemize}
\item[(iii)]
$\Eeeim f$ and $\Eoim f$ are left exact, i.e.\ for any $n\in\Z$ one has
\begin{align*}
\Eeeim f \dere^{\geq n}(\bM) &\subset \dere^{\geq n}(\bN), \\
\Eoim f \dere^{\geq n}(\bM) &\subset \dere^{\geq n}(\bN).
\end{align*}
\item[(iv)]
$\Eopb f$ is exact, i.e.\ for any $n\in\Z$ one has
\begin{align*}
\Eopb f \dere^{\leq n}(\bN) &\subset \dere^{\leq n}(\bM), \\
\Eopb f \dere^{\geq n}(\bN) &\subset \dere^{\geq n}(\bM).
\end{align*}
\end{itemize}
Let $d\in\Z_{\geq 0}$ and assume that $\opb f(y)\subset \unbM$ has soft-dimension $\leq d$ for any $y\in\unbN$.
\begin{itemize}
\item[(v)]
$\Eeeim f(\ast)[d]$ is right exact, i.e.\ for any $n\in\Z$ one has
\[
\Eeeim f \dere^{\leq n}(\bM) \subset \dere^{\leq n+d}(\bN).
\]
\item[(vi)]
$\Eepb f(\ast)[-d]$ is left exact, i.e.\ for any $n\in\Z$ one has
\[
\Eepb f \dere^{\geq n}(\bN) \subset \dere^{\geq n-d}(\bM).
\]
\end{itemize}
\end{proposition}

\begin{proof}
(i) For $K\in\dere(\bM)$ and $K'\in\dere(\bM)$ one has
\[
\LE(K\ctens K') \simeq \LE K \ctens \LE K'.
\]
Then the statement follows from Lemma~\ref{lem:tC-ctenscihom}.

\smallskip\noi
(ii) follows from (i) by adjunction. As we deal here with bifunctors, let us spell out the proof.
Let $K\in\dere^{\leq n}(\bM)$, $K'\in\dere^{\geq n'}(\bM)$, and $L\in\dere^{< n'-n-1}(\bM)$.
Then one has
\begin{align*}
\Hom[\dere(\bM)](L,\cihom(K,K'))
&\simeq \Hom[\dere(\bM)](L \ctens K ,K') \\
&\in \Hom[\dere(\bM)](\dere^{< n'}(\bM), \dere^{\geq n'}(\bM)) =0.
\end{align*}
Then $\cihom(K,K') \in \dere^{< n'-n-1}(\bN)^\bot = \dere^{\geq n'-n-1}(\bN)$.

\smallskip\noi
(iii-1) The fact that $\Eeeim f$ is left exact follows from Proposition~\ref{pro:t-board}, since one has
\[
\LE \comp \Eeeim f \simeq \roimv{f_{\R_\infty\sep!!}} \comp \LE,
\]
where we recall that $f_\bR \defeq f \times \id_\bR$.

\smallskip\noi
(iv) also follows from Proposition~\ref{pro:t-board}, since one has
\[
\LE \comp \Eopb f \simeq \opb{f_\bR} \comp \LE.
\]

\smallskip\noi
(iii-2) The fact that $\Eoim f$ is left exact follows from (iv) by adjunction.

\noi
(v) has a proof similar to (iii-1).

\noi
(vi) follows from (v) by adjunction.
\end{proof}

\begin{proposition}\label{pro:EopbVan}
Let $f\cl\bM\to \bN$ be a morphism of bordered spaces.
Let $n\in\Z$ and $L\in\dere(\bN)$.
Assume
\begin{itemize}
\item[(a)]
$f$ is semi-proper,
\item[(b)]
$\unb f\colon \unbM \to \unbN$ is surjective.
\end{itemize}
Then
\begin{itemize}
\item[(i)]
$\opb f L\in \dere^{\geq n}(\bM)$ implies $L\in \dere^{\geq n}(\bN)$,
\item[(ii)]
$\opb f L\in \dere^{\leq n}(\bM)$ implies $L\in \dere^{\leq n}(\bN)$.
\end{itemize}
\end{proposition}

\begin{proof}
It is enough to apply Proposition~\ref{pro:opbVan} to the morphism $f_{\R_\infty}\colon \bM\times\bR \to \bN\times\bR$ and the object $G=\LE L \in\derd(\bN\times\bR)$.
\end{proof}

The bifunctors
\begin{align*}
\opb\pi(\ast)\tens(\ast) &\colon \derd(\bM) \times \dere(\bM) \to \dere(\bM), \\
\rihom(\opb\pi(\ast),\ast) &\colon \derd(\bM)^\op \times \dere(\bM) \to \dere(\bM)
\end{align*}
are defined as follows, for $L\in\derd(\bM)$ and $F\in\derdR{\bM}$,
\begin{align*}
\opb\pi L\ctens \quot_\bM F &\defeq \quot_\bM(\opb\pi L \tens F), \\
\rihom(\opb\pi L, \quot_\bM F) &\defeq \quot_\bM\rihom(\opb\pi L, F).
\end{align*}

\begin{lemma}\label{lem:Etpitens}\label{lem:Etpiihom}
Let $\bM$ be a bordered space.
\begin{itemize}
\item[(i)]
The bifunctor $\opb\pi(\ast) \tens (\ast)$ is exact, i.e.\ for $n,n'\in\Z$ one has
\begin{align*}
\opb\pi\derd^{\leq n}(\bM) \tens \dere^{\leq n'}(\bM) &\subset \dere^{\leq n+n'}(\bM), \\
\opb\pi\derd^{\geq n}(\bM) \tens \dere^{\geq n'}(\bM) &\subset \dere^{\geq n+n'}(\bM).
\end{align*}
In particular, the functor $\epsilon$ from \eqref{eq:epsilon} is exact.
\item[(ii)]
The bifunctor $\rihom(\opb\pi(\ast), \ast)$ is left exact, i.e.\ for $n,n'\in\Z$ one has
\[
\rihom(\opb\pi\derd^{\leq n}(\bM), \dere^{\geq n'}(\bM)) \subset \dere^{\geq n'-n}(\bM).
\]
\end{itemize}
\end{lemma}

\begin{proof}
(i) For $F\in \derd(\bM)$ and $K\in\dere(\bM)$ one has
\[
\LE(\opb\pi F \tens K) \simeq \opb\pi F \tens \LE K.
\]
Hence the statement follows from Proposition~\ref{pro:t-board}.

(ii) follows by adjunction from (i).
\end{proof}

Let us end this section stating some facts related to Notation~\ref{not:bZ}.

\begin{lemma}\label{lem:Eeimtensanddual}
Let $Z$ be a locally closed subset of $\bM$, and $K\in\dere(\bM)$.
One has
\begin{align*}
\opb\pi\field_{Z} \tens K &\simeq \Eeeim{i_\bZ}\Eopb i_\bZ K, \\
\rihom( \opb\pi\field_Z, K) &\simeq \Eoim{i_\bZ}\Eepb i_\bZ K.
\end{align*}
\end{lemma}

\begin{proof}
Note that $\inbordered{(Z\times\R)}=\bZ\times\bR$ and
$i_\bZ \times \id_\bR = i_{\inbordered{(Z\times\R)}}$.
Hence the statement follows from Lemma~\ref{lem:reimtensanddual}.
\end{proof}

\begin{lemma}\label{lem:EiZZ'}
Let $Z$ be a locally closed subset of $\bM$, and $Z'\subset Z$ a closed subset.
For $K\in\dere(\bM)$, there are distinguished triangles in $\dere(\bZ)$
\begin{align*}
\Eeeim i \Eopb i_{\inbordered{(Z\setminus Z')}} K \to
\Eopb i_\bZ K \to \Eeeim i' \Eopb i_{\inbordered{Z'}} K \tone, \\
\Eoim i' \Eepb i_{\inbordered{Z'}} K \to
\Eepb i_\bZ K \to \Eoim i \Eepb i_{\inbordered{(Z\setminus Z')}} K \tone,
\end{align*}
where $i\colon \inbordered{(Z\setminus Z')} \to \bZ$ and $i'\colon \inbordered{Z'} \to \bZ$ are the natural morphisms.
\end{lemma}

\begin{proof}
Since the proofs are similar, we shall only construct the first distinguished triangle.
By Lemma~\ref{lem:Eeimtensanddual}, applying the functor $\opb\pi(\ast)\tens K$ to the distinguished triangle
\[
\field_{Z\setminus Z'} \to \field_Z \to \field_{Z'} \tone,
\]
one gets the distinguished triangle
\[
\Eeeim{i_{\inbordered{(Z\setminus Z')}}} \Eopb i_{\inbordered{(Z\setminus Z')}} K \to
\Eeeim{i_\bZ} \Eopb i_\bZ K \to \Eeeim{i_{\inbordered{Z'}}}\Eopb i_{\inbordered{Z'}} K \tone.
\]
Since $i_{\inbordered{Z'}} = i_\bZ \circ i'$ and $i_{\inbordered{(Z\setminus Z')}} = i_\bZ \circ i$, the distinguished triangle in the statement is obtained by applying the functor $\Eopb i_\bZ$ to the above distinguished triangle.
\end{proof}

\begin{lemma}\label{lem:ibZfieldoZ}
Let $c\in\R$ and $Z$ a locally closed subset of $\bM$.
\bnum
\item
The following conditions are equivalent:
\bna
\item
$\Eopb i_\bZ K \in \dere^{\leq c}(\bZ)$,
\item
$\opb\pi\field_{\oZ} \tens K \in \dere^{\leq c}(\bM)$.
\ee
\item
The following conditions are equivalent:
\bna
\item
$\Eepb i_\bZ K \in \dere^{\geq c}(\bZ)$,
\item
$\rihom( \opb\pi\field_\oZ, K) \in \dere^{\geq c}(\bM)$.
\ee
\ee
\end{lemma}

\begin{proof}
(i) By Lemma~\ref{lem:Eeimtensanddual}, one has
\begin{align*}
\opb\pi\field_{\oZ} \tens K &\simeq \Eeeim{i_\bZ}\Eopb i_\bZ K, \\
\Eopb i_\bZ K &\simeq \Eopb i_\bZ(\opb\pi\field_{\oZ} \tens K).
\end{align*}
The statement follows, since the functors $\Eeeim{i_\bZ}$ and $\Eopb i_\bZ$ are exact
by Proposition~\ref{pro:t-Eop}.
(It follows that (a) and (b) remain equivalent when replacing $\leq c$ by $\geq c$.)

\smallskip\noi
(ii) is proved similarly.
\end{proof}

\subsection{Stable objects}\label{sse:stable}

Setting
\begin{align*}
\field_{\{t\gg0\}} &\defeq \indlim[a\rightarrow+\infty] \field_{\{t\geq a\}}, &
\field_{\{t<\ast\}} &\defeq \indlim[a\rightarrow+\infty] \field_{\{t< a\}}, \\
\field_{\{0\leq t<\ast\}} &\defeq \indlim[a\rightarrow+\infty] \field_{\{0\leq t< a\}},
\end{align*}
there are distinguished triangles in $\derd(\bM\times\R_\infty)$
\begin{align*}
&\field_{\{t\gg0\}} \to \field_{\{t<\ast\}}[1] \to \field_{M\times\R}[1] \tone, \\
&\field_{\{0\leq t<\ast\}} \to \field_{\{t\geq 0\}} \to \field_{\{t\gg0\}} \tone.
\end{align*}
The objects of $\dere(\bM)$
\begin{align*}
\Efield_\bM&\defeq \quot_\bM(\field_{\{t\gg0\}}) \simeq \quot_\bM(\field_{\{t<\ast\}}[1])\qtq \\
\field_\bM^\tor &\defeq \quot_\bM(\field_{\{0\leq t<\ast\}})
\end{align*}
enter the distinguished triangle
\begin{equation}
\label{eq:tor}
\field_\bM^\tor \to \field_{\{t\geq 0\}} \to \field_\bM^\enh \tone.
\end{equation}
Note that we have
$$\field_\bM^\tor \ctens \field_\bM^\tor \simeq \field_\bM^\tor ,\quad
\field_\bM^\enh\ctens\field_\bM^\enh\simeq\field_\bM^\enh
\qtq\field_\bM^\tor \ctens \field_\bM^\enh\simeq0.$$

\begin{definition}
The category $\dere_\st(\bM)$ of stable enhanced ind-sheaves is the full subcategory of $\dere_+(\bM)$ given by
\begin{align*}
\dere_\st(\bM)
&\defeq \{K\in\dere_+(\bM) \semicolon \field_\bM^\tor\ctens K \simeq 0 \} \\
&= \{K\in\dere_+(\bM) \semicolon K\isoto \field_\bM^\enh\ctens K \} \\
&= \{K\in\dere_+(\bM) \semicolon
\text{$K\simeq \field_\bM^\enh\ctens L$ for some $L\in \BECp M$} \} \\
&= \{K\in\dere_+(\bM) \semicolon
\text{$K\isoto \Qfield_{\{t\geq a\}}\ctens K$ for any $a\geq 0 $}\},
\end{align*}
where the equivalences follow from \eqref{eq:tor} and \cite[Proposition 4.7.5]{DK13}. Similar equivalences hold by replacing $\ctens$ with $\cihom$.
\end{definition}

The embedding $\dere_\st(\bM) \to \dere(\bM)$ has a left adjoint $\field_\bM^\enh\ctens\ast$, and a right adjoint $\cihom(\field_\bM^\enh,\ast)$.
There is an embedding
\begin{equation}
\label{eq:e}
e\colon \derd(\bM) \hookrightarrow \dere_\st(\bM), \quad F\mapsto \Efield_\bM \tens \opb\pi F.
\end{equation}
Note that $e(F) \simeq \Efield_\bM \ctens \epsilon(F)$.

For a locally closed subset $Z\subset \unbM\times\oR$, we set
\begin{equation}
\label{eq:Efield}
\Efield_Z \defeq \Efield_\bM \ctens \Qfield_Z \in \dere_\st(\bM).
\end{equation}

\begin{lemma}
\label{lem:t-enh}\hfill
\begin{itemize}
\item[(i)]
The embedding $e$ from \eqref{eq:e} is fully faithful and exact.
\item[(ii)]
The functor $\Efield_\bM \ctens (\ast)$ is exact.
\end{itemize}
\end{lemma}

\begin{proof}
(i) follows from \cite[Proposition 4.7.15]{DK13}
and Lemma~\ref{lem:Etpitens}, and (ii) from \cite[Lemma 4.7.4]{DK13}.
\end{proof}

The duality functor for stable enhanced ind-sheaves is defined by
\[
\Edual_\bM\colon\dere(\bM)\to \dere_\st(\bM)^\op, \qquad K \mapsto \cihom(K,\omega^\enh_\bM),
\]
where we set $\omega^\enh_\bM \defeq e(\omega_\bM)$.

\begin{lemma}[{\cite[Proposition~4.8.3]{DK13}}]
\label{lem:Edual}
Let $\bM = (\oM,\cM)$.
For $F\in\BDC(\field_{\oM\times\oR})$, one has
\[
\Edual_\bM(\Efield_\bM \ctens \quot_\bM F) \simeq
\Efield_\bM \ctens (\Qdual_\bM\quot_\bM F) \simeq
\Efield_\bM \ctens \quot_\bM(\opb a \dual_{\oM\times\oR} F),
\]
where $a$ is the involution of $M\times\R$ defined by $a(x,t) = (x,-t)$.
\end{lemma}

\section{Enhanced perverse ind-sheaves}\label{se:enhp}

As we recalled in Section~\ref{sec:perv},
a perversity endows
the triangulated category of $\R$-constructible sheaves
on a subanalytic space with a t-structure.
Here, we extend this result to the triangulated category of $\R$-constructible enhanced ind-sheaves. We allow the subanalytic space to be bordered,
and we also discuss exactness of the six Grothendieck operations.

\subsection{Subanalytic bordered spaces}
Recall Notation~\ref{not:Gammaf}.

\begin{definition}\label{def:bMsubsets}
\bnum
\item
A \emph{subanalytic bordered space} $\bM = (\oM,\cM)$ is a bordered space such that $\cM$ is a subanalytic space and $\oM$ is an open subanalytic subset of $\cM$.
\item
A morphism $f\colon \bM \to \bN = (\oN,\cN)$ of subanalytic bordered spaces is a morphism $f\colon\oM\to\oN$ of subanalytic spaces such that its graph $\Gamma_f$ is a subanalytic subset of $\cM\times\cN$, and $q_1|_{\olG}$ is proper.
In particular, $f\colon \bM \to \bN$ is a morphism of bordered spaces.
\item
$\bM$ is \emph{smooth} of dimension $d$ if $M$ is locally
 isomorphic to $\R^d$ as a subanalytic space.
\item
A subset $S$ of $\bM$ (see Definition~\ref{def:subset})
is called subanalytic if it is subanalytic in
$\cM$.
\item
A morphism $f\cl \bM\to\bN$ of subanalytic bordered spaces is
{\em submersive} if the continuous map
$\unb{f}\cl\unbM\to\unbN$ is locally (in $\unbM$)
isomorphic to the projection
$\unbN\times \R^d\to \unbN$ for some $d$.
\ee
\end{definition}

Let $\bM = (\oM,\cM)$ be a subanalytic bordered space, and consider the
embedding $j_\oM\colon \oM \to \cM$.

\begin{definition}
$\BDC_\Rc(\field_\bM)$ is the full subcategory of $\BDC(\field_\oM)$ whose objects $F$ are such that $\reimv{j_{\oM\sep!}}F$ is an $\R$-constructible object of $\BDC(\field_\cM)$.
We regard $\BDC_\Rc(\field_\bM)$ as a full subcategory of $\derd(\bM)$.
\end{definition}

\begin{proposition}
\label{pro:Rcfunctorial}
Let $f\colon \bM\to \bN$ be a morphism of subanalytic bordered spaces.
\bnum
\item
The functors $\opb f$ and $\epb f$ send $\BDC_\Rc(\field_\bN)$ to $\BDC_\Rc(\field_\bM)$.
\item
If $f$ is semi-proper, then the functors $\reeim f$ and $\roim f$ send $\BDC_\Rc(\field_\bM)$ to $\BDC_\Rc(\field_\bN)$.
\ee
\end{proposition}

In particular, the category $\BDC_\Rc(\field_\bM)$ only depends on $\bM$.

\begin{notation}\label{not:bMLCS}
For $\bM$ a subanalytic bordered space, set
\begin{align*}
\CS_\bM &\defeq \{\text{closed subanalytic subsets of $\bM$}\}, \\
\LCS_\bM &\defeq \{\text{locally closed subanalytic subsets of $\bM$}\}.
\end{align*}
For $Z\in\LCS_\bM$, denote by
\[
i_\bZ\colon \bZ \to \bM
\]
the morphism induced by the embedding $Z\subset\unbM$
(see Notation~\ref{not:bZ}).
For $k\in\Z$, set
\begin{align*}
\CS^{< k}_\bM &\defeq \{Z\in \CS_\bM \semicolon d_Z< k\}, \\
\CS^{\leq k}_\bM &\defeq \{Z\in \CS_\bM \semicolon d_Z\leq k\},
\end{align*}
and similarly for $\LCS^{< k}_\bM$ and $\LCS^{\leq k}_\bM$.
\end{notation}

\begin{definition}
Let $p$ be a perversity, $c\in\R$ and $k\in\Z_{\geq 0}$. Consider the following conditions
for $F\in\derd(\bM)$
\begin{align*}
(\ind p_k^{\leq c}) &\colon
\opb i_{\inbordered{(M\setminus Z)}} F\in \derd^{\leq c+p(k)}(\inbordered{(M\setminus Z)}) \text{ for some } Z\in\CS_\bM^{<k}, \\
(\ind p_k^{\geq c}) &\colon
\epb{i_\bZ} F \in \derd^{\geq c+p(k)}(\bZ) \text{ for any } Z\in\CS_\bM^{\leq k}.
\end{align*}
Consider the following strictly full subcategories of $\derd(\bM)$
\begin{align*}
\Dp{\leq c}(\bM)
&\defeq \{ F\in\derd(\bM) \semicolon (\ind p_k^{\leq c})\text{ holds for any $k\in\Z_{\geq 0}$} \}, \\
\Dp{\geq c}(\bM)
&\defeq \{ F\in\derd(\bM) \semicolon (\ind p_k^{\geq c})\text{ holds for any $k\in\Z_{\geq 0}$} \}.
\end{align*}
Let us also set
\begin{align*}
\Dprc{\leq c}(\field_\bM) &\defeq \Dp{\leq c}(\bM) \cap \BDC_\Rc(\field_\bM), \\
\Dprc{\geq c}(\field_\bM) &\defeq \Dp{\geq c}(\bM) \cap \BDC_\Rc(\field_\bM).
\end{align*}
\end{definition}

It is easy to check that $\bl \Dprc{\leq c}(\field_\bM), \Dprc{\geq c}(\field_\bM) \br_{c\in\R}$ satisfies the analogue of Proposition~\ref{pro:t-pre} (i) and (ii).

Note that $\bl \Dp{\leq c}(\bM), \Dp{\geq c}(\bM) \br_{c\in\R}$
is not a t-structure if $\dim M>0$.

\begin{lemma}\label{lem:alD}
For any $c\in\R$ one has
\[
\alpha_\bM \bl \Dp{\leq c}(\bM) \br \subset \Dp{\leq c}(\field_{\unb \bM}).
\]
\end{lemma}

\begin{proof}
This follows from the fact that $\alpha$ commutes with $\opb{i}$.
\end{proof}

\begin{remark}\label{rem:t-betak0}
Since $\alpha$ does not commute with the functors $\epb{i}$, the statement
\[
\alpha_\bM \bl \Dp{\geq c}(\bM) \br \subset \Dp{\geq c}(\field_\unb\bM)
\]
does not hold in general.
For example, as in \cite[Exercise 5.1]{KS01}, let
\[
\bM=\R, \quad S=\{0\}, \quad
F=\indlim[\varepsilon \to 0+]\field_{[-\varepsilon,\varepsilon]}.
\]
Then $\alpha_\bM F \simeq \field_S\in\Dp[1/2]{0}(\field_\oM)$
and $\epb{i_{S}}F\simeq\field_S[-1]$, $\epb{i_{S}}\al_\bM F\simeq\field_S$.
Hence
\[
F\in\Dp[1/2]{\geq 1/2}(\bM)
\qtq[but]
\alpha_\bM F\notin\Dp[1/2]{\geq 1/2}(\field_\oM).
\]
\end{remark}

\subsection{Intermediate enhanced perversities}

Let $\bM = (\oM,\cM)$ be a subanalytic bordered space.

\begin{definition}
\label{def:Etp}
Let $p$ be a perversity, $c\in\R$ and $k\in\Z_{\geq 0}$. Consider the following conditions
for $K\in\dere(\bM)$:
\begin{align*}
(\enh p_k^{\leq c}) &\colon
\Eopb i_{\inbordered{(M\setminus Z)}} K\in \dere^{\leq c+p(k)}(\inbordered{(M\setminus Z)})
\text{ for some } Z\in\CS_\bM^{<k}, \\
(\enh p_k^{\geq c}) &\colon
\Eepb i_\bZ K \in \dere^{\geq c+p(k)}(\bZ) \text{ for any } Z\in\CS_\bM^{\leq k}.
\end{align*}
Consider the following strictly full subcategories of $\dere(\bM)$
\begin{align*}
\tEp{\leq c}(\bM)
&\defeq \{ K\in\dere(\bM) \semicolon (\enh p_k^{\leq c})\text{ holds for any $k\in\Z_{\geq 0}$} \}, \\
\tEp{\geq c}(\bM)
&\defeq \{ K\in\dere(\bM) \semicolon (\enh p_k^{\geq c})\text{ holds for any $k\in\Z_{\geq 0}$} \}.
\end{align*}
\end{definition}

Note that $\bl \tEp{\leq c}(\bM), \tEp{\geq c}(\bM) \br_{c\in\R}$ is not a t-structures if $\dim M>0$.
However, we write
\[
\tEp{<c}(\bM)\seteq\Union_{c'<c}\;\tEp{\leq c'}(\bM), \quad
\tEp{c}(\bM)\seteq\tEp{\leq c}(\bM)\cap\tEp{\ge c}(\bM), \quad\text{etc.}
\]
\begin{remark}\label{rem:Ipk}\hfill
\bnum\item
Conditions $(\enh p_k^{\leq c})$ and $(\enh p_k^{\geq c})$ can be rewritten using the equivalences
\begin{align*}
\Eopb i_{\inbordered{(\oM\setminus Z)}} K \in \dere^{\leq c}(\inbordered{(\oM\setminus Z)})
&\iff \opb\pi\field_{\oM\setminus Z} \tens K \in \dere^{\leq c}(\bM), \\
\Eepb i_\bZ K \in \dere^{\geq c}(\bZ)
&\iff \rihom(\opb\pi\field_Z, K) \in \dere^{\geq c}(\bM),
\end{align*}
which follow from Lemma~\ref{lem:ibZfieldoZ}.
\item One has
$$\Eopb i_{\inbordered{(M\setminus Z)}} K\in
\dere^{\leq c}(\inbordered{(M\setminus Z)})\Rightarrow
\Eopb i_{\inbordered{(M\setminus Z')}} K\in \dere^{\leq c}(\inbordered{(M\setminus Z')})$$
for any $Z$, $Z'\in\CS_\bM$ such that $Z\subset Z'$.

\noi
Similarly,
$$\Eepb i_\bZ K \in \dere^{\geq c}(\bZ)\Rightarrow
\Eepb i_{\inb{Z'}} K \in \dere^{\geq c}(\inb{Z'})$$
for any $Z\in\CS_\bM$
and any locally closed subanalytic subset $Z'$ of $Z$.
Indeed, one has
$$\Eopb i_{\inbordered{(M\setminus Z')}}\simeq
\Eopb j\circ\Eopb i_{\inbordered{(M\setminus Z)}}\qtq
\Eepb i_{\inb{Z'}}\simeq \Eepb {{j'}}\circ\Eepb i_\bZ,$$ 
and $\Eopb j$ is exact and
$ \Eepb {{j'}}$ is left exact for the standard t-structure.
Here, $j\cl\inb{(M\setminus Z')}\to \inbordered{(M\setminus Z)}$
and $j'\cl\inbordered{Z'}\to \inbordered{Z}$ are the canonical morphisms.
\ee
\end{remark}

The following lemma is obvious.

\Lemma\label{lem: stp}
For any $c\in\R$, one has
\eqn
\dere^{\le c+p(d_\oM)}(\bM)\subset& \tEp{\leq c}(\bM)&
\subset \dere^{\le c+p(0)}(\bM), \\
\dere^{\ge c+p(0)}(\bM) \subset&\tEp{\ge c}(\bM)&
\subset \dere^{\ge c+p(d_\oM)}(\bM).
\eneqn
\enlemma

Note that the following lemma is a particular case of
Proposition~\ref{pro:tpt-roeimeopb} below.

\begin{lemma}
\label{lem:ibZ}
For any $c\in\R$ and any $Z\in\LCS_\bM$, one has
\begin{align*}
\Eoimv{i_{\inbordered{Z}}^{\ms{10mu}-1}}\bl\tEp{\leq c}(\bM)\br&
\subset \tEp{\leq c}(\inbordered{Z}),\\
\Eoimv{i_{\inbordered{Z}}^{\ms{10mu}!}}\bl\tEp{\geq c}(\bM)\br&
\subset \tEp{\geq c}(\inbordered{Z}), \\
\Eoimv{i_{\inbordered{Z}\sep*}}\bl \tEp{\geq c}(\inbordered{Z}) \br
&\subset \tEp{\geq c}(\bM),\\
\Eoimv{i_{\inbordered{Z}\sep!!}}\bl \tEp{\leq c}(\inbordered{Z}) \br
&\subset \tEp{\leq c}(\bM).
\end{align*}
\end{lemma}

\begin{proof}
Since the proofs are similar, let us only discuss the third inclusion.
Let $K\in\tEp{\geq c}(\inbordered{Z})$.
For $W\in\CS_\bM^{\leq k}$, consider the Cartesian diagram of bordered spaces
\[
\xymatrix{
\inbordered{(Z\cap W)} \ar[d]^i \ar[r]^-{i'} & \inbordered{W} \ar[d]^{i_{\inbordered{W}}} \\
\inbordered{Z} \ar[r]^{i_{\inbordered{Z}}} & \bM .
}
\]
Noticing that $Z\cap W \in\CS_{\inbordered{Z}}^{\leq k}$ and
that $\Eoim i'$ is left exact by Proposition~\ref{pro:t-Eop}, one has
\begin{align*}
\Eepb i_{\inbordered{W}} \Eoimv{i_{\inbordered{Z}\sep*}} K
&\simeq
\Eoim i' \Eepb i K \\
&\in \Eoim i' \bl\dere^{\geq c+p(k)}(\inbordered{(Z\cap W)})\br \\
&\subset \dere^{\geq c+p(k)}(\inbordered{W}).
\end{align*}
\end{proof}

\begin{lemma}\label{lem:crg}
For any $c\in\R$ and $K\in\dere(\bM)$, the following conditions are equivalent:
\bnum
\item
$K\in\tEp{\geq c}(\bM)$,
\item
$\Eepb i_{\bS}K \in \dere^{\geq c+p(k)}(\bS)$ for any $k\in\Z_{\ge0}$ and
any $S\in \LCS_{\bM}^{\leq k}$,
\item
$\Eepb i_{\bS} K \in \dere^{\geq c+p(k)}(\bS)$ for any $k\in\Z_{\ge0}$
and any smooth $ S\in \LCS_\bM^{\leq k}$,
\item
for any $k\in\Z_{\ge0}$ and any $Z\in \CS_\bM^{\leq k}$, there exists
an open subanalytic subset $Z_0$ of $\inbordered{Z}$
such that $\dim(Z\setminus Z_0) < k$ and
$\Eepb i_{\inbordered{(Z_0)}}K \in \dere^{\geq c+p(k)}(\inbordered{(Z_0)})$,
\item
for any $k\in\Z_{\ge0}$ and any $S\in \LCS_{\bM}^{\leq k}$, there exists
an open subanalytic subset $S_0$ of $\inbordered{S}$
such that $\dim(S\setminus S_0)< k$ and
$\Eepb i_{\inbordered{(S_0)}}K \in \dere^{\geq c+p(k)}(\inbordered{(S_0)})$.
\ee
\end{lemma}

\begin{proof}
The implications in the following diagram are clear
\[
\xymatrix@R=-1ex{
& & {\rm(iii)} \ar@{=>}[dr] \\
{\rm(i)} \ar@{=>}[r] &{\rm(ii)} \ar@{=>}[rd] \ar@{=>}[ur] & & {\rm(iv)}.\\
&&{\rm(v)} \ar@{=>}[ru]
}
\]
Here the less trivial implication (i)$\Rightarrow$(ii) follows from
Remark~\ref{rem:Ipk}~(ii).

It remains to show that (iv)$\implies$(i).
That is, we have to show that for any $Z\in\CS_\bM^{\leq k}$ one has
\begin{equation}
\label{eq:Ztemp}
\rihom(\opb\pi\field_Z, K) \in \dere^{\geq c+p(k)}(\bM).
\end{equation}
We shall prove it by induction on $k\in\Z_{\geq 0}$.
When $k=0$, \eqref{eq:Ztemp} is true,
because $Z_0$ in (iv) coincides with $Z$. Assume that $k>0$.
Let $Z_0\subset Z$ be an open subanalytic subset as in (iv), so that
\[
\rihom(\opb\pi\field_{Z_0},K) \in \dere^{\geq c+p(k)}(\bM).
\]
Since $Z\setminus Z_0\in \CS_\bM^{\leq k-1}$, the induction hypothesis
implies
\begin{align*}
\rihom(\opb\pi\field_{Z\setminus Z_0},K)
&\in \dere^{\geq c+p(k-1)}(\bM) \subset \dere^{\geq c+p(k)}(\bM).
\end{align*}
Then \eqref{eq:Ztemp} follows from the distinguished triangle
\begin{align*}
\rihom(\opb\pi\field_{Z\setminus Z_0},K)&\to \rihom(\opb\pi\field_Z,K) \\
&\to \rihom(\opb\pi\field_{Z_0},K) \tone.
\end{align*}
\end{proof}

\begin{proposition}\label{pro:tE-fihom}
For any $c,c'\in\R$, one has:
\begin{align*}
\fihom(\tEp{\leq c}(\bM), \tEp{\geq c'}(\bM)) &\subset \derd^{\geq c'-c}(\bM), \\
\fhom(\tEp{\leq c}(\bM), \tEp{\geq c'}(\bM))
&\subset \derd^{\geq c'-c}(\field_\unb\bM).
\end{align*}
In particular, $\Hom[\dere(\bM)](\tEp{\leq c}(\bM), \tEp{\geq c'}(\bM))=0$
if $c'>c$.
\end{proposition}

\begin{proof}
(i) Let $K\in\tEp{\leq c}(\bM)$ and $K'\in\tEp{\geq c'}(\bM)$.
Reasoning by decreasing induction on $k\in\Z_{\geq -1}$, let us show that
\begin{itemize}
\item[(i)$_k$] there exists $Z_k\in\CS_{\bM}^{\leq k}$ such that
\[
\rihom(\field_{M\setminus Z_k}, \fihom(K,K'))\in\derd^{\geq c'-c}(\bM).
\]
\end{itemize}
The above statement is obvious for $k\geq d_M$. Assuming that (i)$_k$
holds true for $k\geq 0$, let us prove (i)$_{k-1}$.
Since $K'\in\tEp{\geq c'}(\bM)$, one has
\[
\rihom(\opb{\pi}\field_{Z_k},K')\in \dere^{\geq c'+p(k)}(\bM).
\]
Moreover, since $K\in\tEp{\leq c}(\bM)$, there exists $W_{k-1}\in\CS_{\bM}^{\leq k-1}$ with
\[
\opb{\pi}\field_{M\setminus W_{k-1}}\tens K \in \dere^{\leq c+p(k)}(\bM).
\]
Then
\begin{align*}
\rihom&(\field_{Z_k\setminus W_{k-1}}, \fihom(K,K')) \\
&\simeq \rihom(\field_{M\setminus W_{k-1}}\tens\field_{Z_k}, \fihom(K,K')) \\
&\simeq \fihom(\opb{\pi}\field_{M\setminus W_{k-1}} \tens K,\rihom(\opb{\pi}\field_{Z_k},K')) \\
&\in \fihom \bl \dere^{\leq c+p(k)}(\bM), \dere^{\geq c'+p(k)}(\bM) \br \\
&\subset \derd^{\geq c'-c}(\bM),
\end{align*}
where the last inclusion follows from Proposition~\ref{pro:t-homE}.

Considering the distinguished triangle
\eqn
&&\rihom(\field_{Z_k\setminus W_{k-1}},{} \fihom(K,K')) \\
&&\hs{12ex}\to \rihom(\field_{M\setminus (Z_k\cap W_{k-1})}, \fihom(K,K')) \\
&&\hs{24ex}\to \rihom(\field_{M\setminus Z_k}, \fihom(K,K')) \tone,
\eneqn
we deduce (i)$_{k-1}$ for $Z_{k-1} = Z_k\cap W_{k-1}$.

\medskip
\noi
(ii)\ The second inclusion follows from the first
since $\fhom \simeq \alpha_\bM\;\fihom$.

\medskip
\noi
(iii)\ The last assertion follows from \eqref{eq:Hfhom}.
\end{proof}

\begin{lemma}\label{lem:t-fihomEE'}
For any $c,c'\in\R$, one has:
\[
\fihom(\dere^{\leq c}(\bM), \tEp{\geq c'}(\bM)) \subset \Dp{\geq c'-c}(\bM),
\]
and in particular,
\[
\fihom(\Qfield_\bM, \tEp{\geq c}(\bM)) \subset \Dp{\geq c}(\bM).
\]
\end{lemma}

\begin{proof}
Let $k\in\Z_{\ge0}$, $Z\in\CS^{\leq k}_\bM$, $K\in\dere^{\leq c}(\bM)$ and $K'\in\tEp{\geq c'}(\bM)$.
One has
\begin{align*}
\rihom(\field_Z,\fihom(K,K') )
&\simeq \fihom(K, \rihom(\opb\pi\field_Z,K') ) \\
&\in \fihom(\dere^{\leq c}(\bM), \dere^{\geq c'+p(k)}(\bM)) \\
&\subset \derd^{\geq c'-c+p(k)}(\bM),
\end{align*}
where the last inclusion follows from Proposition~\ref{pro:t-homE}.
\end{proof}

\begin{remark}\label{rem:Ehomcounter}
For $c,c'\in\R$, the inclusion
\[
\fhom(\dere^{\leq c}(\bM), \tEp{\geq c'}(\bM)) \subset \Dp{\geq c'-c}(\field_\unb\bM)
\]
does not hold in general. For example, with notations as in Remark~\ref{rem:t-betak0}, let $\bM=\R$, $K=\Efield_\bM$ and $K' = \Efield_\bM \tens \opb\pi F$.
Then $K\in\dere^{0}(\bM)$, $K'\in{}_{1/2}\dere^{\geq 1/2}(\bM)$
and
\[
\fhom( K,K') \simeq \alpha_\bM F \notin \Dp[1/2]{\geq 1/2}(\field_\bM).
\]
Here, ${}_{1/2}\dere\seteq{}_{\mathsf m}\dere$
and $\Dp[1/2]{}\seteq\Dp[{\mathsf m}]{}$
for $\mathsf m(n) \defeq -n/2$ the middle perversity.
\end{remark}

\begin{proposition}\label{pro:Eprightorth}
For $c\in\R$ one has
\[
\bl\tEp{< c}(\bM)\br^\bot = \tEp{\geq c}(\bM).
\]
\end{proposition}

\begin{proof}
One has $\tEp{\geq c}(\bM) \subset \bl\tEp{< c}(\bM)\br^\bot$ by Proposition~\ref{pro:tE-fihom}.

Let $K\in\bl\tEp{< c}(\bM)\br^\bot$. We have to show that for any $Z\in\CS_\bM^{\leq k}$ one has
\[
\Eepb i_{\bZ} K \in \dere^{\geq c+p(k)}(\bZ).
\]
Since $\dere^{\geq c+p(k)}(\bZ) = \bl \dere^{< c+p(k)}(\bZ) \br^\bot$, this is equivalent to show that for any $L\in \dere^{< c+p(k)}(\bZ)$ one has
\[
\Hom[\dere(\bZ)](L,\Eepb i_{\bZ} K) \simeq 0.
\]
By Lemma~\ref{lem: stp}, one has
$\dere^{< c+p(k)}(\bZ) \subset \tEp{< c}(\bZ)$.
Then Lemma~\ref{lem:ibZ} implies
$\Eoimv{i_{\bZ\sep!!}} L \in \tEp{< c}(\bM)$, so that
\[
\Hom[\dere(\bZ)](L,\Eepb i_{\bZ} K)
\simeq \Hom[\dere(\bM)](\Eoimv{i_{\bZ\sep!!}} L,K)
\simeq 0.
\]
\end{proof}

\Prop\label{prop:stack}
Let $M$ be a subanalytic space.
For any interval $I\subset\R$ such that $I\to\R/\Z$ is injective,
the prestack on $M$
\[
U\mapsto \tEp{I}(U)
\]
is a stack.
\enprop

\begin{proof}
(i) Let $K,L\in\tEp{I}(M)$. By Proposition~\ref{pro:tE-fihom}, one has
\[
\fhom(K,L)\in\derd^{>-1}(M)
=\derd^{\geq 0}(M).
\]
Hence
the presheaf
\[
U\mapsto\Hom[\tEp{I}(U)](\Eopb i_U K,\Eopb i_UL)\simeq
\sect\bl U;H^0(\fhom(K,L))\br
\]
is a sheaf.
Thus $U\mapsto \tEp{I}(U)$ is a separated prestack on $M$.

\medskip\noi
(ii) Let $M=\Union\nolimits_{a\in A}U_a$ be an open cover.
Let $K_a \in\tEp{I}(U_a)$ and let $u_{ab}\colon K_b|_{U_a\cap U_b} \isoto K_a|_{U_a\cap U_b}$ be isomorphisms such that $u_{ab}\circ u_{bc} = u_{ac}$ on $U_a\cap U_b\cap U_c$ ($a,b,c\in A$).
 We have to show that there exist $K \in\tEp{I}(M)$ and isomorphisms $u_{a}\colon K|_{U_a} \isoto K_a$ such that $u_{ab}\circ u_b = u_a$ on $U_a\cap U_b$ ($a,b\in A$).
This follows from Proposition~\ref{prop:patch}
by applying it to $\reeim{{j_a}}\LE K_a\in\derd(U_a\times\cR)$,
where $j_a\cl U_a\times\bR\to U_a\times\cR$ is the canonical morphism.
\end{proof}

\begin{lemma}\label{lem:recoll}
Let $\bM$ be a bordered space.
Let $c\in\R$, $Z\in\CS_\bM$ and $K\in\dere(\bM)$.
Set $U=\unbM\setminus Z$.
Then, considering the morphisms
\[
\xymatrix{
\inbordered{Z} \ar[r]^-i & \bM & \inbordered{U}\ar[l]_-j,
}
\]
one has:
\begin{itemize}
\item[(i)]
$K\in\tEp{\leq c}(\bM)$ if and only if
\[
\Eopb i K\in\tEp{\leq c}(\inbordered{Z}) \text{ and }
\Eopb j K\in\tEp{\leq c}(\inbordered{U} ),
\]
\item[(ii)]
$K\in\tEp{\geq c}(\bM)$ if and only if
\[
\Eepb i K\in\tEp{\geq c}(\inbordered{Z}) \text{ and }
\Eepb j K\in\tEp{\geq c}(\inbordered{U}).
\]
\end{itemize}
\end{lemma}

\begin{proof}
Since the proofs are similar, let us only discuss (i).

If $K\in\tEp{\leq c}(\bM)$, then $\Eopb i K$ and $\Eopb j K$ satisfy the required conditions since the functors $\Eopb i$ and $\Eopb j$ are right exact
by Lemma~\ref{lem:ibZ}.

Conversely, assume that $\Eopb i K\in\tEp{\leq c}(\inbordered{Z})$ and
$\Eopb j K\in\tEp{\leq c}(\inbordered{U})$. For $k\in\Z_{\geq 0}$,
let $S_U\in\CS^{<k}_{\inbordered{U}}$ be such that
\[
\opb\pi\field_{U\setminus S_U} \tens \Eopb j K \in \dere^{\leq c +p(k)}(\inbordered{U}),
\]
and $S_Z\in\CS^{<k}_{\inbordered{Z}}$ be such that
\[
\opb\pi\field_{Z\setminus S_Z} \tens \Eopb i K \in \dere^{\leq c +p(k)}(\inbordered{Z}).
\]
Set $S=S_Z\cup \overline{S_U}\in\CS^{<k}_\bM$ and $S_Z' = S_Z\cup(Z\cap\overline{S_U}) \in \CS^{<k}_{\inbordered{Z}}$. (Here the closure of $S_U$ is taken in $\unbM$.)
Then $ S\cap U=S_U$ and $S\cap Z=S'_Z$.
 Since
\[
\opb\pi\field_{U\setminus S_U} \tens K \in \dere^{\leq c +p(k)}(\bM), \quad
\opb\pi\field_{Z\setminus S_Z'} \tens K \in \dere^{\leq c +p(k)}(\bM),
\]
one concludes that
$\opb\pi\field_{M\setminus S} \tens K\in \dere^{\leq c+p(k)}(\bM)$
by considering the distinguished triangle
\[
\opb\pi\field_{U\setminus S_U} \tens K \to \opb\pi\field_{M\setminus S} \tens K
\to \opb\pi\field_{Z\setminus S_Z'} \tens K \tone.
\]
\end{proof}

A \emph{subanalytic stratification} $\{M_\alpha\}_{\alpha\in A}$ of $\bM\seteq(\oM,\cM)$ is a locally finite (in $\cM$) family
of smooth $M_\alpha\in\LCS_\bM$ such that $\oM=\DUnion\nolimits_{\alpha\in A} M_\alpha$ and $\overline M_\alpha \cap M_\beta \neq 0$ implies $\overline M_\alpha \supset M_\beta$.

\begin{proposition}
\label{pro:tE-stratnoRc}
Let $\{M_\alpha\}_{\alpha\in A}$ be a subanalytic stratification of $\bM$,
and set $\bM_\alpha = \inbordered{(M_\alpha)}$.
Let $K\in\dere(\bM)$.
\bnum
\item
$K\in\tEp{\leq c}(\bM)$ if and only if
$\Eopb i_{\bM_\alpha}K \in \tEp{\leq c}(\bM_\alpha)$ for any $\alpha\in A$,
\item $K\in\tEp{\geq c}(\bM)$ if and only if
$\Eepb i_{\bM_\alpha}K \in \tEp{\geq c}(\bM_\alpha)$ for any $\alpha\in A$.
\ee
\end{proposition}

\begin{proof}
The statement follows from Lemma~\ref{lem:recoll}.
\end{proof}

\subsection{$\R$-constructible enhanced ind-sheaves}
Here, we extend the definition of $\R$-constructible enhanced ind-sheaves from \cite[\S4.9]{DK13} to the case of subanalytic bordered spaces.

Let $\bM=(\oM,\cM)$ be a subanalytic bordered space.

\begin{definition}
\begin{itemize}
\item[(i)]
An object $K\in\dere(\bM)$ is \emph{$\R$-constructible} if for any
relatively compact subanalytic open subset $U$ of $\bM$,
one has
\[
\Eopb i_{\bU} K \simeq \Efield_\bU \ctens \quot_\bU F \
\text{ in $\dere(\inb U)$ for some $F\in \BDC_\Rc(\field_{\bU\times\bR})$.}
\]
In particular, $K$ is stable.
\item[(ii)]
$\Erc(\bM)$ is the strictly full subcategory of $\dere(\bM)$ whose objects are $\R$-constructible.
\end{itemize}
\end{definition}

Recall the morphism $j_\bM\colon\bM\to\cM$.

\begin{lemma}
Let $K\in\dere(\bM)$. Then $K\in\Erc(\bM)$ if and only if $\Eoimv{j_{\bM\sep!!}} K\in\Erc(\cM)$.
\end{lemma}

\Prop[\cite{DK13}]\label{prop:dualRc}
Let $f\cl\bM\to\bN$ a morphism of
subanalytic bordered spaces.
\bnum
\item
$\Erc(\bM)$ is a triangulated subcategory of $\dere(\bM)$.
\item
The duality functor $\Edual_\bM$ gives an equivalence $\Erc(\bM)^\op\isoto\Erc(\bM)$,
and there is a canonical isomorphism of functors
$\id_{\Erc(\bM)}\isoto \Edual_\bM\circ\Edual_\bM$.
\item The functors
$\Eopb f$ and $\Eepb f$ send $\Erc(\bN)$ to $\Erc(\bM)$, and
$$\Edual_\bM\circ\Eopb f\simeq \Eepb f\circ\Edual_\bN
\qtq \Edual_\bM\circ\Eepb f\simeq \Eopb f\circ\Edual_\bN.$$
\item Assume that $f$ is semi-proper. Then
the functors
$\Eoim f$ and $\Eeeim f$ send $\Erc(\bM)$ to $\Erc(\bN)$, and
$$\Edual_\bN\circ\Eoim f\simeq \Eeeim f\circ\Edual_\bM
\qtq \Edual_\bN\circ\Eeeim f\simeq \Eoim f\circ\Edual_\bM.$$
\ee
\enprop

See \cite[Corollary 4.9.4, Theorem 4.9.12, Propositions 4.9.14, 4.8.2]{DK13}.

\begin{definition}
\begin{itemize}
\item[(i)]
An \emph{$\enh$-type} on $\bM$ is the datum
\begin{equation}
\label{eq:typeT}
\shl = (\varphi_a, m_a, \psi^\pm_b, n_b)_{a\in A,\ b\in B}
\end{equation}
consisting of
\begin{itemize}
\item[(a)] finite sets $A, B$,
\item[(b)] integers $m_a$ and $n_b$ for any $a\in A$ and $b\in B$,
\item[(c)] morphisms of subanalytic bordered spaces
\[
\varphi_a,\psi^\pm_b\colon \bM\to\bR
\]
for any $a\in A$ and $b\in B$,
such that $\psi^-_b(x) < \psi^+_b(x)$ for any $x\in \oM$.
\end{itemize}
\item[(ii)]
An $\enh$-type $\shl$ as in \eqref{eq:typeT} is called \emph{stable} if for any $b\in B$
\begin{equation}
\label{eq:PsiHyp}
\overline{ \{(x,t) \in \oM\times\oR \semicolon t= \psi^+_b(x) - \psi^-_b(x) \} }
\cap (\cM\times\{+\infty\}) \neq \emptyset,
\end{equation}
where $\ol{\ast}$ denotes the closure in $\cM\times\cR$.
\end{itemize}
\end{definition}

\begin{notation}\label{not:PhiaPsib}
For an $\enh$-type $\shl$ on $\bM$ as in \eqref{eq:typeT}, set
\begin{align*}
\Phi_a &\defeq \{(x,t)\in \oM\times\R \semicolon t \geq \varphi_a(x)\}, \\
\Psi_b &\defeq \{(x,t)\in \oM\times\R \semicolon \psi^-_b(x) \leq t < \psi^+_b(x)\},
\end{align*}
and
\begin{align*}
\Qfield_\shl &\defeq \bl \soplus_{a\in A} \Qfield_{\Phi_a}[-m_a] \br
\oplus
\bl \soplus_{b\in B} \Qfield_{\Psi_b}[-n_b] \br
\in\dere(\bM), \\
\Efield_\shl &\defeq \bl \soplus_{a\in A} \Efield_{\Phi_a}[-m_a] \br
\dsum
\bl \soplus_{b\in B} \Efield_{\Psi_b}[-n_b] \br \\
&\simeq \Efield_\bM \ctens \Qfield_\shl \in\Erc(\bM).
\end{align*}
\end{notation}

Note that $\Efield_{\Psi_b}\not\simeq 0$
if and only if \eqref{eq:PsiHyp} holds true.

\begin{definition}
One says that $K\in\dere(\bM)$ is \emph{free} (resp.\ \emph{stably free}) on $\bM$ if, for any connected component $S$ of $\unbM$, there exists an $\enh$-type $\shl$ on $\inbordered S$ such that $\Eopb i_{\inbordered S} K \simeq \Qfield_\shl$ (resp.\ $\Eopb i_{\inbordered S} K \simeq \Efield_\shl$).
(Note that $\Eopb i_{\inbordered S} \simeq \Eepb i_{\inbordered S}$.)
\end{definition}

If $K\in\dere(\bM)$ is stably free, then it is $\R$-constructible.
If $K$ is free, then it is constructible in the sense of Remark~\ref{rem:TRc} below.

A \emph{regular filtration} $(M_k)_{k\in\Z}$ of $\bM$
is an increasing sequence of closed subanalytic subsets
$M_k$ of $\bM$ such that $M_k=\emptyset$ for $k\leq -1$, $M_k=\unbM$ for
$k\geq d_\unbM$, and $M_k \setminus M_{k-1}$ is smooth of dimension $k$. In particular,
\[
\emptyset = M_{-1} \subset M_0 \subset \cdots \subset M_{d_M-1} \subset M_{d_M} = \unbM.
\]

\begin{lemma}[{\cite[Lemma 4.9.9]{DK13}}]\label{lem:Rcstrat}
For any $K\in\Erc(\bM)$ there exists
a regular filtration $(M_k)_{k\in\Z}$ of $\bM$ such that
both $\Eopbv {i_{\inbordered{(M_k \setminus M_{k-1})}}^{\hs{.8ex}-1}} K$ and
$\Eepbv{i_{\inbordered{(M_k \setminus M_{k-1})}}^{\hs{.8ex}!}} K$ are stably free.
\end{lemma}

\begin{definition}\label{def:dualE}
Consider an \emph{$\enh$-type} on $\bM$
\[
\shl = (\varphi_a, m_a, \psi^\pm_b, n_b)_{a\in A,\ b\in B},
\]
and assume that $\bM$ is smooth of dimension $d$.
The dual of $\shl$, denoted by
\[
\shl^* = (\varphi_a^*, m_a^*, \psi^{\pm\sep*}_b, n_b^*)_{a\in A,\ b\in B},
\]
is the $\enh$-type on $\bM$ defined by
\begin{align*}
\varphi_a^* \defeq -\varphi_a, &\quad m_a^* \defeq -m_a-d, \\
\psi^{\pm\sep*}_b \defeq -\psi^\mp_b, &\quad n_b^* \defeq -n_b-d-1.
\end{align*}
Accordingly, we set
\eqn
\Phi^*_a &\defeq& \{(x,t)\in \oM\times\R \semicolon t \geq -\varphi_a(x)\}, \\
\Psi^*_b &\defeq& \{(x,t)\in \oM\times\R \semicolon
-\psi^+_b(x) \leq t < -\psi^-_b(x)\}.
\eneqn
\end{definition}

\begin{lemma}
Let $\shl$ be an $\enh$-type on $\bM$.
Assume that $\oM$ is smooth and equidimensional. Then
\[
\Qdual_\bM \Qfield_\shl \simeq \Qfield_{\shl^*}\qtq
\Edual_\bM \Efield_\shl \simeq \Efield_{\shl^*}\;.
\]
\end{lemma}

\begin{proof}
This follows from Lemma~\ref{lem:PhiPsiD} below.
\end{proof}

\begin{lemma}\label{lem:PhiPsiD}
Recall {\rm Notation~\ref{not:PhiaPsib}} and
{\rm Definition~\ref{def:dualE}}.
If $\bM$ is smooth of dimension $d$, one has
\begin{align*}
\Qdual_{\bM}(\Qfield_{\Phi_a}) &\simeq \Qfield_{\Phi^*_a}[d], &
\Edual_{\bM}(\Efield_{\Phi_a}) &\simeq \Efield_{\Phi^*_a}[d], \\
\Qdual_{\bM}(\Qfield_{\Psi_b}) &\simeq \Qfield_{\Psi^*_b}[d+1], &
\Edual_{\bM}(\Efield_{\Psi_b}) &\simeq \Efield_{\Psi^*_b}[d+1].
\end{align*}
\end{lemma}

\begin{proof}
By Lemma~\ref{lem:Edual},
 one has
\begin{align*}
\Qdual_{\bM}(\Qfield_{\Phi_a})
&\simeq \Qfield_{\{t<-\varphi_a(x)\}}[d+1] \simeq \Qfield_{\Phi^*_a}[d], \\
\Qdual_{\bM}(\Qfield_{\Psi_b})
&\simeq \Qfield_{\{-\psi^+_b(x) \leq t < -\psi^-_b(x) \}}[d+1] = \Qfield_{\Psi^*_b}[d+1].
\end{align*}
The other statements also follow from Lemma~\ref{lem:Edual}.
\end{proof}

\begin{definition}
For $p$ a perversity and $c\in\R$, we set
\begin{align*}
\tEprc{\leq c}(\bM) &\defeq \tEp{\leq c}(\bM) \cap \Erc(\bM), \\
\tEprc{\geq c}(\bM) &\defeq \tEp{\geq c}(\bM) \cap \Erc(\bM).
\end{align*}
\end{definition}

\begin{proposition}
\label{pro:t-pre1}
The following properties hold.
\begin{itemize}
\item[(i)]
$\bl \tEprc{\leq c}(\bM), \tEprc{\geq c}(\bM) \br_{c\in\R}$
is a t-structure on $\Erc(\bM)$.
\item[(ii)]
Assume that $\bM=M$ is a subanalytic space.
For any interval $I\subset\R$ such that $I\to\R/\Z$ is injective,
the prestack on $M$
\[
U\mapsto \tEprc{I}(U)
\]
is a stack of quasi-abelian categories.
\end{itemize}
\end{proposition}

\begin{proof}[Plan of the proof]
(i) We have to prove that the conditions in Definition~\ref{def:gent} are satisfied.
Conditions (a) and (b) are clear. Condition (c) follows from Proposition~\ref{pro:tE-fihom}.
Condition (d) is checked in Proposition~\ref{pro:Epdt} below.

\smallskip\noi
(ii) follows from Proposition~\ref{prop:stack}.
\end{proof}

\begin{notation}
We denote by
\[
\bl \tEprc[1/2]{\leq c}(\bM), \tEprc[1/2]{\geq c}(\bM) \br_{c\in\R}
\]
the t-structure associated with the middle perversity $\mathsf m(n) = -n/2$.
\end{notation}

\begin{remark}
The t-structures $\bl \tEprc{\leq c}(\bM), \tEprc{\geq c}(\bM) \br_{c\in\R}$
are not well behaved with respect to duality, as one observes
in Lemma~\ref{lem:t-PhiPsi} below.
We will come back to this point in \S\,\ref{sse:enhp}.
\end{remark}

\begin{lemma}
\label{lem:t-PhiPsi}
Assume that $\bM$ is smooth of dimension $d$.
Using {\rm Notation~\ref{not:PhiaPsib}}, one has
\begin{itemize}
\item[(i)]
$\Efield_{\Phi_a}, \Efield_{\Psi_b} \in \tEprc{-p(d)}(\bM)$,
\item[(ii)]
$\Edual_\bM\Efield_{\Phi_a} \in \tEprc[p^*]{p(d)}(\bM)$\qtq
$\Edual_\bM\Efield_{\Psi_b} \in \tEprc[p^*]{p(d)-1}(\bM)$.
\end{itemize}
\end{lemma}

\begin{proof}
(i) As the proofs are similar, let us only discuss $\Efield_{\Phi_a}$.

(i-1) It is straightforward that $\Efield_{\Phi_a}
\in \tEprc{\leq-p(d)}(\bM)$.

(i-2) Let us show that $\Efield_{\Phi_a}
\in \tEprc{\geq-p(d)}(\bM)$. We have to prove that for any smooth
$Z\in \LCS_{\bM}^{\leq k}$ one has
\[
\enh\,i_{\bZ}^{\,!}(\Efield_{\Phi_a}) \in \dere^{\geq -p(d) + p(k)}(\bZ).
\]
We may assume that $k<d$.
Note that
\begin{align*}
\enh\,i_{\bZ}^{\,!}(\Efield_{\Phi_a})
&\simeq \Efield_\bM \ctens \quot_\bM \bl i_{Z}^{\,!}(\field_{\{t\geq\varphi_a(x)\}}) \br \\
&\simeq \Efield_\bM \ctens \quot_\bM \bl
\opb i_Z(\field_{\{t\geq\varphi_a(x)\}}) \tens i_{Z}^{\;!}\field_\oM\br.
\end{align*}
Locally on $Z$, one has
$\ i_{Z}^{\;!}\field_\oM\simeq \field_{Z}[k-d]$.
Hence
\[
\enh\,i_{\bZ}^{\,!}(\Efield_{\Phi_a}) \in \dere^{\geq d-k}(\bZ)
\]
by Lemmas~\ref{lem:dereQ} and \ref{lem:t-enh}.
One concludes since $d-k \geq -p(d) + p(k)$ by perversity.

(ii) Using Lemma~\ref{lem:PhiPsiD} and (i), one has
\begin{align*}
\Edual_\bM\Efield_{\Phi_a}
&\simeq \Efield_{\Phi_a^*}[d] \in \tEprc[p^*]{-p^*(d)-d}(\bM) = \tEprc[p^*]{p(d)}(\bM), \\
\Edual_\bM\Efield_{\Psi_b}
&\simeq \Efield_{\Psi_b^*}[d+1] \in \tEprc[p^*]{-p^*(d)-d-1}(\bM) = \tEprc[p^*]{p(d)-1}(\bM).
\end{align*}
\end{proof}

\begin{lemma}
\label{lem:t-L}
Assume that $\unbM$ is non empty and smooth of dimension $d$.
For
\[
\shl = (\varphi_a, m_a, \psi^\pm_b, n_b)_{a\in A,\ b\in B}
\]
a stable $\enh$-type on $\bM$, and $c\in\R$, one has
\begin{itemize}
\item[(i)]
$\Efield_\shl \in \tEprc{\leq c}(\bM)$ if and only if for any $a\in A$ and $b\in B$
\[
m_a\leq c+p(d), \quad n_b \leq c+p(d),
\]
\item[(ii)]
$\Efield_\shl \in \tEprc{\geq c}(\bM)$ if and only if for any $a\in A$ and $b\in B$
\[
m_a\geq c+p(d), \quad n_b \geq c+p(d),
\]
\item[(iii)]
$\Edual_\bM\Efield_\shl \in \tEprc[p^*]{\geq -c}(\bM)$ if and only if for any $a\in A$ and $b\in B$
\[
m_a\leq c+p(d), \quad n_b \leq c+p(d)-1,
\]
\item[(iv)]
$\Edual_\bM\Efield_\shl \in \tEprc[p^*]{\leq -c}(\bM)$ if and only if for any $a\in A$ and $b\in B$
\[
m_a\geq c+p(d), \quad n_b \geq c+p(d)-1.
\]
\end{itemize}
\end{lemma}

\begin{proof}
Since
\[
\Efield_\shl =
\bl \DSum_{a\in A} \Efield_{\Phi_a}[-m_a] \br
\dsum
\bl \DSum_{b\in B} \Efield_{\Psi_b}[-n_b] \br,
\]
the statement follows from Lemma~\ref{lem:t-PhiPsi}.
Note that a non-zero object of $\tEp{c}(\bM)$
belongs to $\tEp{\le c'}(\bM)$ (resp.\ $\tEp{\ge c'}(\bM)$)
if and only if $c\le c'$ (resp.\ $c\ge c'$)
by Proposition~\ref{pro:tE-fihom}.
\end{proof}

\begin{corollary}\label{cor:tEkL}
Assume that $\bM$ is smooth of dimension $d$.
Let $K\in\Erc(\bM)$ be a stably free object.
Then, for $c\in\R$, one has
\begin{itemize}
\item[(i)]
$K\in\tEprc{\leq c}(\bM)$ if and only if $K\in \dere_\Rc^{\leq c+p(d)}(\bM)$,
\item[(ii)]
$K\in \tEprc{\geq c}(\bM)$ if and only if $K\in \dere_\Rc^{\geq c+p(d)}(\bM)$.
\end{itemize}
\end{corollary}

\begin{lemma}\label{lem:EpdtFree}
Let $c\in\R$ and $K\in\Erc(\bM)$.
Assume that $\bM$ is smooth and $K$ is stably free on $\bM$.
Then there are distinguished triangles in $\Erc(\bM)$
\[
K_{\leq c} \to K \to K_{> c} \tone\qtq
K_{< c} \to K \to K_{\geq c}\tone
\]
with $K_{L}\in \tEprc{L}(\bM)$ for $L$ equal to $\leq c$, $>c$, $<c$ or $\geq c$.
\end{lemma}

\begin{proof}
It is obvious since $K$ is a direct sum of objects
belonging to $\tEprc{a}(\bM)$ for some $a\in\R$
by Lemma~\ref{lem:t-PhiPsi}.
\end{proof}

\begin{proposition}\label{pro:Epdt}
Let $c\in\R$ and $K\in\Erc(\bM)$.
Then there are distinguished triangles in $\Erc(\bM)$
\[
K_{\leq c} \to K \to K_{> c} \tone\qtq
K_{< c} \to K \to K_{\geq c} \tone
\]
with $K_{L}\in \tEprc{L}(\bM)$ for $L$ equal to $\leq c$, $>c$, $<c$ or $\geq c$.
\end{proposition}

\begin{proof}
Since the proof of the existence of the second distinguished triangle
follows from the first one,
we will construct only the first distinguished triangle.
The arguments we use are standard (see e.g.\ \cite[Lemma 5.8]{Kas15}).

 Let $\bM=(\oM,\cM)$.
Reasoning by decreasing induction on $k\in\Z_{\geq -1}$, let us show that
\begin{itemize}
\item[(dt)$_k$] there exists $Z_k\in\CS_\bM^{\leq k}$ and a distinguished triangle
\[
K'_k \to \Eopb{j_k} K \to K''_k \tone,
\]
with $K'_k\in \tEprc{\leq c}(\inbordered{(\oM\setminus Z_k)})$ and $K''_k\in \tEprc{> c}(\inbordered{(\oM\setminus Z_k)})$.
\end{itemize}
Here, $j_k$ is the morphism indicated in the diagram below, where we picture all the morphisms that will be used in the proof.
\[
\xymatrix{
\inbordered{(Z_k\setminus Z_{k-1})} \ar[r]^-{i'_{k}} \ar[dr]_-{i_k} & \inbordered{(\oM\setminus Z_{k-1})} \ar[d]^-{j_{k-1}} & \inbordered{(M\setminus Z_k)} \ar[l]_-{j'_{k}} \ar[dl]^-{j_k} \\
& \bM
}
\]
The statement (dt)$_k$ is obvious for $k\geq d_M$. Assuming that (dt)$_k$ holds true for some $k\ge0$, let us prove (dt)$_{k-1}$.

The morphism $K'_k \to \Eopb{j_k} K \simeq \Eepb{j_k} K$ induces by adjunction a morphism $\Eoimv{j_{k\sep!!}} K'_k \to K$, that we complete in a distinguished triangle in $\dere_\Rc(\bM)$
\[
\Eoimv{j_{k\sep!!}} K'_k \to K \to L \tone.
\]
Let $Z_{k-1}\in\CS_\bM^{\leq k-1}$ be such that $Z_k\setminus Z_{k-1}$ is smooth and $\Eepb i_k L$ is stably free. By Lemma~\ref{lem:EpdtFree}, there is a distinguished triangle
\begin{equation}
\label{eq:LLL}
L' \to \Eepb i_k L \to L'' \tone,
\end{equation}
with $L'\in \tEprc{\leq c}(\inbordered{(Z_k\setminus Z_{k-1})})$ and $L''\in \tEprc{> c}(\inbordered{(Z_k\setminus Z_{k-1})})$.

The morphism $L' \to \Eepb i_k L \simeq \Eopbv{i_{k}^{\prime\sep!}}\Eepb j_{k-1} L$ induces by adjunction a morphism $\Eoimv{i'_{k\sep!!}} L' \to \Eepb j_{k-1} L \simeq \Eopb j_{k-1} L$, that we complete in a distinguished triangle in $\dere_\Rc(\inbordered{(\oM\setminus Z_{k-1})})$
\begin{equation}
\label{eq:LLK}
\Eoimv{i'_{k\sep!!}} L' \to \Eopb j_{k-1} L \to K''_{k-1} \tone.
\end{equation}
Consider the composite morphism $\Eopb j_{k-1}K \to \Eopb j_{k-1}L \to K''_{k-1}$, and complete it in a distinguished triangle in $\dere_\Rc(\inbordered{(\oM\setminus Z_{k-1})})$
\[
K'_{k-1} \to \Eopb j_{k-1}K \to K''_{k-1} \tone.
\]
We claim that this satisfy (dt)$_{k-1}$.

Note that
\begin{align*}
\Eopbv{j_{k}^{\prime\sep-1}}K''_{k-1} &\simeq \Eopb j_k L \simeq K''_k \in \tEprc{> c}(\inbordered{(\oM\setminus Z_k)}),\\
\Eopbv{j_{k}^{\prime\sep-1}}K'_{k-1} &\simeq K'_k \in \tEprc{\leq c}(\inbordered{(\oM\setminus Z_k)}).
\end{align*}
Hence, by Lemma~\ref{lem:recoll}, we are reduced to prove
\begin{align}
\label{eq:opbiK'}
\Eopbv{i_{k}^{\prime\sep-1}} K'_{k-1} & \in \tEprc{\leq c}(\inbordered{(Z_k\setminus Z_{k-1})}), \\
\label{eq:epbiK'}
\Eopbv{i_{k}^{\prime\sep!}}K''_{k-1} & \in \tEprc{> c}(\inbordered{(Z_k\setminus Z_{k-1})}).
\end{align}

Applying the functor $\Eopbv{i_{k}^{\prime\sep!}}$ to \eqref{eq:LLK}, we get a distinguished triangle
\[
L' \to \Eepb i_k L \to \Eopbv{i_{k}^{\prime\sep!}} K''_{k-1} \tone.
\]
Thus \eqref{eq:LLL} gives $\Eopbv{i_{k}^{\prime\sep!}} K''_{k-1} \simeq L''\in \tEprc{> c}(\inbordered{(Z_k\setminus Z_{k-1})})$, which proves \eqref{eq:epbiK'}.

By the octahedral axiom, there is a diagram in $\dere(\inbordered{(\oM\setminus Z_{k-1})})$
\[
\xymatrix@!0@C=7em@R=3.5em{
& K'_{k-1} \ar@{.>}[dr]\ar[ddl] \\
\Eopb j_{k-1} \Eoimv{j_{k\sep!!}}K'_k \ar@{.>}[ur]\ar[d] && \Eoimv{i'_{k\sep!!}} L' \ar[ll]|-{+1}\ar[ddl] \\
\Eopb j_{k-1}K \ar[dr]\ar[rr] && K''_{k-1}\;, \ar[u]_{+1}\ar[uul]|(.65){+1} \\
& \Eopb j_{k-1} L \ar[ur]\ar[uul]|(.63){+1}
}
\]
and hence a distinguished triangle
\[
\Eopb j_{k-1}\Eoimv{j_{k\sep!!}}K'_k \to K'_{k-1} \to \Eoimv{i'_{k\sep!!}} L' \tone.
\]
Applying the functor $\Eopbv{i_{k}^{\prime\sep-1}}$, we get
\[
\Eopbv{i_{k}^{\prime\sep-1}} K'_{k-1} \simeq L' \in \tEprc{\leq c}(\inbordered{(Z_k\setminus Z_{k-1})}),
\]
which proves \eqref{eq:opbiK'}.
\end{proof}

\begin{definition}
For $p\colon\Z_{\geq 0}\to\R$ a perversity and $d\in\Z_{\geq 0}$, the shifted perversity $p[d]$ is given by
\[
p[d](n) = p(d+n).
\]
\end{definition}

Note that the soft dimension of
a subanalytic space is equal to its dimension.

\begin{proposition}
\label{pro:tpt-roeimeopb}
Let $f\colon \bM \to \bN$ be a morphism of subanalytic bordered spaces, and $d\in\Z_{\geq 0}$.
Assume that $\dim \unb{f}{}^{-1}(y) \leq d$
for any $y\in\unbN$. Then, for any $c\in\R$ one has
\bnum
\item
$\Eopb f \bl\tEp[{p[d]}]{\leq c}(\bN)\br \subset \tEp{\leq c}(\bM)$,
\item
$\Eepb f \bl\tEp[{p[d]}]{\geq c}(\bN)\br \subset \tEp{\geq c-d}(\bM)$.
\item
$\Eoim f \bl\tEp{\geq c}(\bM)\br \subset \tEp[{p[d]}]{\geq c}(\bN)$,
\item
$\dere_\Rc(\bN) \cap \Eeeim f \bl\tEp{\leq c}(\bM)\br \subset \tEp[{p[d]}]{\leq c+d}(\bN)$.
\ee
\end{proposition}

\begin{proof}
Let $\bM=(\oM,\cM)$ and $\bN=(\oN,\cN)$.

\noi
(i) Let $L\in\tEp[{p[d]}]{\leq c}(\bN)$. We have to prove that,
for any $k\in\Z_{\ge0}$,
there exists $Z\in\CS_\bM^{<k}$ such that $\Eopb i_{\inbordered{(M\setminus Z)}} \Eopb f L\in
 \dere^{\leq c+ p(k)}(\inbordered{(M\setminus Z)})$. Let $W\in\CS_\bN^{<k-d}$ be such that $\Eopb i_{\inbordered{(N\setminus W)}} L\in \dere^{\leq c+ p(k)}(\inbordered{(N\setminus W)})$. Note that if $0\le k<d$, then $W=\emptyset$ will do
because $L\in \dere^{\leq c+ p[d](0)}(\bN)\subset\dere^{\leq c+ p(k)}(\bN)$.

Then $Z\seteq f^{-1}(W)\in\CS_\bM^{<k}$ satisfies the desired condition.
Indeed, denoting $f_0\colon \inbordered{(M\setminus Z)} \to \inbordered{(N\setminus W)}$ the morphism induced by $f|_{M\setminus Z}$, one has
\begin{align*}
\Eopb i_{\inbordered{(M\setminus Z)}} \Eopb f L
&\simeq \Eopb f_0 \Eopb i_{\inbordered{(N\setminus W)}} L \\
&\in \Eopb f_0 \dere^{\leq c+ p(k)}(\inbordered{(N\setminus W)}) \\
&\subset \dere^{\leq c+ p(k)}(\inbordered{(M\setminus Z)}),
\end{align*}
where the last inclusion follows from Proposition~\ref{pro:t-Eop}.

\smallskip\noi
(ii) Let $L\in\tEp[{p[d]}]{\geq c}(\bN)$.
We have to show that for any $Z\in\CS_\bM^{\leq k}$ there exists
an open subanalytic subset $Z_0$ of $\inbordered{Z}$
such that $\dim(Z\setminus Z_0) < k$ and
\begin{equation}
\label{eq:tempiZ0}
\Eepb i_{\inbordered{(Z_0)}} \Eepb f L \in \dere^{\geq c+p(k)-d}(\inbordered{(Z_0)}).
\end{equation}
Recall Notation~\ref{not:Gammaf}.
Replacing $\cM$ with $\olG$, we may assume that $f$ extends to a morphism of subanalytic spaces $\bclose f\colon\cM\to\cN$.

Since \eqref{eq:tempiZ0} is local on $\cM$, we may assume that $Z$ is
relatively compact in $\cM$.
Then, there exists an open subanalytic subset $Z_0$ of
$Z$ satisfying the following properties:
\bna
\item $\dim(Z\setminus Z_0)<k$,
\item $Z_0=\DUnion_{i\in I}S_i$, where
$\{S_i\}_{i\in I}$ is a family of subanalytic
smooth subsets of dimension $k$,
\item $T_i\seteq f(S_i)$ is a smooth equidimensional subset of $\bN$
for any $i\in I$,
\item $f$ induces a submersive morphism
$f_i\cl \inb{(S_i)}\to\inb{(T_i)}$ for any $i\in I$.
\ee
We claim that $Z_0$ satisfies \eqref{eq:tempiZ0}.
In fact, for any $i\in I$, one has
\begin{align*}
\Eepb i_{\inbordered{(S_i)}} \Eepb f L
&\simeq \Eepb f_i \Eepb i_{\inbordered{(T_i)}} L \\
&\in \Eepb f_i\, \dere^{\ge c+p(d_{T_i}+d)}(\inbordered{(T_i)}).
\end{align*}
Since $f_i$ is submersive, we have
$\Eepb f_i \simeq\ori_{S_i/T_i}\tens \Eopb f_i [d_{S_i}-d_{T_i}]$,
where $\ori_{S_i/T_i}$ is the relative orientation sheaf
(see \S\,\ref{subsec:sheaf}).
Hence we have
\eqn
\Eepb f_i\, \dere^{\ge c+p(d_{T_i} +d)}(\inbordered{(T_i)})
&\subset& \dere^{\ge c+p(d_{T_i} +d)+d_{T_i}-d_{S_i}}(\inbordered{(S_i)})\\
&\subset &\dere^{\ge c+p(d_{S_i})-d}(\inbordered{(S_i)}).
\eneqn
Here, the last inclusion follows from
$d_{T_i}+d\ge d_{S_i}$ and $p(d_{T_i} +d)+d_{T_i}+d\ge p( d_{S_i})+d_{S_i}$
by perversity.

Thus we obtain
$\Eepb i_{\inbordered{(S_i)}} \Eepb f L\in \dere^{\ge c+p(k)-d}(\inbordered{(S_i)})$
for any $i\in I$,
which implies \eqref{eq:tempiZ0}.

\medskip\noi
(iii) and (iv) follow from (i) and (ii) by adjunction
using Proposition~\ref{pro:Eprightorth} and Proposition~\ref{pro:t-pre1} (i),
respectively.
\end{proof}

\begin{remark}
Concerning (iv) above, the inclusion
\[
\Eeeim f \bl\tEp{\leq c}(\bM)\br \subset \tEp[{p[d]}]{\leq c+d}(\bN)
\]
does not hold in general,
since $\tEp{\leq c}(\bM)$ is not stable by $\indsum$.
For example, let $M=\R\setminus\{0\}$, $N=\R$
and let $f\cl M\to N$ be the inclusion map.
Let $x_n=1/n$ and set
$F_n=\opb{\pi}\cor_{\{x_n\}}\tens \field_{\{t\ge0\}}\in\Mod(\cor_{M\times\bR})$.
Let $K=\quot_M\bl\DSum_{n\ge1}\;F_n\br\in\dere(M)$.
Then $K\in\tEp[1/2]{0}(M)$ but
$\Eeeim f K \simeq\quot_N\bl \indsum\limits_{n\ge1}\eeim {{f_\R}} F_n\br\in\dere(N)$ does not belong to $\tEp[1/2]{\le 0}(N)$. Here
$f_\R\seteq f\times\id_\R\cl M\times\R\to N\times\R$.
Indeed, there is no $Z\in\CS^{<1}( N)$ such that
$\Eopb i_{\inbordered{(N\setminus Z)}} \Eeeim f K\in
\dere^{\le -1/2}(\inbordered{(N\setminus Z)})$, i.e.\
such that $\Eopb i_{\inbordered{(N\setminus Z)}} \Eeeim f K\simeq 0$.
\end{remark}

\subsection{Dual intermediate enhanced perversity}
Let $p$ be a perversity and let $\bM$ be a subanalytic bordered space.
Since the t-structure $\bl \tEprc{\leq c}(\bM), \tEprc{\geq c}(\bM) \br_{c\in\R}$ is not well behaved with respect to duality, we consider also its dual t-structure.

\begin{notation}
For $c\in\R$, set
\begin{align*}
\dEprc{\leq c}(\bM) &\defeq \{ K\in\Erc(\bM) \semicolon \Edual_\bM K \in \tEprc[p^*]{\geq -c}(\bM) \}, \\
\dEprc{\geq c}(\bM) &\defeq \{ K\in\Erc(\bM) \semicolon \Edual_\bM K \in \tEprc[p^*]{\leq -c}(\bM) \}.
\end{align*}
\end{notation}

The following result is a consequence of Proposition~\ref{pro:t-pre1}.

\begin{proposition}
\label{pro:t-dpre1}
$\bl \dEprc{\leq c}(\bM), \dEprc{\geq c}(\bM) \br_{c\in\R}$
is a t-structure on $\Erc(\bM)$.
\end{proposition}

Note that, by the definition, for any $c\in\R$ the duality functor $\Edual_\bM$ interchanges $\tEprc{\leq c}(\bM)$ and $\dEprc[p^*]{\geq -c}(\bM)$, as well as $\tEprc{\geq c}(\bM)$ and $\dEprc[p^*]{\leq -c}(\bM)$.

\begin{lemma}\label{lem:recollD}
Let $\bM$ be a bordered space.
Let $c\in\R$, $Z\in\CS_\bM$, and $K\in\Erc(\bM)$.
Set $U=\unbM\setminus Z$.
Then, considering the morphisms
\[
\xymatrix{
\inbordered{Z} \ar[r]^-i & \bM & \inbordered{U} \ar[l]_-j,
}
\]
one has:
\begin{itemize}
\item[(i)]
$K\in\dEprc{\leq c}(\bM)$ if and only if
\[
\Eopb i K\in\dEprc{\leq c}(\inbordered{Z}) \text{ and }
\Eopb j K\in\dEprc{\leq c}(\inbordered{U}),
\]
\item[(ii)]
$K\in\dEprc{\geq c}(\bM)$ if and only if
\[
\Eepb i K\in\dEprc{\geq c}(\inbordered{Z}) \text{ and }
\Eepb j K\in\dEprc{\geq c}(\inbordered{U}).
\]
\end{itemize}
\end{lemma}

\begin{proof}
The statement follows from Lemma~\ref{lem:recoll}, noticing that
\begin{align*}
\Edual_{\inbordered{Z}}\Eopb i K \simeq \Eepb i \Edual_\bM K, &\quad
\Edual_{\inbordered{U}}\Eopb j K \simeq \Eepb j \Edual_\bM K, \\
\Edual_{\inbordered{Z}}\Eepb i K \simeq \Eopb i \Edual_\bM K, &\quad
\Edual_{\inbordered{U}}\Eepb j K \simeq \Eopb j \Edual_\bM K,
\end{align*}
which is a consequence of Proposition~\ref{prop:dualRc}.
\end{proof}

\begin{lemma}\label{lem:E'E''}
For any $c\in\R$ one has:
\begin{align*}
\dEprc{\leq c}(\bM) &\subset \tEprc{\leq c}(\bM) \subset \dEprc{\leq c+1}(\bM), \\
\tEprc{\geq c}(\bM) &\subset \dEprc{\geq c}(\bM) \subset \tEprc{\geq c-1}(\bM) .
\end{align*}
\end{lemma}

\begin{proof}
Let $K\in\dere(\bM)$.
By Lemma~\ref{lem:Rcstrat}, there exists
a regular filtration $(M_k)_{k\in\Z}$ of $\bM$ such that
both $\Eopbv {i_{\inbordered{(M_k \setminus M_{k-1})}}^{\hs{.8ex}-1}} K$ and
$\Eepbv{i_{\inbordered{(M_k \setminus M_{k-1})}}^{\hs{.8ex}!}} K$ are stably free.
In order to check the inclusions in the statement,
by Lemmas~\ref{lem:recoll} and \ref{lem:recollD}, we may assume that $\bM$ is smooth equidimensional, and that $K$ is stably free.
Then one concludes using Lemma~\ref{lem:t-L}.
\end{proof}

\begin{proposition}
\label{pro:dt-roeimeopb}
Let $f\colon \bM \to \bN$ be a morphism of subanalytic bordered spaces, and $d\in\Z_{\geq 0}$.
Assume that $\dim f^{-1}(y) \leq d$ for any $y\in\unbN$. Then, for any $c\in\R$ one has
\bnum
\item
$\Eopb f \bl\dEprc[{p[d]}]{\leq c}(\bN)\br \subset \dEprc{\leq c}(\bM)$,
\item
$\Eepb f \bl\dEprc[{p[d]}]{\geq c}(\bN)\br \subset \dEprc{\geq c-d}(\bM)$,
\item
$\dere_\Rc(\bN)\cap\Eoim f \bl\dEprc{\geq c}(\bM)\br \subset \dEprc[{p[d]}]{\geq c}(\bN)$,
\item
$\dere_\Rc(\bN)\cap\Eeeim f \bl\dEprc{\leq c}(\bM)\br \subset \dEprc[{p[d]}]{\leq c+d}(\bN)$.
\ee
\end{proposition}
\begin{proof}

(i) Let $K\in\dEprc[{p[d]}]{\leq c}(\bN)$, that is,
$\Edual_\bN K\in\tEprc[{p[d]^*}]{\geq -c}(\bM)$.
Since $p[d]^*(n) = p^*[d](n)+d$,
Proposition~\ref{pro:tpt-roeimeopb} implies
\[
\Edual_\bM \Eopb f K \simeq \Eepb f \Edual_\bN K \in
\tEprc[{p^*}]{\geq -c}(\bM).
\]
Hence
\[
\Eopb f K \in \dEprc[{p}]{\leq c}(\bM).
\]

\smallskip \noi
(ii) is proved similarly.

\smallskip \noi
(iii) and (iv) follows from (i) and (ii) by adjunction.
\end{proof}

\subsection{Enhanced perversity}\label{sse:enhp}
Let $p$ be a perversity and $\bM$ a subanalytic bordered space.

\begin{definition}
For $c\in\R$, consider the strictly full subcategories of $\Erc(\bM)$
given by
\begin{align*}
\Eprc{\leq c}(\bM) &\defeq \tEprc{\leq c}(\bM) \cap \dEprc{\leq c+1/2}(\bM) \\
&= \{K\in\Erc(\bM) \semicolon K\in\tEprc{\leq c}(\bM),\ \Edual_\bM K \in \tEprc[p^*]{\geq -c-1/2}(\bM) \}, \\
\Eprc{\geq c}(\bM) &\seteq \tEprc{\geq c-1/2}(\bM) \cap \dEprc{\geq c}(\bM)\\
&=\{K\in\Erc(\bM) \semicolon \Edual_\bM K \in \Eprc{\leq -c}(\bM) \} \\
&= \{K\in\Erc(\bM) \semicolon K\in\tEprc{\geq c-1/2}(\bM),\ \Edual_\bM K \in \tEprc[p^*]{\leq -c}(\bM) \}.
\end{align*}
\end{definition}

By Lemma~\ref{lem:E'E''} one has
\begin{equation}
\label{eq:E'E''E}
\ba{lll}
\dEprc{\leq c}(\bM) &\subset \Eprc{\leq c}(\bM)
&\subset \tEprc{\leq c}(\bM)\qtq\\
\tEprc{\geq c}(\bM) &\subset \Eprc{\geq c}(\bM)
&\subset\dEprc{\geq c}(\bM).
\ea
\end{equation}

In the rest of this section, we will give a proof of the following result.
\begin{theorem}
\label{thm:mide}
Let $\bM$ be a subanalytic bordered space.
\begin{itemize}
\item[(i)]
$\bl \Eprc{\leq c}(\bM), \Eprc{\geq c}(\bM) \br_{c\in\R}$
is a t-structure on $\Erc(\bM)$.
\item[(ii)]
For any $c\in\R$, the duality functor $\Edual_M$ interchanges $\Eprc{\leq c}(\bM)$ and $\Eprc[p^*]{\geq -c}(\bM)$.
\item[(iii)]
Assume that $\bM=M$ is a subanalytic space.
For any interval $I\subset\R$ such that $I\to\R/\Z$ is injective, the prestack on $M$
\[
U\mapsto \Eprc{I}(U)
\]
is a stack of quasi-abelian categories.
\end{itemize}
\end{theorem}

\begin{proof}[Plan of the proof]
(i) As in the proof of Proposition~\ref{pro:t-pre1}, the statement follows from Propositions~\ref{pro:Eprcfihom} and \ref{pro:Eprcdt} below.

(ii) is clear from the definitions.

(iii) has a proof analogous to that of Proposition~\ref{prop:stack}.
\end{proof}

\begin{lemma}
\label{lem:pt-L}
Assume that $\unbM$ is non empty and smooth of dimension $d$.
For $c\in\R$ and $\shl$ a stable $\enh$-type on $\bM$ as in \eqref{eq:typeT}, one has
\begin{itemize}
\item[(i)]
$\Efield_\shl \in \Eprc{\leq c}(\bM)$ if and only if for any $a\in A$ and $b\in B$
\[
m_a\leq c+p(d), \quad n_b \leq c+p(d)-1/2,
\]
\item[(ii)]
$\Efield_\shl \in \Eprc{\geq c}(\bM)$ if and only if for any $a\in A$ and $b\in B$
\[
m_a\geq c+p(d), \quad n_b \geq c+p(d)-1/2.
\]
\end{itemize}
\end{lemma}

\begin{proof}
The statement follows from Lemma~\ref{lem:t-L}.
\end{proof}

\begin{proposition}
\label{pro:Eprcfihom}
The bifunctors $\fihom$ and $\fhom$ are left exact, i.e., for any $c,c'\in\R$ one has:
\begin{align*}
\fihom(\Eprc{\leq c}(\bM), \Eprc{\geq c'}(\bM)) &\subset \derd^{\geq c'-c}(\bM), \\
\fhom(\Eprc{\leq c}(\bM), \Eprc{\geq c'}(\bM))
&\subset\derd^{\geq c'-c}(\field_\unb\bM).
\end{align*}
In particular,
$\Hom[\dere_\Rc(\bM)]\bl\Eprc{\leq c}(\bM), \Eprc{\geq c'}(\bM)\br=0$
if $c<c'$.
\end{proposition}

\begin{proof}
The second inclusion follows from the first one, since $\fhom \simeq \alpha_\bM\fihom$.
Let us prove the first inclusion.

Let $K\in\Eprc{\leq c}(\bM)$ and $K'\in\Eprc{\geq c'}(\bM)$.
As in the proof of Proposition~\ref{pro:tE-fihom}, reasoning by decreasing induction on $k\in\Z_{\geq -1}$, let us show that
\begin{itemize}
\item[(i)$_k$] there exists $Z_k\in\CS_\bM^{\leq k}$ such that
\[
\rihom(\field_{M\setminus Z_k}, \fihom(K,K'))\in\derd^{\geq c'-c}(\bM).
\]
\end{itemize}
The above statement is obvious for $k\geq d_M$. Assuming that (i)$_k$ holds true for some $k$, let us prove (i)$_{k-1}$.
There exists $Z_{k-1}\in\CS_\bM^{\leq k-1}$ such that
$Z_{k-1}\subset Z_k$, $Z_k\setminus Z_{k-1}$ is smooth of dimension $k$, and
$\Eopb i_{\inbordered{(Z_k\setminus Z_{k-1})}} K$ and
$\Eepb i_{\inbordered{(Z_k\setminus Z_{k-1})}} K$
are stably free.
Consider the distinguished triangle
\eqn
&&\rihom(\field_{Z_k\setminus Z_{k-1}}, \fihom(K,K')) \\
&&\hs{15ex}\to \rihom(\field_{M\setminus Z_{k-1}}, \fihom(K,K')) \\
&&\hs{27ex}\to \rihom(\field_{M\setminus Z_k}, \fihom(K,K')) \tone.
\eneqn
Then (i)$_{k-1}$ will follow if we show that
\[
\rihom(\field_{Z_k\setminus Z_{k-1}}, \fihom(K,K')) \in \derd^{\geq c'-c}(\bM).
\]
This is equivalent to
\[
\epb {i_{\inbordered S}} \fihom(K,K') \in \derd^{\geq c'-c}(\inbordered S)
\]
for any connected component $S$ of $Z_k\setminus Z_{k-1}$.
One has
\begin{align*}
\epb {i_{\inbordered S}} \fihom(K,K')
&\simeq \fihom(\Eopb i_{\inbordered S} K,\Eepb i_{\inbordered S} K').
\end{align*}
By the assumption,
$\Eopb i_{\inbordered S}K \simeq \Efield_\shl$ and $\Eepb i_{\inbordered S} K'
\simeq \Efield_{\shl'}$ for some stable $\enh$-types
\[
\shl = (\varphi_a, m_a, \psi^\pm_b, n_b)_{a\in A,\ b\in B}\qtq
\shl' = (\varphi_{a'}, m_{a'}, \psi^\pm_{b'}, n_{b'})_{a'\in A',\ b'\in B'}.
\]
Then we are reduced to prove
\begin{equation}
\label{eq:fiomLL'}
\fihom(\Efield_\shl,\Efield_{\shl'}) \in \derd^{\geq c'-c}(\inbordered S).
\end{equation}
Recall that
\begin{align*}
\Efield_\shl &= \bl \DSum_{a\in A} \Efield_{\Phi_a}[-m_a] \br
\dsum
\bl \DSum_{b\in B} \Efield_{\Psi_b}[-n_b] \br\in\Eprc{\le c}(\inb S), \\
\Efield_{\shl'} &= \bl \DSum_{a'\in A'} \Efield_{\Phi_{a'}}[-m_{a'}] \br
\dsum
\bl \DSum_{b'\in B'} \Efield_{\Psi_{b'}}[-n_{b'}] \br\in\Eprc{\ge c'}(\inb S).
\end{align*}
By Lemma~\ref{lem:pt-L} and Proposition~\ref{pro:t-homE}, one has
\begin{align*}
\fihom(\Efield_{\Psi_b}[-n_b], {}& \Efield_{\Psi_{b'}}[-n_{b'}]) \\
&\in \fihom(\dere^{\leq c+p(k)-1/2}(\inbordered S),\dere^{\geq c'+p(k)-1/2}(\inbordered S)) \\
&\subset \derd^{\geq c'-c}(\inbordered S).
\end{align*}
Similarly, one has
\begin{align*}
\fihom(\Efield_{\Phi_a}[-m_a], \Efield_{\Psi_{b'}}[-n_{b'}])
&\in \derd^{\geq c'-c-1/2}(\inbordered S), \\
\fihom(\Efield_{\Psi_b}[-n_b], \Efield_{\Phi_{a'}}[-m_{a'}])
&\in\derd^{\geq c'-c+1/2}(\inbordered S), \\
\fihom(\Efield_{\Phi_a}[-m_a], \Efield_{\Phi_{a'}}[-m_{a'}])
&\in\derd^{\geq c'-c}(\inbordered S).
\end{align*}
Hence \eqref{eq:fiomLL'} reduces to show that for any $a\in A$ and $b'\in B'$
\[
H^{m}\fihom(\Efield_{\Phi_a}[-m_a],\Efield_{\Psi_{b'}}[-n_{b'}]) \simeq 0
\]
for any $m\in\Z$ such that $c'-c-1/2\le m<c'-c$.
Since we have
$$H^{m}\fihom(\Efield_{\Phi_a}[-m_a],\Efield_{\Psi_{b'}}[-n_{b'}])
\simeq H^{m+m_a-n_{b'}}\fihom(\Efield_{\Phi_a},\Efield_{\Psi_{b'}}),$$
we may assume that $m+m_a-n_{b'}\ge0$.
Since $m_a\leq c+p(k)$ and $n_{b'}\geq c'+p(k)-1/2$,
one has $m+m_a-n_{b'}\le m-c'+c + 1/2<1/2$.
Then, we have $m+m_a-n_{b'}=0$.

Let $\pi\cl \inb S\times\bR\to \inb S$ and
$\ol\pi\cl\inb S\times \cR\to\inb S$ be the projections.
Then one concludes by noticing that
\begin{align*}
H^{m}&\fihom(\Efield_{\Phi_a}[-m_a],\Efield_{\Psi_{b'}}[-n_{b'}]) \\
&\simeq H^0\fihom(\Efield_{\Phi_a},\Efield_{\Psi_{b'}})
\simeq H^0\fihom(\Qfield_{\Phi_a},\Efield_{\Psi_{b'}}) \\
&\simeq H^0\roim\pi\rihom(\field_{\{t \geq \varphi_a(x)\}},\field_{\{t\gg 0\}} \ctens \field_{\{\psi^-_{b'}(x) \leq t < \psi^+_{b'}(x)\}}) \\
&\simeq H^0\reeim{\ol{\pi}}\rihom(\field_{\{t \geq \varphi_a(x)\}},\indlim[s\to+\infty] \field_{\{\psi^-_{b'}(x)+s \leq t < \psi^+_{b'}(x)+s\}}) \\
&\simeq \eeim{\ol{\pi}}\ihom(\field_{\{t \geq \varphi_a(x)\}},\indlim[s\to+\infty] \field_{\{\psi^-_{b'}(x)+s \leq t < \psi^+_{b'}(x)+s\}}) \\
&\underset{(*)}\simeq \indlim[s\to+\infty] \eeim{\ol{\pi}}\ihom(\field_{\{t \geq \varphi_a(x)\}},\field_{\{\psi^-_{b'}(x)+s \leq t < \psi^+_{b'}(x)+s\}}) \\
&\simeq \indlim[s\to+\infty] \oim\pi\hom(\field_{\{t \geq \varphi_a(x)\}},\field_{\{\psi^-_{b'}(x)+s \leq t < \psi^+_{b'}(x)+s\}}) \simeq 0,
\end{align*}
where $(*)$ holds because $\eeim{\ol{\pi}}$ and $\ihom(\field_{\{t \geq \varphi_a(x)\}},\scbul)$ commute with inductive limits.
\end{proof}

\begin{proposition}
\label{pro:Eprcdt}
For any $c\in\R$ and $K\in \Erc(\bM)$ there are distinguished triangles in $\Erc(\bM)$
\[
K_{\leq c} \to K \to K_{> c} \tone, \quad
K_{< c} \to K \to K_{\geq c}\tone,
\]
with $K_{L}\in \Eprc{L}(\bM)$ for $L$ equal to $\leq c$, $>c$, $<c$ or $\geq c$.
\end{proposition}

\begin{proof}
Since the proofs are similar, we will construct only the first distinguished triangle.

As in the proof of Proposition~\ref{pro:Epdt}, one reduces to the case where
$\bM$ is smooth and connected, and $K$ is stably free.
Then $K$ is a direct sum of objects in $\Eprc{a}(\bM)$ for some $a\in\R$
by Lemma~\ref{lem:pt-L}.
\end{proof}

As a corollary of Propositions~\ref{pro:tpt-roeimeopb} and \ref{pro:dt-roeimeopb}, one has

\begin{proposition}
\label{pro:pt-roeimeopb}
Let $f\colon \bM \to \bN$ be a morphism of subanalytic bordered spaces, and $d\in\Z_{\geq 0}$.
Assume that $\dim {\unb f}{}^{-1}(y)\leq d$
for any $y\in\unbN$. Then, for any $c\in\R$ one has
\bnum
\item
$\Eopb f \bl\Eprc[{p[d]}]{\leq c}(\bN)\br \subset \Eprc{\leq c}(\bM)$,
\item
$\Eepb f \bl\Eprc[{p[d]}]{\geq c}(\bN)\br \subset \Eprc{\geq c-d}(\bM)$,
\item
$\dere_\Rc(\bN)\cap\Eoim f \bl\Eprc{\geq c}(\bM)\br \subset \Eprc[{p[d]}]{\geq c}(\bN)$,
\item
$ \dere_\Rc(\bN)\cap \Eeeim f\bl\Eprc{\leq c}(\bM)\br \subset \Eprc[{p[d]}]{\leq c+d}(\bN)$.
\ee
\end{proposition}

\Proof (i) and (ii) follow from Propositions~\ref{pro:tpt-roeimeopb}
and \ref{pro:dt-roeimeopb},
and (iii) and (iv) follow from them by adjunction. \QED

\begin{proposition}
Let $\bM$ be a subanalytic bordered space.
The embedding
\[
e\colon \BDC_\Rc(\field_\bM) \hookrightarrow \dere_\Rc(\bM)
\]
induced by \eqref{eq:e} is exact, i.e.,\ for any $c\in\R$ one has
\begin{align*}
e\bl \Dprc{\leq c}(\field_\bM)\br &\subset \Eprc{\leq c}(\bM), \\
e\bl \Dprc{\geq c}(\field_\bM)\br &\subset \Eprc{\geq c}(\bM).
\end{align*}
\end{proposition}
\Proof
It follows from the exactness of $e$
with respect to the standard t-structures and
\begin{align*}
\opb{\pi}\cor_S\tens e(F)&\simeq e(\cor_S\tens F), \\
\rihom(\opb{\pi}\cor_S, e(F))&\simeq e(\rihom(\cor_S, F))
\end{align*}
for any $F\in \BDC(\cor_\bM)$ and $S\in\LCS(\bM)$,
by \cite[Corollary 4.7.11]{DK13}.
\QED

\begin{definition}\label{def:Emid}
The \emph{enhanced middle perversity t-structure}
\[
\bl \Emid{\leq c}(\bM), \Emid{\geq c}(\bM) \br_{c\in\R}
\]
is the one associated with the middle perversity $\mathsf m(n) = -n/2$.
It is a self-dual t-structure indexed by $\frac12\Z$
\end{definition}

\begin{example}
Let $\bM=\point$. Note that one has:
\begin{itemize}
\item[(i)]
$\Efield_{\{a\leq t < b\}}\simeq 0$ for $a,b\in\R$ with $a<b$,
\item[(ii)]
$\Efield_{\{t \geq a\}} \simeq \Efield_\bM$ for $a\in\R$,
\item[(iii)]
$\Edual\Efield_\bM \simeq \Efield_\bM$.
\end{itemize}
Hence $\Efield_\bM \in \Emid{0}(\point)$,
and any object of $\Erc(\point)$
is a finite direct sum of shifts of copies of $\Efield_\bM$.
\end{example}

\begin{example}\label{ex:1/x}
Let $\bM=M=\R$ and let
\[
K= \Efield_{\{x>0,\ 0\leq t < 1/x\}\cup\{x=0,\ t\geq 0\}},
\]
so that
\[
\Edual_M K \simeq \Efield_{\{x>0,\ -1/x \leq t <0 \}} [2].
\]
Noticing that
\begin{align*}
\Eepb i_{\{0\}} K &\simeq \Edual_{\{0\}} \Eopb i_{\{0\}} \Edual_M K
\simeq 0, \\
\Eepb i_{\{0\}} \Edual_M K &\simeq \Edual_{\{0\}} \Eopb i_{\{0\}} K
\simeq \Edual_{\{0\}}\Efield_{\{0\}} \simeq \Efield_{\{0\}},
\end{align*}
one has $K\in\tEprc[1/2]{1/2}(\R)$ and $\Edual K\in\tEprc[1/2]{-3/2}(\R)$,
so that $K\in\dEprc[1/2]{3/2}(\R)$.
Hence $K \in \Emid{1}(\R)$.
\end{example}

\begin{example} Let $\{ M_\al\}_\al$ be a subanalytic stratification of $\bM$,
and set $\bM_\al\seteq\inb{(M_\al)}$. Let $K\in\dere_\Rc(\bM)$.
Assume that $\Eopb i_{\bM_\al} K$ and $\Eepb i_{\bM_\al} K$ are stably free.
Recall Notation~\ref{not:PhiaPsib}.
Even if only direct summands containing $\Phi_a$ appear in $\Eopb i_{\bM_\al} K$, direct summands containing $\Psi_b$ can appear in $\Eepb i_{\bM_\al} K$,
as seen in this example.

Let $M = \R^2\times\R_{>0}$ with coordinates $(x,y,z)$,
consider the bordered space $\bM\seteq(M,\R^2\times\cR)$, and set
\begin{align*}
S&\seteq\set{(x,y,z,t)\in M\times\R}{x\geq 0,\ y>0,\ t\geq\dfrac{zx}{x+y}},\\
K &\defeq \Efield_S\in \dere(\bM).
\end{align*}
Set $Z=\{x=y=0\}\subset M$.
Then one has:
\begin{align*}
\Eopb i_{\bZ} K \simeq 0, &\quad
\Eepb i_{\bZ} K \simeq \Efield_{\{0\leq t <z\}}[-1], \\
\Eepb i_{\bZ} \Edual_\bM K \simeq 0, &\quad
\Eopb i_{\bZ} \Edual_\bM K \simeq \Edual_{\bZ} \Eepb i_{\bZ} K
\simeq \Efield_{\{-z\leq t <0\}}[3].
\end{align*}
We deduce that $K\in\tEprc[1/2]{3/2}(\bM)$ and $\Edual_\bM K\in\tEprc[1/2]{-3/2}(\bM)$.
Hence $K\in\dEprc[1/2]{3/2}(\bM)$, so that
\[
K\in \Emid{3/2}(\bM).
\]
\end{example}

\begin{remark}\label{rem:TRc}
Let $M$ be a subanalytic space.
The triangulated category of enhanced sheaves on $M$
(cf.~\cite{Tam08,GS14}) is defined by
\[
\dere^{\mathrm{b}}(\field_M) \defeq \BDC(\field_{M\times\R})/
\opb \pi \BDC(\field_M),
\]
where $\pi\colon M\times\R\to M$ is the projection.
One similarly defines $\dere^{\mathrm{b}}_\pm(\field_M)$,
so that $\dere^{\mathrm{b}}(\field_M) \simeq 
\dere^{\mathrm{b}}_+(\field_M) \dsum \dere^{\mathrm{b}}_-(\field_M)$. 
Note that
\[
\dere^{\mathrm{b}}_\pm(\field_M) \simeq 
\{K\in\dere_\pm(M)\semicolon \LE K\in \BDC(\field_{M\times\bR})\}.
\]
We say that an object $K\in\dere^{\mathrm{b}}_+(\field_M)$ is $\R$-constructible
if so is $\LE K\in \BDC(\field_{M\times\bR})$.
Let $p\colon\Z_{\geq 0}\to \R$ be a perversity. Then, with obvious notations,
\[
\bl \Eprc{\leq c}(\field_M), \Eprc{\geq c}(\field_M) \br_{c\in\R}
\]
satisfies the analogue of Theorem~\ref{thm:mide}.
Moreover, a description analogous to that in Lemma~\ref{lem:Rcstrat} holds, 
replacing ``stably free'' with ``free''.
\end{remark}

\begin{remark}
Let $M$ be a subanalytic space.
It is shown in \cite{KS15} that $\fhom$ induces a functor
\[
\fhom(\Efield_{t\geq 0},\ast) \colon \dere^{\mathrm{b}}_\Rc(M) \to \BDC_\Rc(\field_M).
\]
This is neither left nor right exact with respect to the middle perversity
t-structures.
For example, let
$M=\R^n$ and $K=\Efield_{\{x\not=0,\ t=- 1/|x|\}}$.
Then $K\in\Emid{n/2}(M)$ and
$F\seteq\fhom(\Efield_M,K)\simeq \field_{\{x\not=0\}}$
by \cite[{Corollary 6.6.6.}]{KS15}.
Hence, $\tensor*[^{1/2}]{H}{^{n/2}}(F)\simeq\cor_M$ and
$\tensor*[^{1/2}]{H}{^{1}}(F)\simeq\cor_{\{0\}}$ when $n\ge3$.
Therefore, $\fhom(\Efield_{t\geq 0},\ast)$ is not left exact.
Since $\fhom(\Efield_{t\geq 0},\ast)$ commutes with duality,
$\fhom(\Efield_{t\geq 0},\ast)$ is not right exact either.
\end{remark}

\section{Riemann-Hilbert correspondence}\label{se:RH}

On a complex manifold, the Riemann-Hilbert correspondence embeds the triangulated
category of holonomic $\D$-modules into that of $\R$-constructible enhanced
ind-sheaves. We prove here the exactness of the embedding, when the target category
is endowed with the middle perversity t-structure.

\subsection{Subanalytic ind-sheaves}
For subanalytic sheaves and ind-sheaves we refer to \cite{KS01}
(where subanalytic sheaves are called ind-$\R$-constructible sheaves).

\smallskip

Let $M$ be a subanalytic space.
An ind-sheaf on $M$ is called subanalytic if it is isomorphic to a small filtrant ind-limit of $\R$-constructible sheaves.
Then, being subanalytic is a local property.

Let us denote by $\ind_\suban(\cor_M)$
the category of subanalytic ind-sheaves.
Note that it is
a strictly full subcategory of $\ind(\cor_M)$
stable by kernels, cokernels and extensions.

Let $\operatorname{Op}_{M_\sa}$ be the category of relatively compact subanalytic open subsets of $M$, whose morphisms are inclusions.

\begin{definition}[cf.~\cite{KS96,KS01}]
A \emph{subanalytic sheaf} $F$ is a functor $\operatorname{Op}_{M_\sa}^\op\to\Mod(\cor)$ which satisfies
\bnum
\item
$F(\emptyset) = 0$,
\item
For $U,V\in\operatorname{Op}_{M_\sa}$, the sequence
\[
0 \To F(U\cup V) \To[r_1] F(U) \dsum F(V) \To[r_2] F(U\cap V)
\]
is exact. Here $r_1$ is given by the restriction maps and $r_2$ is given by the restriction $F(U) \to F(U\cap V)$ and the opposite of the restriction $F(V) \to F(U\cap V)$.
\ee
Denote by $\Mod(\cor_{M_\sa})$ the category of subanalytic sheaves.
\end{definition}

The following result is proved in \cite{KS01}.

\begin{proposition}
The category $\ind_\suban(\cor_M)$ of subanalytic ind-sheaves and the category $\Mod(\cor_{M_\sa})$ of subanalytic sheaves are equivalent by the functor associating with $F\in\ind_\suban(\cor_M)$ the subanalytic sheaf
\[
\operatorname{Op}_{M_\sa} \owns U \longmapsto \Hom[\ind(\cor_M)](\cor_U,F).
\]
\end{proposition}

\subsection{Enhanced tempered distributions}
{\em Hereafter, we take the complex number field $\C$
as the base field $\cor$.}
\medskip

Let $M$ be a real analytic manifold.
Denote by $\Db_M$ the sheaf of Schwartz's distributions on $M$.
The subanalytic sheaf of tempered distributions on $M$ is defined by
\begin{align*}
\Dbt_M(U) &\defeq \Im(\Db_M(M)\to\Db_M(U))\\
&\simeq \Db_M(M) / \sect_{M\setminus U}(M; \Db_M)
\end{align*}
for any $U\in\operatorname{Op}_{M_\sa}$.
We still denote by $\Dbt_M$ the corresponding subanalytic ind-sheaf.

Denote by $\PR$ the real projective line, and
let $t\in\R\subset\PR$ be the affine coordinate.
Considering the natural morphism of bordered spaces
\[
j\colon M\times\bR \to M\times\PR ,
\]
one sets
\[
\DbT_M \defeq
\epb j \bl \Dbt_{M\times\PR} \To[\partial_t-1] \Dbt_{M\times\PR} \br \in \derdR{M},
\]
where the above complex sits in degrees $-1$ and $0$.

By the results in \cite[\S8.1]{DK13} one has

\begin{proposition}
\label{pro:DbT}\hfill
\begin{itemize}
\item[(i)] There are isomorphisms in $\derdR{M}$
\begin{align*}
\DbT_M
&\isoto \cihom(\Cfield_{\{t\geq0\}}, \DbT_M) \\
&\isofrom \cihom(\Cfield_{\{t\geq a\}}, \DbT_M) \quad \text{for any $a\geq 0$}.
\end{align*}
\item[(ii)] The complex $\DbT_M$ is concentrated in degree $-1$.
\item[(iii)] There are natural monomorphisms in $\ind(\C_{M\times\bR})$
\[
\xymatrix{
\Cfield_{\{t<*\}} \tens \opb\pi\Dbt_M \ar@{ >->}[r] &
H^{-1}\DbT_M \ar@{ >->}[r] &
\opb\pi\Db_M.
}
\]
\end{itemize}
\end{proposition}

The enhanced ind-sheaf of tempered distributions is defined by
\[
\DbE_M \defeq \quot_M(\DbT_M) \in \dere(M).
\]

Part (iii) in the following proposition is new.

\begin{proposition}
\label{pro:DbE}\hfill
\begin{itemize}
\item[(i)] $\DbE_M$ is stable, i.e.\ $\ECfield_M\ctens\DbE_M\simeq\DbE_M$.
\item[(ii)] $\RE\DbE_M \simeq \DbT_M$. In particular, it is concentrated in degree $-1$.
\item[(iii)] $\DbE_M \in \dere^0(M)$. In other words, the complex $\LE\DbE_M$ is concentrated in degree $0$.
\end{itemize}
\end{proposition}

\begin{proof}
(i) follows from Proposition~\ref{pro:DbT} (i).

(ii) By Proposition~\ref{pro:DbT} (i), one has $\RE\DbE_M \simeq \DbT_M$.
This is concentrated in degree $-1$ by Proposition~\ref{pro:DbT} (ii),

(iii) By (ii), $\RE\DbE_M \simeq \DbT_M$ is concentrated in degree $-1$.
Hence Lemma~\ref{lem:tC-ctenscihom} implies
\[
\LE\DbE_M \simeq \Cfield_{\{t\geq 0\}}\ctens \DbT_M \in\derdR[{[-1,0]}]{M},
\]
and we are reduced to prove that $H^{-1}\LE\DbE_M \simeq 0$.

By \cite[Proposition 4.3.10]{DK13}, there is a distinguished triangle
\[
\opb\pi_M\roimv{\pi_{M\sep!!}}\DbT_M \to \LE\DbE_M \to \DbT_M \tone.
\]
By Proposition~\ref{pro:DbT} (iii),
\[
H^{-1}\roimv{\pi_{M\sep!!}}\DbT_M \simeq \pi_{M\sep!!}H^{-1}\DbT_M \subset \pi_{M\sep!!}\opb\pi\Db_M = 0.
\]
Thus, the above distinguished triangle induces the exact sequence
\[
0 \to H^{-1}\LE\DbE_M \to H^{-1}\DbT_M \to[\gamma] \opb\pi\derr^1\pi_{M\sep!!}H^{-1}\DbT_M .
\]
To conclude, we have to show that $\gamma$ is a monomorphism.

By Proposition~\ref{pro:DbT} (iii), there is a commutative diagram
\[
\xymatrix{
 H^{-1}\DbT_M \ar[r]^-\gamma \ar@{ >->}[d] & \opb\pi_M\derr^1\pi_{M\sep!!}H^{-1}\DbT_M \ar[d] \\
\opb\pi_M\Db_M \ar[r]^-\sim & \opb\pi_M\derr^1\pi_{M\sep!!}\opb\pi_M\Db_M.
}
\]
Hence $\gamma$ is a monomorphism.
\end{proof}

\subsection{$\D$-modules}

Let $X$ be a complex manifold.
We denote by $d_X^\C$ its complex dimension.
Denote by $\O_X$ and $\D_X$ the sheaves of algebras of holomorphic functions and of
differential operators, respectively. Denote by $\Omega_X$ the sheaf of
differential forms of top degree.

Denote by $\Mod(\D_X)$ the category of left $\D_X$-modules, and by $\BDC(\D_X)$ its bounded derived category.
For $f\colon X\to Y$ a morphism of complex manifolds, denote by $\dtens$,
$\dopb f$, $\doim f$ the operations for $\D$-modules.

Consider the dual of $\shm\in\BDC(\D_X)$ given by
\[
\ddual\shm = \rhom[\D_X](\shm,\D_X\tens[\O_X]\Omega_X^{\otimes-1})[d_X^\C].
\]

A $\D_X$-module $\shm$ is called \emph{quasi-good} if, for any relatively compact open subset $U\subset X$, $\shm\vert_U$
is isomorphic (as an $\O_X\vert_U$-module) to a filtrant inductive limit of
coherent $\O_X\vert_U$-modules.
A $\D_X$-module $\shm$ is called \emph{good} if it is quasi-good and coherent.

Recall that to a coherent $\D_X$-module $\shm$ one associates its characteristic variety $\chv(\shm)$, a closed conic involutive subset of the cotangent bundle $T^*X$. If $\chv(\shm)$ is Lagrangian, $\shm$ is called holonomic. For the notion of regular holonomic $\D_X$-module, refer e.g.\ to \cite[\S5.2]{Kas03}.

Denote by $\BDC_{\hol}(\D_X)$ and $\BDC_{\reghol}(\D_X)$ the full
subcategories of $\BDC(\D_X)$ whose objects have holonomic and regular
holonomic cohomologies, respectively. These are triangulated categories.

Let $f\colon X \to Y$ be a morphism of complex manifolds.
For $x_0\in X$ consider
\begin{align*}
&\operatorname{rank}_{x_0}^\C(f) \defeq \operatorname{rank}^\C(T_{x_0}X \to[df(x_0)] T_{f(x_0)}Y)\qtq \\
&\operatorname{flat-dim}_{\D_{X,x_0}}(\D_{X\to Y,x_0}),
\end{align*}
the complex dimension of the image of $df(x_0)$, and the flat dimension of $\D_{X\to Y,x_0}$ as a left $\D_{X,x_0}$-module, respectively.

\begin{proposition}
\label{pro:hdtrans}
Let $f\colon X \to Y$ be a morphism of complex manifolds.
For $x_0\in X$ one has
\[
\operatorname{flat-dim}_{\D_{X,x_0}}(\D_{X\to Y,x_0}) \leq d_X^\C - \operatorname{rank}_{x_0}^\C(f).
\]
\end{proposition}

\begin{proof}
Set $n=d_X^\C$, $m=d_Y^\C$, $d= \operatorname{rank}_{x_0}^\C(f)$, and $y_0=f(x_0)$.

Choose a system of local coordinates $y=(y_1,\dots,y_m)$ of $Y$
on a neighborhood of $y_0$, such that $\partial_{y_1},\dots,\partial_{y_d}$ generate $df(x_0)(T_{x_0}X) \subset T_{f(x_0)}Y$.
Set $x_k = y_k \circ f$ for $k\leq d$, and complete them to a system of local coordinates $x=(x_1,\dots,x_n)$ of $X$ on a neighborhood of $x_0$.

Consider the subring
\[
\shr\defeq \O_{X,\,x_0}[\partial_{x_1},\dots,\partial_{x_d}]\subset\D\defeq\D_{X,\,x_0}.
\]

Then $\D_{X\to Y,\,x_0} \simeq \O_{X,\,x_0}\tens[\O_{Y,\,y_0}]\D_{Y,\,y_0}$ is a free $\shr$-module.
In fact, one has
\[
\D_{X\to Y,\,x_0} \simeq \DSum_{\beta\in \{0\}^d\times\Z^{m-d}_{\geq 0}\subset\Z^m_{\geq 0}}\shr \partial_y^\beta.
\]

The statement follows by Lemma~\ref{lem:flat} below.
\end{proof}

\begin{lemma}
\label{lem:flat}
Use notations as in the proof above. Let $\shm$ be a left $\D$-module.
If $\shm$ is flat as a left $\shr$-module, then
\[
\operatorname{flat-dim}_{\D}(\shm) \leq n - d.
\]
\end{lemma}

\begin{proof}
Set $\O\defeq\O_{X,\,x_0}$ and $\D'\defeq \O[\partial_{x_{d+1}},\dots,\partial_{x_n}]$, so that
$\D \simeq \D'\tens[\O]\shr$.
Set $K\defeq\C\partial_{x_{d+1}}\dsum\cdots\dsum\C\partial_{x_n}$. Then the Spencer resolution of $\shm$, considered as a $\D'$-module, is
\[
0\to (\D'\tens\extp^{n-d} K)\tens[\O]\shm \to \cdots \to
\D'\tens[\O]\shm \to \shm \to 0.
\]
Since $\D'\tens[\O]\shr \simeq \D$, the above resolution reads as
\[
0\to (\D\tens\extp^{n-d} K)\tens[\shr]\shm \to \cdots \to
\D\tens[\shr]\shm \to \shm \to 0.
\]
Since $\shm$ is a flat left $\shr$-module, this is a flat resolution of $\shm$ as a left $\D$-module.
\end{proof}

For a category $\shc$,
let $\operatorname{Pro}(\shc)$ be the category of pro-objects in $\catc$,
and $\prolim$ the projective limit in $\operatorname{Pro}(\shc)$.

\begin{lemma}\label{lem:Mflatprolim}
Let $\shm$ be a quasi-good $\D_X$-module, flat over $\D_X$.
Let $\{\shm_i\}_{i\in I}$ be a filtrant inductive system of coherent $\D_X$-modules
such that $\shm\simeq \colim_i\shm_i$. Then, for any $x\in X$ and any $k\neq 0$ one has
\[
\prolim[i] \ext[\D_X]k(\shm_i,\D_X)_x \simeq 0
\quad\text{in }\operatorname{Pro}(\Mod(\D_{X,x}^\op)).
\]
\end{lemma}

\begin{proof}
There exists a filtrant inductive system $\{L_j\}_{j\in J}$ of free $\D_{X,x}$-modules of finite rank such that
$\shm_{x} \simeq \colim_jL_j$ (see \cite{L69}). It implies that
$\indlim[i]\shm_{i,x}\simeq \indlim[j]L_j$ in $\Ind(\Mod(\D_{X,x}))$.
Hence, for any $i\in I$ there exist $j\in J$,
a morphism $u\colon i\to i'$ in $I$ and a commutative diagram
\[
\xymatrix@R=1ex@C=1em{
\shm_{i,x} \ar[rr] \ar[dr] && \shm_{i',x}\,. \\
 & L_j \ar[ur].
}
\]
It follows that the morphism induced by $u$,
\[
\ext[\D_{X,x}]k(\shm_{i',x},\D_{X,x}) \to \ext[\D_{X,x}]k(\shm_{i,x},\D_{X,x}),
\]
is the zero morphism.
\end{proof}

For a hypersurface $Y$ of $X$,
denote by $\sho_X(*Y)$ the sheaf of meromorphic functions on $X$ with poles in $Y$. We set
$\D_X(*Y)=\sho_X(*Y)\tens_{\sho_X}\D_X\simeq \D_X\tens_{\sho_X}\sho_X(*Y)$.
It is a sheaf of $\C$-algebras on $X$.
For a $\D_X$-module $\shm$, we set $\shm(*Y)\seteq\shd_X(*Y)\tens[\shd_X]\shm$.

\begin{lemma}\label{lem:MflatprolimY}
Let $Y\subset X$ be a closed complex analytic hypersurface, and let $\shm$ be a quasi-good $\D_X$-module.
Assume that $\shm|_{X\setminus Y}$ is flat over $\D_{X\setminus Y}$.
Let $\{\shm_i\}_{i\in I}$ be a filtrant inductive system of coherent $\D_X$-modules
such that $\shm(*Y)\simeq \colim_i\shm_i$. Then, for any $V\ssubset X$ one has
\bnum
\item
$\prolim[i] \ext[\D_X]k(\shm_i,\D_X(*Y))|_V \simeq 0$
in $\operatorname{Pro}\bl\Mod(\D_V^\op)\br$ for any $k\neq 0$,
\item
$\prolim[i] \rhom[\D_X](\shm_i,\D_X(*Y))|_V$\newline
\hs{10ex}$ \simeq
\prolim[i] \hom[\D_X](\shm_i,\D_X(*Y))|_V$
in $\operatorname{Pro}\bl\BDC(\D_V^\op)\br$.
\ee
\end{lemma}

\begin{proof}
(i)\ For $i\in I$, denote by $I^i$ the category whose objects are morphism $i\to i'$ in $I$ with source $i$, and whose morphisms are commutative diagrams in $I$
\[
\xymatrix@R=1ex@C=1em{
& i \ar[dl] \ar[dr] \\
i' \ar[rr] && i''\,.
}
\]
It is enough to show that for any $i\in I$ there exists $(u_0\colon i\to i_0)\in I^i$ such that the induced morphism
\[
u_0' \colon \ext[\D_X]k(\shm_{i_0},\D_X(*Y))|_V \to \ext[\D_X]k(\shm_i,\D_X(*Y))|_V
\]
is the zero morphism.
For $(u\colon i\to i')\in I^i$, set
\[
\shn_u = \Im \bl \ext[\D_X]k(\shm_{i'},\D_X) \to \ext[\D_X]k(\shm_i,\D_X) \br.
\]
It is a coherent $\D_X^\op$-module.
Since $I^i$ is filtrant by \cite[Corollary 3.2.3]{KS06},
$\{\supp(\shn_u)\}_{u\in I^i}$ is a decreasing family of closed complex analytic subsets. Hence it is locally stationary. Thus, there exists $(u_0\colon i\to i_0)\in I^i$ such that
\[
\supp(\shn_{u_0}|_V) = \Inter_{u\in I^i} \supp(\shn_u|_V).
\]
By Lemma~\ref{lem:Mflatprolim}, one has
\[
\Inter_{u\in I^i} \supp(\shn_u|_V) \subset Y.
\]
Hence, $\supp(\shn_{u_0}|_V)\subset Y$ and one has
\[
0 \simeq (\shn_{u_0} \tens[\D_X] \D_X(*Y))|_V \simeq
\Im(u_0').
\]
Hence we obtain (i).

\smallskip
\noi
(ii) follows from (i).
\end{proof}

\begin{proposition}\label{pro:MYflat}
Let $Y\subset X$ be a closed complex analytic hypersurface, and let $\shm$ be a quasi-good $\D_X$-module.
Assume that $\shm|_{X\setminus Y}$ is flat over $\D_{X\setminus Y}$.
Then $\shm(*Y)$ is a flat $\D_X$-module.
\end{proposition}

\begin{proof}
The question being local,
we can write $\shm(*Y)\simeq \colim_i\shm_i$ with $\{\shm_i\}_{i\in I}$ a filtrant inductive system of coherent $\D_X$-modules. Set
\[
\shm_i^* \seteq \hom[\D_X](\shm_i,\D_X).
\]
Then $\hom[\D_X](\shm_i,\D_X(*Y)) \simeq \shm_i^*(*Y)$.
By Lemma~\ref{lem:MflatprolimY},
one has
\[
\prolim[i] \rhom[\D_X](\shm_i,\D_X(*Y))
\simeq \prolim[i]\shm_i^*(*Y)\quad
\text{in $\operatorname{Pro}(\BDC(\D_X^\op))$,}
\]
by shrinking $X$ if necessary.
Let $\shp\in\Mod(\D_X^\op)$. We have to show that, for $k< 0$,
\begin{equation}
\label{eq:tempHkP}
H^k\bl\shp\ltens[\D_X]\shm(*Y)\br \simeq 0.
\end{equation}
One has
\begin{align*}
H^k\bl\shp \ltens[\D_X]\shm(*Y)\br
&\simeq H^k\bl\shp(*Y)\ltens[\D_X]\shm(*Y)\br \\
&\simeq \colim_{i} H^k\bl\shp(*Y)\ltens[\D_X]\shm_i\br.
\end{align*}
Moreover,
\begin{align*}
\indlim[i] \shp(*Y)&{}\ltens[\D_X]\shm_i \\
&\simeq \indlim[i] \rhom[\D_X^\op]\bl\rhom[\D_X](\shm_i,\D_X),\;\shp(*Y)\br \\
&\simeq \indlim[i] \rhom[\D_X(*Y)^\op]\bl\rhom[\D_X]\bl\shm_i,\D_X(*Y)\br,\;\shp(*Y)\br \\
&\simeq \indlim[i] \rhom[\D_X(*Y)^\op]\bl\shm_i^*(*Y),\,\shp(*Y)\br.
\end{align*}
Hence we obtain
$$H^k\bl\shp \ltens[\D_X]\shm(*Y)\br\simeq
\colim_i H^k \rhom[\D_X^\op(*Y)]\bl\shm_i^*(*Y),\shp(*Y)\br,$$
which vanishes for $k<0$.
\end{proof}

Let us denote by $\dere(\D_X^\op)$ the category of
enhanced ind-sheaves on $X$ with $\D_X^\op$-action
(see \cite[{\S\;4.10}]{DK13} where $\dere(\D_X^\op)$
is denoted by $\dere^{\mathrm b}(\mathrm{I}\,\D_X^\op)$).

Consider the forgetful functor
\[
\operatorname{for} \colon \dere(\D_X^\op) \to \dere(X).
\]

\begin{lemma}\label{lem:EMXY0}
Let $c\in\R$, 
$X$ a complex manifold, $Y\subset X$ a complex analytic subset,
$K\in \dere(\D_X^\op)$, and $\shm$ a quasi-good $\D_X$-module.
Set $U=X\setminus Y$.
Assume
\begin{itemize}
\item[(a)]
$K\simeq \rihom(\opb\pi\field_U, K)$,
\item[(b)]
$\operatorname{for}(K) \in \dere^{\geq c}(X)$,
\item[(c)]
$\shm|_U$ is flat over $\D_U$.
\end{itemize}
Then,
\[
K \ltens[\D_X] \shm \in \dere^{\geq c}(X).
\]
\end{lemma}

\begin{proof}
(i) Let $\varphi\colon X' \to X$ be a projective morphism such that $Y' \defeq \varphi^{-1}(Y)$ is a hypersurface, and $\varphi$ induces an isomorphism $
U'\seteq\varphi^{-1}(U) \isoto U$.
Set
\begin{align*}
K' &\defeq \rihom\bl
\pi^{-1}\C_{U'},\Eopb\varphi K \ltens[\opb\varphi\D_X] \D_{X\from X'}\br
\in \dere(\D_{X'}^\op) , \\
\shm' &\defeq (\dopb\varphi\shm)(*Y').
\end{align*}
Then we have $\operatorname{for}(K')\in\dere^{\geq c}(X')$.
Note that $\shm'$ is concentrated in degree zero. Moreover, by Proposition~\ref{pro:MYflat}, $\shm'$ is a flat $\D_{X'}$-module.
Since
\[
K \ltens[\D_X] \shm \simeq \Eoim\varphi(K' \ltens[\D_{X'}] \shm'),
\]
and since $\Eoim\varphi$ is left exact, we reduce to the case where $\shm$ is flat over $\D_X$.

(ii)
 Let $\shm$ be a quasi-good flat $\D_X$-module.
Let $\{\shm_i\}_{i\in I}$ be a filtrant inductive system of coherent $\D_X$-modules
such that $\shm\simeq \colim\limits_i\shm_i$. Set
\[
\shm_i^* \seteq \hom[\D_X](\shm_i,\D_X).
\]
Then Lemma~\ref{lem:MflatprolimY} implies that
\eqn
&&
\prolim[i] \rhom[\D_X](\shm_i,\D_X)
\simeq \prolim[i]\shm_i^*\quad
\text{in $\operatorname{Pro}(\BDC(\D_X^\op))$,}
\eneqn
by shrinking $X$ if necessary.
Hence one has
\begin{align*}
H^k(K \ltens[\D_X] \shm)
&\simeq \indlim[i] H^k(K \ltens[\D_X] \shm_i) \\
&\simeq \indlim[i] H^k\rhom[\D_X^\op](\rhom[\D_X](\shm_i,\D_X),K) \\
&\simeq \indlim[i] H^k\rhom[\D_X^\op](\shm_i^*,K) \simeq 0
\end{align*}
for $k<c$.
\end{proof}

\begin{proposition}\label{pro:EMXY}
Let $\ell\in\Z_{\geq 0}$, $c\in\R$,
$X$ a complex manifold, $Y\subset X$ a complex analytic subset,
$K\in \dere(\D_X^\op)$, and $\shm$ a quasi-good $\D_X$-module.
Set $U=X\setminus Y$.
Assume
\bna
\item
$\operatorname{for}(K) \in \dere^{\geq c}(X)$,
\item
$\operatorname{flat-dim}_{\D_{X,x}}(\shm_x) \leq \ell$ for any $x\in U$.
\ee
Then,
\[
\rihom(\opb \pi \cor_U,K) \ltens[\D_X] \shm \in \dere^{\geq c-\ell}(X).
\]
\end{proposition}

\begin{proof}
Replacing $K$ with $\rihom(\opb \pi \cor_U,K)$,
we may assume that
$K\simeq \rihom(\opb \pi \cor_U,K)$ from the beginning.
We proceed by induction on $\ell$.
The case $\ell=0$ follows from Lemma~\ref{lem:EMXY0}.
Let $\ell>0$.
Then, there is locally a short exact sequence
\[
0 \to \shn \to \shl \to \shm \to 0,
\]
with a free $\D_X$-module $\shl$.
Hence $\shn$ is a quasi-good $\D_X$-module such that $\operatorname{flat-dim}_{\D_{X,x}}(\shn_x) \leq \ell-1$ for any $x\in U$.
One has $K \ltens[\D_X] \shn \in \dere^{\geq c-\ell+1}(X)$ by the
induction hypothesis.
Moreover, $K \ltens[\D_X] \shl \in \dere^{\geq c}(X)$ since $\shl$ is free.
One concludes by considering the distinguished triangle
\[
K \ltens[\D_X] \shl \to
K \ltens[\D_X] \shm \to
K \ltens[\D_X] \shn[1] \tone.
\]
\end{proof}

\subsection{Enhanced tempered holomorphic functions}

Let $X$ be a complex manifold.

\begin{proposition}\label{pro:t-OE}
One has $\OEn_X \in \tEp[1/2]{\geq d_X^\C}(X)$.
\end{proposition}

\begin{proof}
By Lemma~\ref{lem:crg}, it is enough to show
that for any $k\in\Z_{\ge0}$ and any $Z\in \CS_{X_\R}^{\leq k}$ there exists
an open subanalytic subset $Z_0$ of $Z$ such that $\dim(Z\setminus Z_0) < k$ and
\begin{equation}
\label{eq:OEtprimetemp}
\enh i_{\inbordered{(Z_0)}}^{\ms{15mu}!}\OEn_X\in \dere^{\geq d^\C_X-k/2}(\inbordered{(Z_0)}).
\end{equation}
Since the question is local on $X$, we may assume
from the beginning that $Z$ is compact.
Let $Z_0$,
$W_0\subset N$, $\,\ell=d_N$ and $g\cl N\to M$
be as obtained by Lemma~\ref{lem:ZMred} below, for $M=X_\R$ the real analytic manifold underlying $X$.
There exists a complexification $Y$ of $N$
such that $g\cl N\to X$ extends to a holomorphic map $f\cl Y\to X$.

Then, $d_Y^\C = \ell$ and there is a commutative diagram

\[
\xymatrix@C=5em{
\llap{$\inbordered{(W_0)}\seteq$}(W_0,N)\ar[d]_{g_0}\ar[r]_-{i_ {\inbordered{(W_0)}}}
\ar@/^1.5pc/[rr]^j& N
\ar[r]_{i_N} & Y \ar[d]^f \\
\llap{$\inbordered{(Z_0)}\seteq$}(Z_0,Z)\ar[rr]^{i_{\inbordered{(Z_0)}}}&& X.
}
\]

Note that for any $w\in W_0$, setting $x=f(w)\in Z_0$, one has
\begin{equation}
\label{eq:rkf}
\operatorname{rank}_{w}^\C(f) = \dim^\C(T_{x} Z + \sqrt{-1} T_{x}Z)
\ge (\dim T_{x} Z)/2=k/2.
\end{equation}Set
$$ V\seteq\set{y\in Y}{\operatorname{rank}_{y}^\C(f)\ge k/2}.$$
Then $V$ is an open subset of $Y$ such that
$Y\setminus V$ is a closed complex analytic subset. Moreover $W_0\subset V$.
Hence Proposition~\ref{pro:hdtrans} implies
\eq
&&\hs{3ex}\on{flat-dim}_{\D_{Y,y}^\op}(\D_{X\from Y})\le d^\C_Y-k/2=\ell-k/2
\quad\text{for any $y\in V$.}\label{eq:flatdim}
\eneq

By Proposition~\ref{pro:EopbVan}, in order to see \eqref{eq:OEtprimetemp}
it is enough to show
\eq\label{eq:g0O}
\Eopb g_0\enh\, i_{\inbordered{(Z_0)}}^{\ms{15mu}!}\OEn_X
\in\dere^{\geq d^\C_X-k/2}(\inbordered{(W_0)}).
\eneq

Since $W_0\to Z_0$ is smooth, one has
\eqn
\Eopb g_0\enh\, i_{\inbordered{(Z_0)}}^{\ms{15mu}!}\OEn_X
&\simeq&\ori_{W_0/Z_0}\tens\Eepb g_0\enh\, i_{\inbordered{(Z_0)}}^{\ms{15mu}!}\OEn_X[d_{Z_0}-d_N]\\
&\simeq&\ori_{W_0/Z_0}\tens\enh\, i_{\inbordered{(W_0)}}^{\ms{15mu}!}
\Eepb i_N\Eepb f\OEn_X[k-\ell]\\
&\simeq&\ori_{W_0/Z_0}\tens\Eepb j\Eoim{i_N}\Eepb i_N\Eepb f\OEn_X
[k-\ell],
\eneqn
where $\ori_{W_0/Z_0}\seteq H^{k-\ell}(\epb{g_0}\C_{Z_0})$
is the relative orientation sheaf.

By \cite[Theorem 9.1.2]{DK13}, one has
\[
\Eepb f \OEn_X \simeq \D_{X\from Y} \ltens[\D_Y] \OEn_Y [d^\C_Y-d^\C_X].
\]
Moreover, denoting by $\ori_{N/Y} \simeq \epb{i_N}\C_Y[\ell]$
the relative orientation sheaf, one has
\begin{align*}
\epb{i_N}\OEn_Y
&\simeq \ori_{N/Y}\tens\DbE_N[-d_N].
\end{align*}
Thus, we obtain
\eqn
&&\ori_{W_0/Z_0}\tens\Eopb g_0\enh\, i_{\inbordered{(Z_0)}}^{\ms{15mu}!}\OEn_X\\
&&\hs{3ex}\simeq\Eepb j\Eoim{i_N}\Eepb i_N
\bl\D_{X\from Y} \ltens[\D_Y] \OEn_Y\br[k-d^\C_X]\\
&&\hs{3ex}\simeq \Eepb j\bl
\D_{X\from Y} \ltens[\D_Y]
\Eoim{i_N}(\ori_{N/Y}\tens\DbE_N)\br[k-d^\C_X-d_N]\\
&&\hs{3ex}\simeq \Eepb j\bl \D_{X\from Y} \ltens[\D_Y]
\rihom(\opb\pi\Cfield_{ V},\Eoim{i_N}(\ori_{N/Y}\tens\DbE_N))\br\\
&&\hs{52ex}[k-d^\C_X-\ell].
\eneqn
By Proposition~\ref{pro:DbE}, one has
$$\Eoim{i_N}(\ori_{N/Y}\tens\DbE_N)\in \dere^{\geq 0}(Y).$$
Hence Proposition \ref{pro:EMXY} and \eqref{eq:flatdim} implies that
\begin{align*}
\D_{X\from Y} \ltens[\D_Y]
\rihom\bl\opb\pi\Cfield_{ V},\Eoim{i_N}(\ori_{N/Y}\tens\DbE_N)\br
& \in \dere^{\geq k/2-\ell}(Y).
\end{align*}
Finally, we obtain \eqref{eq:g0O}.
\end{proof}

\begin{corollary}
One has $\Ot_X \in \Dp[1/2]{\geq d_X^\C}(X)$.
\end{corollary}

\begin{proof}
Since $\Ot_X \simeq \fihom(\ECfield_X,\OEn_X)$, the statement follows from Proposition~\ref{pro:t-OE} and Lemma~\ref{lem:t-fihomEE'}.
\end{proof}

Here is the lemma which is used in the course of the proof of
Proposition~\ref{pro:t-OE}.

\begin{lemma}
\label{lem:ZMred}
Let $M$ be a real analytic manifold, and let $Z\in \CS_M^{\leq k}$ for $k\in\Z_{\ge0}$.
Assume that $Z$ is compact. Then there exist
\bnum
\item
an open subset $Z_0$ of $Z$ which is a real analytic submanifold of dimension $k$,
\item
a real analytic manifold $N$ of dimension $\ell\ge k$,
\item
a real analytic proper map $g\colon N\to M$,
\item
an open subanalytic subset $W_0$ of $N$
\ee
such that one has
\bna
\item
$\dim(Z\setminus Z_0) < k$,
\item
$g(N)=Z$, $g(W_0)=Z_0 $ and
$g$ induces a smooth morphism $W_0\to Z_0$ of real analytic manifolds.
\ee
\end{lemma}

\Proof
It follows immediately from the existence of a
real analytic manifold $N$ and a
proper real analytic map $g\cl N\to M$ such that $g(N)=Z$.
Note that we may assume that $N$ is equidimensional,
by multiplying each connected component of $N$ with a sphere if necessary.
\QED

\subsection{Riemann-Hilbert correspondence}

Let $X$ be a complex manifold.
The enhanced de Rham and solution functors
\begin{align*}
\drE &\colon \BDC(\D_X) \to \dere(X), \\
\solE &\colon \BDC(\D_X)^\op \to \dere(X),
\end{align*}
are defined by
\begin{align*}
\drE(\shm) &\defeq \OvE_X \ltens[\D_X] \shm, \\
\solE(\shm) &= \rhom[\D_X](\shm,\OEn_X),
\end{align*}
where $\OvE_X \defeq \Omega_X \ltens[\O_X] \OEn_X$.

The Riemann-Hilbert correspondence of \cite[Theorem 9.5.3]{DK13} implies that these functors induce fully faithful functors
\eq\ba{lcr}
\drE &\colon& \BDC_\hol(\D_X) \to \Erc(X), \\
\solE &\colon& \BDC_\hol(\D_X)^\op \to \Erc(X).
\ea
\eneq

\begin{theorem}
The functors $\drE$ and $\solE{}[d_X^\C]$
are exact. That is, for any $c\in\R$ one has
\begin{align*}
\drE\bl \derd_\hol^{\leq c}(\D_X) \br \subset \Emid{\leq c}(X),
&\qquad \solE\bl \derd_\hol^{\leq c}(\D_X) \br \subset \Emid{\geq d_X^\C-c}(X), \\
\drE\bl \derd_\hol^{\geq c}(\D_X) \br \subset \Emid{\geq c}(X),
&\qquad \solE\bl \derd_\hol^{\geq c}(\D_X) \br \subset \Emid{\leq d_X^\C-c}(X).
\end{align*}
In particular, there are commutative diagrams of embeddings
\[
\xymatrix@C3em{
\Mod_\hol(\D_X) \ar@{ >->}[r]^{\drE} & \Emid0(X) \\
\Mod_\reghol(\D_X) \ar@{ >->}[u] \ar@{ >->}[r]^{\dr} & \Dmid0(\Cfield_X), \ar@{ >->}[u]
}
\hs{1ex}
\xymatrix@C3em{
\Mod_\hol(\D_X)^\op \ar@{ >->}[r]^{\solE} & \Emid{d_X^\C}(X) \\
\Mod_\reghol(\D_X)^\op \ar@{ >->}[u] \ar@{ >->}[r]^{\sol} & \Dmid{d_X^\C}(\Cfield_X). \ar@{ >->}[u]
}
\]
\end{theorem}

\begin{proof}
It is enough to show that for any $\shm\in\Mod_\hol(\D_X)$ one has
\[
\drE(\shm) \in \Emid{0}(X), \quad
\solE(\shm) \in \Emid{d_X^\C}(X).
\]\
\noi
(i)\ By the definition,
\begin{align*}
\solE(\shm) &\simeq \rhom[\D_X](\shm,\OEn_X).
\end{align*}
By Proposition~\ref{pro:t-OE},
\[
\OEn_X \in \tEp[1/2]{\geq d_X^\C}(X).
\]
Hence
\begin{align*}
\solE(\shm) &\in \tEprc[1/2]{\geq d_X^\C}(X) \subset \Eprc[1/2]{\geq d_X^\C}(X),
\end{align*}
where the inclusions follow from \eqref{eq:E'E''E}.
Then
\[
\drE(\shm)\simeq\solE(\ddual\shm)[d^\C_X]
\in\Eprc[1/2]{\geq 0}(X).
\]

\medskip\noi
(ii)\
Note that $\ddual\shm \in \Mod_\hol(\D_X)$. Moreover, by \cite[Theorem 9.4.8]{DK13},
\begin{align*}
\Edual_X\drE(\shm) &\simeq \drE(\ddual\shm).
\end{align*}
We thus get from (i)
\begin{align*}
\drE(\shm) &\in \Eprc[1/2]{\leq 0}(X), \quad\text{and hence}\\
\solE(\shm)&\simeq \drE(\ddual\shm)[-d^\C_X]\in\Eprc[1/2]{\leq d_X^\C}(X).
\end{align*}
\end{proof}

\end{document}